\documentclass[reqno,11pt]{amsart}

\usepackage{graphicx}
\usepackage{latexsym}
\usepackage{amstext}
\usepackage {amsmath}
\usepackage {amsfonts}
\usepackage {amssymb}
\usepackage {amsthm}
\usepackage {bbm}
\usepackage {mhequ}
\usepackage {times}

\usepackage {microtype}
\usepackage{enumerate}
\usepackage{array}
\usepackage{comment}
\usepackage{xspace}
\usepackage[bookmarks=true]{hyperref}
\usepackage{enumerate}
\ifx\cs\documentclass \footheight 12pt \fi \footskip 30pt
\DeclareMathOperator{\Sym}{Sym}

\DeclareMathOperator{\tr}{Tr}

\newcommand{\eps}{\varepsilon}
\newcommand{\R}{\ensuremath{\mathbb{R}}}
\newcommand{\Rn}{\ensuremath{{\mathbb{R}^n}}}
\newcommand{\N}{\ensuremath{\mathbb{N}}}
\newcommand{\Z}{\ensuremath{\mathbb{Z}}}
\newcommand{\Ex}{\ensuremath{\mathbb{E}}} 

\def\Prob{\ensuremath{{\mathbb{P}}}}
\def\BV{\mathrm{BV}}
\def\${|\!|\!|} 
\newcommand{\F}{\mathcal{F}} 

\newcommand{\CaX}{\mathcal{C}^\alpha_X}
\newcommand{\CaT}{\mathcal{C}^\alpha_{X,T}}
\newcommand{\XX}{\mathbf{X}}
\newcommand{\Da}{\mathcal{D}^\alpha}
\newcommand{\DaT}{\mathcal{D}^\alpha_T}
\newcommand{\Ap}{\mathcal{A}^{\alpha,p}}
\def\d{\mathrm{d} }
\newcommand{\stf}{\sigma^3_{K_1, K_2 , K_3} }
\newcommand{\stt}{\sigma^2_{K_1, K_2 } }
\newcommand{\CT}{C^\alpha_T }
\renewcommand{\vartheta}{\kappa}

\def \BDG{Burkholder-Davies-Gundy~}

\theoremstyle{plain}
\numberwithin{equation}{section}
\newtheorem{lemma}{Lemma}[section]
\newtheorem{theorem}[lemma]{Theorem}
\newtheorem{proposition}[lemma]{Proposition}
\newtheorem{definition}[lemma]{Definition}

\theoremstyle{definition}
\newtheorem{remark}[lemma]{Remark}

\begin{document}

\title[Rough Burgers]{Rough Burgers-like equations with multiplicative noise}

\author{Martin Hairer}
\address{Martin Hairer, Mathematics Department, University of Warwick }

\author{Hendrik Weber}
\address{Hendrik Weber, Mathematics Department, University of Warwick }

\email{M.Hairer@Warwick.ac.uk, Hendrik.Weber@Warwick.ac.uk}



\date{\today}



%
%

\begin{abstract}
We construct solutions to vector valued Burgers type equations perturbed by a multiplicative space-time white noise in one space dimension. 
Due to the roughness of the driving noise, solutions are not regular enough to be amenable to classical methods. 
We use the theory of controlled rough paths to give a meaning to the spatial integrals involved in the definition of a weak solution. 
Subject to the choice of the correct reference rough path, we prove unique solvability for the equation and we show that 
our solutions are stable under smooth approximations of the driving noise. 
\end{abstract}

\maketitle

\section{Introduction}
\label{sec:intro}

The main goal of this article is to first provide a good notion of a solution and then to establish their existence and uniqueness
for systems of stochastic partial differential equations
of the form
\begin{gather}
\mathrm{d}u\, =  \, \big[ \Delta u + g(u) \partial_x u \big] \d t + \theta(u)\, \mathrm{d}W(t) \label{eq:Stoch-Burg},
\end{gather}
where $u\colon [0,T] \times [0, 1]  \times \Omega \to \Rn$. Here, the functions $g \colon \Rn \to \R^{n \times n}$ as well as 
$\theta \colon\Rn \to \R^{n \times n}$ are assumed to be sufficiently smooth and the operator $\Delta = \partial_x^2$ denotes the Laplacian with periodic boundary conditions on $[0,1]$ acting on each coordinate of $u$. Finally, $W( t)$ is a cylindrical Wiener process on $L^2([0,1],\Rn)$, 
i.e. we deal with space-time white noise.

Our motivation for studying \eqref{eq:Stoch-Burg} is twofold. On the one hand, similar equations, but with additive noise, arise in the context of path sampling,
when the underlying diffusion has additive noise \cite{HairerStuartVoss07,Ha10}. While we expect the case of underlying diffusions with multiplicative noise
to be more complex than the equations considered in this article, we believe that it already provides a good stepping stone for the understanding of 
SPDE with rough solutions and state-dependent noise. Another motivation is the understanding of the solutions to the one-dimensional
KPZ equation \cite{KPZ,BCJ94} and their construction without relying on the Cole-Hopf transform \cite{KPZPreprint}. 
While that equation might at first sight be quite remote from the problem at hand
(it has additive noise and the nonlinearity has a different structure), it is possible to perform formal manipulations of its solutions that yield
equations with features that are very close to those of the equations considered in the present work.

The main obstacle we have to overcome lies in the spatial regularity of the solutions. Actually, solutions to the linear stochastic heat equation 
\begin{gather}
\mathrm{d} X \, =  \,  \Delta X \, \mathrm{d}t + \mathrm{d}W(t) \label{eq:SHE},
\end{gather}
are not differentiable as a function of $(t,x)$. The function $X$ is almost surely $\alpha$-H\"older continuous for every $\alpha <\frac{1}{2}$ as a function of the space variable $x$ and $\frac{\alpha}{2}$-H\"older continuous as a function of the time variable $t$. However, it is almost surely \emph{not} $\frac{1}{2}$-H\"older continuous as a function of $x$.  Since we expect $u$ to have similar regularity properties, 
it is a priori not clear how to interpret the spatial derivative in the nonlinear term $g(u)\partial_x u$. 

So far, equation \eqref{eq:Stoch-Burg} has mostly been studied under the assumption that there exists a function $G$ such that $g=DG$  (see e.g. \cite{Gy98,BCJ94,dPDT94} ). Clearly in the one-dimensional case $n=1$ such a function always exists whereas in the higher-dimensional case this is not true. Then weak solutions can be defined as processes satisfying 
\begin{gather}
 \mathrm{d} \langle \varphi,u \rangle \, =  \, \big[  \langle \Delta \varphi, u \rangle  - \langle\partial_x \varphi,  G(u) \rangle \big]  \d t +  \langle \varphi \, \theta(u) , \d W(t) \rangle  \label{eq:Stoch-Burg-weak},
\end{gather} 
for every smooth periodic test function $\varphi$. 
It is known \cite{Gy98} that under suitable regularity and growth assumptions on $G$ and $\theta$, such solutions can then indeed be constructed
and are unique.
When such a primitive $G$ does not exist, this method cannot be applied and a concept of solutions has not yet been provided.

Recently, in \cite{Ha10} a new approach was proposed to deal with this problem in the additive noise case $\theta =1$. There the non-linearity in \eqref{eq:Stoch-Burg-weak}  is rewritten as
\begin{equation}
\langle \varphi,  g(u) \partial_x u \rangle \,=\,  \int_0^1 \! \varphi(x) \,  g\big(u(x) \big) \,\mathrm{d}_x u(x). \label{eq:stoch-int}
\end{equation}
The (spatial) regularity of $u$ is not sufficient to make sense of this integral in a  pathwise sense using Young's integration theory. This means in particular that extra stochastic cancellation effects are necessary  to give a meaning to  \eqref{eq:stoch-int}. In \cite{Ha10} these problems are treated using Lyons' rough path theory (see  \cite{Ly98, Gu04, FV10}).  

In this approach the definition of integrals like \eqref{eq:stoch-int} is separated into two steps: First a \emph{reference} rough path has to be constructed i.e. a pair of stochastic process $(X,\XX)$ satisfying a certain algebraic condition. The process $\XX$ should be thought of as the \emph{iterated integral}
\begin{equation}\label{eq:itint}
\XX(x,y) \, = \, \int_x^y \big( X(z)-X(x)\big) \otimes \d X(z). 
\end{equation} 
It is usually constructed using a stochastic integral like the  It\^o or Stratonovich integral. 
If on small scales $u$ behaves like $X$ (see Section \ref{sec:RP} for a precise definition of a \emph{controlled rough path}) integrals like \ref{eq:stoch-int} can be defined as continuous functions of their data.

 The advantage of this approach is that the stochastic cancellation effects are captured in the reference rough path and can be dealt with independently of the rest of the construction.  As a reference rough path the solution $X$ of the linear heat equation \eqref{eq:SHE} is chosen. As this process is Gaussian known existence and continuity results \cite{FV07, FV10} rough paths can be applied.

This article provides an extension of \cite{Ha10} to the multiplicative noise case $\theta \neq 1$. The multiplicative noise is included by a second  fixed point argument. One tricky part is that, even if  $\theta$ is a fixed adapted stochastic process, it is no longer clear how to interpret solution to the 
``linear'' heat equation
\begin{equation}
\d \Psi^\theta\, = \, \Delta u +\theta \, \d W 
\end{equation}      
as rough path valued processes. The $\Psi^\theta$ are not Gaussian in general and so the results
of \cite{FV07,FV10} do not apply directly. We resolve this issue by showing that as soon as $\theta$ has sufficient space-time regularity, 
the process $\Psi^\theta$ can be interpreted for every fixed $t>0$ as a rough path (in space) controlled by the rough path $X$
constructed from solutions to \eqref{eq:SHE}. We are able to use this knowledge in Definition~\ref{def:weak-sol} 
below to formulate what we actually mean by a solution to \eqref{eq:Stoch-Burg}.
Such solutions are then  obtained by combining a fixed point argument in a space of deterministic functions to deal with the non-linearity and another  fixed point argument in a space of stochastic processes to deal with the multiplicative noise.

The solutions we construct in this way depend on the choice of the reference rough path $(X,\XX)$. Since there is a priori some freedom in the definition of the iterated integral \eqref{eq:itint}, similar to the choice between It\^o or Stratonovich integral, we can get several possible solutions. 
This will be discussed at the end of Section~\ref{sec:MR}.
 Actually,  these different solutions are not only an artefact of the classical ill-posedness of our equation, since
 they also appear in the gradient case $g=DG$. This may sound surprising, but in a series of recent works \cite{ Ha10-2, HV10,HM10}, approximations to stochastic equations with similar regularity properties as \eqref{eq:Stoch-Burg} were studied. It was shown there that several seemingly natural approximation schemes involving different approximations of the non-linear term may produce non-trivial correction terms in the limit, that correspond exactly to this kind of It\^o-Stratonovich correction.  In the gradient case, this ambiguity can be removed by imposing the additional assumption that the reference rough path is a geometric rough path
(see Section \ref{sec:MR} for a discussion). 
 
 In the non-gradient case however, this is not sufficient to characterise the solution uniquely. Even among the geometric reference rough paths, different \emph{pure area type} terms may appear. We argue that the ``canonical'' construction of $\XX$ given in \cite{FV07,FV10} 
 is still natural for the problem at hand since, if we replace the noise by a mollified version and then send the scale $\eps$ of the mollifier to $0$,
 we show that the corresponding sequence of solutions converges to the solution constructed in this article. Note that the fact that this 
 sequence of solutions even remains bounded as $\eps \to 0$ does not follow from standard bounds. In particular, we would not know a priori
 how to use these approximations to construct solutions to \eqref{eq:Stoch-Burg}. 
These convergence results do not contradict the appearance of the additional correction terms discussed in \cite{Ha10-2, HV10,HM10},
since the approximations considered there also involve an approximation of the nonlinear term $g(u) \d_x u$.

Of course our work is by no means the only work that establishes an application of rough path methods to the theory of stochastic PDE. Recently Gubinelli and Tindel \cite{GT10} used a rough path approach to construct mild solutions to equations of the type
\begin{equation}
\mathrm{d}u \,=\, A u\, dt + F(u)\, \mathrm{d}W 
\end{equation} 
for a rough path $W$ taking values in some space of possibly quite irregular functions. However, in the case of $A$ being the Laplacian on $[0,1]$
and $F$ being a composition operator, they required the covariance operator of the noise to decay at least like $ ( - \partial_{xx}) ^{-1/6}$, thus ruling out
space-time white noise.
Teichmann \cite{Te09} suggests another approach using the method of the moving frame, but these results deal with a different class of equations and are designed to deal with noise terms that are rough in time, but rather smooth in space. Yet another approach to treat equations of the type
\begin{equation}
\mathrm{d}u(t,x) \,=\,F(t,x,Du,D^2 u) \, \mathrm{d}t - Du(t,x) V(x)  \, \mathrm{d} W(t)
 \end{equation}
 using the stochastic method of characteristics is presented by Caruana, Friz, and Oberhauser in \cite{CFO09}, extending ideas from \cite{LS98}. 
  In all of the above works, rough path theory is used to define the temporal integrals. To our knowledge \cite{Ha10} was the first article to make use of rough integrals to deal with spatial regularity issues. 
  
There is also a wealth of literature on the problem of how to treat stochastic PDEs with solutions that are very rough in space. A large part of it
is inspired by the ideas of renormalisation theory coming from quantum field theory. For example, it was possible in 
\cite{MR815192,AlbRock,DPD03,MR1941997} to rigorously construct solutions to the stochastic Allen-Cahn and Navier-Stokes equations in two spatial dimensions, driven by space-time white noise. It is not clear whether these techniques apply to our problem. Furthermore, our results
allow us to obtain very fine control on the solutions and on the convergence of approximations, on the contrary of the 
above mentioned works,  where solutions are only constructed for a set of initial conditions that is of full measure with respect to the invariant measure.
A related but somewhat different approach is to use Wick calculus in Wiener space as in \cite{BDP,NuRo,MR1743612}, but this leads to different equations,
the interpretation of which is not clear.
  
The remainder of the article is structured as follows:  In Section \ref{sec:RP} we give a brief account of those notions of  rough path theory that we will need. We follow Gubinelli's approach \cite{Gu04} to define rough integrals and recall the existence and continuity results for Gaussian rough path from \cite{FV07}. In Section \ref{sec:MR} we give a rigorous definition of our notion of a solution to \eqref{eq:Stoch-Burg} and we state the main results of this work. 
This section also contains a discussion of the dependence of solutions on the choice of reference rough path.
In Section \ref{sec:Prel-Calc},  we then discuss the stochastic convolution  $\Psi^\theta$. In particular, we show that as soon as $\theta$ possesses the right space-time regularity it is controlled by the linear Gaussian stochastic convolution. 
The proof of the main results is then given in  Section \ref{sec:Solutions}. The beginning of this section also contains a sketch of the
main ``two-level'' fixed point argument upon which our proofs rely.

\subsection{Notation}
We will deal with functions $u\,= \, u(t,x; \omega)$ depending on a time and a space variable as well as on randomness. Norms that only depend on the behaviour of $u$ as a function of space for fixed $t$ and $\omega$ will be denoted with $| \cdot |$, norms that depend on the behaviour as a function of $(t,x)$ for fixed $\omega$ with $\| \cdot \|$ and norms that depend on all parameters with $ \$ \cdot \$ $. 

We will denote by $C$ a generic constant that may change its value at every occurrance.

\subsection*{Acknowledgements}

{\small
We are grateful to Peter Friz, Massimilliano Gubinelli, Terry Lyons, Jan Maas, Andrew Stuart, and Jochen Vo\ss\ for several fruitful discussions
 about this work.
Financial support was kindly provided by the EPSRC through grant EP/D071593/1, as
well as by the Royal Society through a Wolfson Research Merit Award and by the Leverhulme Trust through a Philip Leverhulme Prize.
}

\section{Rough Paths}
\label{sec:RP}

In this section we recall the elements of rough path theory which we will need in the sequel and refer the reader to e.g. \cite{ FV10,Gu04,LQ02, LCL07} for a more complete account. We introduce Gubinelli's notion of a controlled rough path \cite{Gu04} and give the main existence and continuity statements for Gaussian rough path from \cite{FV07}. 


For a normed vector space $V$ we denote by $C(V)$ the space of continuous functions from $[0,1]$ to $V$ and by $\Omega C(V)$ the space of continuous functions from $[0,1]^2$ to $V$ which vanish on the diagonal (i.e. for  $R \in \Omega C(V)$ we have $R(x,x)=0$ for all $x \in [0,1]$). We will often omit the reference to the space $V$ and simply write $C$ and $\Omega C$.

It will be useful to introduce for $X \in C$ and $R \in \Omega C$ the operators
\begin{equation}\label{eq:def-delta}
 \delta X (x,y)\,=\, X(y)-X(x)
 \end{equation}
 and
 \begin{equation}\label{eq:def-N}
 N R (x,y,z)\,=\,R(x,z)-R(x,y)-R(y,z). 
\end{equation}
The operator $\delta$ maps $C$ into $\Omega C$ and $N \circ \delta =0$. The quantity $N R$ can be interpreted as an indicator to how far $R$ is from the image of $\delta$. 

In the sequel $\alpha$ will always be a parameter in $\big( \frac{1}{3},\frac{1}{2} \big)$.  For $X \in C$ and $R \in \Omega C$ we define H\"older type semi-norms:
\begin{gather}
| X |_{\alpha} = \sup_{x \neq y } \frac{|\delta X(x,y)|}{|x-y|^\alpha} \quad \text{and} \quad | R |_{\alpha} = \sup_{x \neq y } \frac{|R(x,y)|}{|x-y|^{\alpha}}.\label{eq:norms-rp} 
\end{gather}
We denote by $C^\alpha$ resp. $\Omega C^\alpha$ the set of functions for which these semi-norms are finite. The space $C^\alpha$ endowed with $| \cdot  |_{C^\alpha} = |\cdot |_0 + | \cdot |_\alpha $ is a Banach space. Here $|\cdot|_0$ denotes the supremum norm.  The space $\Omega C^\alpha(V)$ is a Banach space endowed with $| \cdot |_\alpha$ alone. 
\begin{remark}
One might feel slightly uneasy working in these H\"older spaces as they are not separable. We will neglect this issue noting that all the processes we will consider will actually take values in the slightly smaller space of $C^\alpha$-functions that can be approximated by smooth functions which is  separable.
\end{remark}


\begin{definition}\label{def:RP}
An $\alpha$-rough path on $[0, 1]$ consists of a pair $X \in C^\alpha(\Rn)$ and $\mathbf{X}\in \Omega C^{2\alpha}\big(\Rn \otimes  \Rn \big)$ satisfying the relation
\begin{gather}
N \mathbf{X}^{ij}(x,y,z) \,=\,  \delta X^i(x,y)\delta X^j(y,z) \label{eq:cond-it-int}
\end{gather}
for all $0\leq x \leq y\leq z \leq 1$ and all indices $i,j \in \{1, \ldots, n\}$. We will denote the set of $\alpha$-rough paths by $\mathcal{D}^\alpha(\Rn)$ or simply by $\Da$.
\end{definition}

As explained above $\mathbf{X}(x,y)^{ij}$ - called the \emph{iterated integral} should be interpreted as the value of 
\begin{gather}
\mathbf{X}(x,y)^{ij} \,=\, \int_x^y \big(X^i(z)-X^i(x) \big) \mathrm{d} X^j(z) \label{eq:it-int}.
\end{gather}
If $X$ is smooth enough for \eqref{eq:it-int} to make sense it is straightforward to check that $\mathbf{X}(x,y)^{ij}$ defined by \eqref{eq:it-int} satisfies \eqref{eq:cond-it-int}. Indeed, \eqref{eq:cond-it-int} only reflects the fact, that the integral over two disjoint intervals is given by the sum of the integrals over these intervals.  On the other hand we do not require consistency in the sense that $\mathbf{X}(x,y)^{ij}$ needs not be defined by \eqref{eq:it-int}, even if it would make sense. Due to the non-linear constraint \eqref{eq:cond-it-int} the set of $\alpha$-rough paths is not a vector space. But it is a subset of the Banach space $C^\alpha \oplus  \Omega C^\alpha$ and we will use the notation 
\begin{equation}
| (X,\mathbf{X}) |_{\mathcal{D}^\alpha} =  |X|_{C^\alpha} + |\mathbf{X}|_{2 \alpha}.
\end{equation} 
%


Let us recall the basic existence and continuity properties for Gaussian rough paths. Following \cite{FV07} we define for $\rho \geq1$ the 2-dimensional $\rho$-variation of a function $K \colon [0,1]^2 \to \R$ in the cube $[x_1,y_1]\times[x_2,y_2] \subseteq  [0,1]^2$ as
\begin{align}
|K|_{\rho-\text{var}}&\big([x_1,y_1]\times[x_2,y_2] \big) \notag\\
 =\, &  \Big( \sup \sum_{i,j} \big|K(z_1^{i+1}, z_2^{j+1})+K(z_1^{i}, z_2^{j}) - K(z_1^{i}, z_2^{j+1}) - K(z_1^{i+1}, z_2^{j}) \big|^{\rho} \Big)^{\frac{1}{\rho}}\label{eq:2-dvar1},
\end{align}
where the supremum is taken over all finite partitions $x_1 \leq z_1^1 \leq \ldots \leq  z_1^{m_1}=y_1$ and $x_2 \leq z_2^1 \leq \ldots \leq z_2^{m_2}=y_2$

This notion of finite two-dimensional 1-variation does not coincide with the classical concept of being of bounded variation. Actually a function is $\BV$ if its gradient is a vector valued measure, whereas $K$ is of finite 2-dimensional 1-variation if $\partial_x  \partial_y \, K$ is a measure.

The following lemma is a slightly modified version of the existence and continuity results from \cite{FV07}:


\def\FrizVictoir{\cite[Thm~35]{FV07}  and \cite[Cor.~15.31]{FV10}}
\begin{lemma}[\FrizVictoir]\label{lem:gau-rp}
Assume that $X=(X_1(x),\ldots,X_n(x))$, $x \in [0,1]$ is a centred Gaussian process. We assume that the components are independent of each other and denote by $K_i(x,y)$ the covariance function of the $i$-th component.

Assume that there the $K^X_i$ satisfy the following bound for every $0 \leq x \leq y \leq 1$
\begin{equation}\label{eq:2-dvar2}
|K^X_i|_{\rho-\text{var}}\big([x,y]^2  \big) \leq C |y-x|. 
\end{equation}
Then for every $\alpha < \frac{1}{2\rho }$ the process $X$ can canonically be lifted to an $\alpha$  rough path $(X,\mathbf{X})$ and for every $p \geq 1$
\begin{equation}\label{eq:rp-mo-bo}
\Ex\big[|X|_\alpha^p \big] \, < \infty \quad \text{ and } \quad \Ex\big[|\mathbf{X}|_{2\alpha}^p \big] \, < \infty.   
\end{equation}

Furthermore, let  $Y= (Y_1(x), \ldots , Y_n(x) )$ be another  process such that  $(X,Y)$ is jointly Gaussian and  $(X_i,Y_i)$ and $(X_j,Y_j)$ are mutually independent for $i \neq j$. Assume that the covariance functions $K^X_i$ and $K^Y_i$ satisfy \eqref{eq:2-dvar2} and that 
\begin{equation}\label{eq:2-dvar3}
|K^{X-Y}_i |_{\rho-\text{var}}\big([x,y]^2  \big) \leq C \eps^2 |y-x|^{\frac{1}{\rho}}, 
\end{equation}
then for every $p$
\begin{equation}\label{eq:rp-cont}
\Ex \big[ |\mathbf{X}-\mathbf{Y}|_{2\alpha}^p \big]^{1/p} \, \leq \, C \eps . 
\end{equation}
\end{lemma}

In this context canonically lifted means that $\mathbf{X}$ is constructed by considering approximations to $X$, defining approximate iterated integrals using \eqref{eq:it-int} and then passing to the limit. Several approximations including piecewise linear, mollification and a spectral decomposition yield the same result. In the case where $X$ is an $n$-dimensional Brownian motion, $\XX$ is given by \eqref{eq:it-int} where the integral is interpreted as a Stratonovich integral. 

Note that the condition \eqref{eq:2-dvar2} implies that the classical Kolmogorov criterion 
\begin{equation}
\Ex \big[ | X(x) -X(y)|^2   \big]  \, \leq C |x-y|^{\frac{1}{\rho}},
\end{equation}
for $\alpha$-H\"older continuous sample paths is satisfied. 

\begin{remark}\label{rem:gRP}
The results in \cite{FV07, FV10} imply more than we state in Lemma \ref{lem:gau-rp}. In particular, there it is proven that $(X,\XX)$ is a \emph{geometric} rough path. We will not discuss the concept of geometric rough path in detail (see e.g. \cite{FV10} for the geometric  aspects of rough path theory), but remark that it implies that for every $x,y$ the symmetric part of $\XX(x,y)$ is given as
\begin{equation}\label{eq:geom-rp}
\Sym \big( \XX(x,y) \big) = \frac{1}{2} \Big(  \XX(x,y) + \XX(x,y)^T \Big) = \frac{1}{2} \delta X(x,y) \otimes \delta X(x,y).  
\end{equation}
\end{remark}


Our goal is to define integrals against rough paths. Here Gubinelli's approach seems to be best suited to our needs. Thus following  \cite{Gu04} we define:

\begin{definition}
Let $(X,\mathbf{X})$ be in $\mathcal{D}^\alpha(\R^n)$. A pair $(Y,Y')$ with  $Y \in C^\alpha(\R^m) $  and $Y'\in C^{\alpha}(\R^m \otimes \Rn)$ is said to be \emph{controlled} by $(X,\mathbf{X})$ if for all $0 \leq x  \leq y \leq 1$
\begin{gather}
\delta Y(x,y)= Y'(x) \delta X(x,y)  + R_Y(x,y),\label{eq:contr-rp}
\end{gather}
with a remainder $R_Y \in \Omega C^{2\alpha}(\R^m)$. Denote the space of $\R^m$-valued controlled  rough paths $(X, \mathbb{X})$ by $\mathcal{C}_X^{\alpha}(\Rn)$ (or simply $\mathcal{C}_X^{\alpha}$) .  
\end{definition}

Note that the constraint \eqref{eq:contr-rp} is linear and in particular  $\mathcal{C}_X^{\alpha}$ is a vector space although $\mathcal{D}^\alpha $ is not. We will use the notation 
\begin{gather}
| (Y,Y')|_{\mathcal{C}_X^{\alpha}} \, = \, | Y |_{C^\alpha} + | Y' |_{C^\alpha} +  | R_Y |_{2 \alpha}  
\end{gather}
for $(Y,Y') \in \mathcal{C}_X^{\alpha}$.

In general the decomposition \eqref{eq:contr-rp} need not be unique but it is unique as soon as for every $x \in [0,1]$ there exists a sequence $x_n \to x$ such that 
\begin{equation}
\frac{|X(x)-X(x_n)|}{|x-x_n|^{2 \alpha}} \to \infty,
\end{equation}
i.e. if $X$ is rough enough. In most of the situations we will encounter, this will be the case almost surely and there is a natural choice of $Y'$. We will therefore often drop the reference to the derivative $Y'$ and simply refer to $Y$ as a controlled rough path.


Given two $\alpha$-H\"older  paths $Y, Z \in C^\alpha(\Rn)$ it is generally not possible to prove convergence of the Riemann sums
\begin{equation}
\sum_{i} Y(x_i)  \otimes \big(Z(x_{i+1})- Z(x_{i})  \big)
\end{equation}
if the partition $0 \leq x_1 \leq \ldots \leq x_m=1$ becomes finer. If we know in addition that $Y,Z$ are controlled by a rough path $(X,\mathbf{X})$ it is natural to construct the integral $\int Y \mathrm{d} Z$ by a second order approximation 
\begin{equation}\label{eq:Riem-sum2}
\sum_{i} Y(x_i) \otimes \big(Z(x_{i+1})- Z(x_{i})  \big) + Y'(x_i) \mathbf{X}(x_i,x_{i+1})Z'(x_i)^T.
\end{equation}
As we always assume $\alpha >\frac{1}{3}$, it turns out that the approximations \eqref{eq:Riem-sum2} converge:

\def\Gubinelli{\cite[Thm~1 and Cor.~2]{Gu04}}
\begin{lemma}[\Gubinelli]\label{lem:Gub-int}
Suppose $(X, \XX) \in \mathcal{D}^\alpha(\Rn)$ and $Y,Z \in \mathcal{C}_X^{\alpha}(\R^m)$ for some $\alpha  > \frac{1}{3}$. Then the Riemann sums defined in \eqref{eq:Riem-sum2} converge as the mesh of the partition goes to zero. We call the limit \emph{rough integral} and denote it by  $\int Y(x) \otimes \, \mathrm{d} Z(x)$. 

The mapping $(Y,Z) \mapsto \int Y \otimes  \mathrm{d} Z$ is bilinear and we have the following bound:
\begin{gather}\label{eq:IB1}
\int_x^y Y(z) \otimes  \mathrm{d} Z(z)  \, = \, Y(x) \otimes  \delta Z(x,y) + Y'(x )\XX (x,y)Z'(x)^T + Q(x,y), 
\end{gather}
where the remainder satisfies
\begin{align}
\big| Q \big|_{3 \alpha} \, \leq \,&  C \bigg[  |R_Y|_{2 \alpha} |Z|_\alpha + |Y'|_{0}  |X|_{\alpha}  |R_Z|_{2 \alpha}  \notag\\
& \qquad +  |\mathbf{X}|_{2 \alpha} \Big(  |Y'|_\alpha|Z'|_0 +  |Y'|_0|Z'|_\alpha  \Big)  + |X|_\alpha^2 |Y'|_{0} |Z|_\alpha  \bigg]. \label{eq:IB2}    
\end{align}
\end{lemma}

\begin{remark}
In the simplest possible  (one-dimensional) case where $Z(x) =x$, $Y(x)$ is a $C^2$ function and $x_i=\frac{i}{N}$ the second order approximation \eqref{eq:Riem-sum2} corresponds to 
\begin{equation}
\sum_{i} Y (x_i) \big(x_{i+1}-x_{i}  \big) +  \frac{1}{2} Y'(x_i) \big(x_{i+1}-x_{i}  \big)^2,
\end{equation}
and the convergence of this approximation towards $\int Y \mathrm{d}x$ is of order $N^{-2}$ instead of $N^{-1}$ for the simple approximation $\sum_{i} Y (x_i) \big(x_{i+1}-x_{i}  \big)$.
\end{remark}

\begin{remark}
In most situations it will be sufficient to simplify the bounds \eqref{eq:IB1} and \eqref{eq:IB2} to 
\begin{equation}\label{eq:IB3}
\bigg| \int Y(z) \otimes \mathrm{d} Z(z) \bigg|_{\alpha} \leq C \Big(1 + |(X,\mathbf{X})|_{\mathcal{D}^\alpha}^2 \Big)|Y|_{\CaX} |Z|_{\mathcal{C}^\alpha_X},
\end{equation}
Note however that in the original bounds  \eqref{eq:IB1} and \eqref{eq:IB2}, no term including either the product  $ |R_Y |_{2 \alpha} |R_Z|_{2 \alpha}$ or the product $|Y'|_\alpha |Z'|_\alpha$ appears. We will need this fact when deriving a priori bounds in Section \ref{sec:Solutions}.
\end{remark}

The rough integral also possesses continuity properties with respect to different rough paths. More precisely we have:

\begin{lemma}[\cite{Gu04} page 104]\label{lem:Gub-int-stab}
Suppose $(X, \XX), (\bar{X}, \bar{\XX})  \in \mathcal{D}^\alpha$ and $Y,Z \in \mathcal{C}_X^{\alpha}$ as well as $\bar{Y},\bar{Z} \in \mathcal{C}_{\bar{X}}^{\alpha}$ for some $\alpha  > \frac{1}{3}$. Then we get the following bound
\begin{align}\label{eq:IB4}
\Big| \int  Y(z) \otimes  \mathrm{d} Z(z) & - \int \bar{Y}(z) \otimes  \mathrm{d} \bar{Z}(z) \Big|_\alpha \,  \notag\\
\leq \, & C \big( |Y|_{\CaX} + |\bar{Y}|_{\mathcal{C}^\alpha_{\bar{X}}}     \big)  \big( |Z|_{\CaX} + |\bar{Z}|_{\mathcal{C}^\alpha_{\bar{X}}}    \big) \big( |X-\bar{X} |_{C^\alpha} +  |\XX-\bar{\XX} |_{2 \alpha} \big)     \notag\\
 & + C  \big( |Z|_{\CaX} + |\bar{Z}|_{\mathcal{C}^\alpha_{\bar{X}}}    \big)  \big( 1+ \big| \big( X,\XX \big) |_{\Da} + \big| \big(\bar{ X},\bar{\XX} \big) |_{\Da}  \big) \notag\\
 & \qquad \qquad  \cdot \big( |Y- \bar{Y}|_{C^\alpha} +   |Y'- \bar{Y'}|_{C^\alpha}+  |R_Y - R_{\bar{Y}}|_{2 \alpha}   \big) \notag\\
 &+ C   \big( 1+ \big| \big( X,\XX \big) |_{\Da} + \big| \big(\bar{ X},\bar{\XX} \big) |_{\Da}  \big)  \big( |Y|_{\CaX} + |\bar{Y}|_{\mathcal{C}^\alpha_{\bar{X}}}    \big) \notag\\
 & \qquad \qquad  \cdot \big( |Z- \bar{Z}|_{C^\alpha} +   |Z'- \bar{Z'}|_{C^\alpha}+  |R_Z - R_{\bar{Z}}|_{2\alpha}   \big).  
\end{align}
\end{lemma}

This bound behaves as if the rough integral were trilinear in $Y,Z$, and $X$. Unfortunately, this statement makes no sense, as $\Da$ is not a vector space, and $Y$ and $\tilde{Y}$ (resp. $Z$ and $\tilde{Z}$) take values in different spaces $\mathcal{C}^\alpha_X$ and $\mathcal{C}^\alpha_{\bar{X}}$ .

%
%
We will need the following Fubini type theorem for rough integrals:
\begin{lemma}\label{lem:fubini}
Let $(X,\mathbf{X})$ be an $\alpha$-rough path and $Y,Z  \in \mathcal{C}_X^{\alpha}$. Furthermore, assume that $f=f(\lambda,y), \, (\lambda,y) \in [0,1]^2$ is a function which is continuous in $\lambda$ and uniformly $C^1$ in $y$. Denote by $W_\lambda(y)= f(\lambda,y)Y(y)$ and by $W(y)= \big( \int_0^1 f(\lambda,y)\, \mathrm{d}\lambda \big) \, Y(y)$. Then $W$ as well as the $W_\lambda$ are rough path controlled by $X$ and the following Fubini type property holds true
\begin{equation}\label{eq:Fub}
\int_0^1 \Big( \int_0^1 f(\lambda,x)\, \mathrm{d}\lambda \Big)  \, Y(x) \, \mathrm{d} Z(x) \, = \int_0^1 \Big( \int_0^1 f(\lambda,x)\,   \, Y(x) \, \mathrm{d} Z(x) \Big)\,  \mathrm{d}\lambda. 
\end{equation}    
\end{lemma}
\begin{proof}
From the assumptions on $f$ it follows directly that $W$ is an $X$-controlled rough path whose decomposition is given as 
\begin{align}
\delta W(x,y) \,= \,& \bigg( \int_0^1 f(\lambda,x) \, \mathrm{d} \lambda \bigg) \  Y'(x) \delta X(x,y) +  \bigg( \int_0^1 f(\lambda,x) \, \mathrm{d} \lambda \bigg) R_Y(x,y) \notag\\
& \qquad + \bigg( \int_0^1 f(\lambda,x) - f(\lambda,y)\, \mathrm{d} \lambda \bigg) Y(y),
\end{align}
and similarly for the $W_\lambda$. Then we get from the definition of the rough integral \eqref{eq:Riem-sum2}, \eqref{eq:IB1} and \eqref{eq:IB2} that 
\begin{align}
\int_0^1  W(x) \, & \mathrm{d}  Z(x) \notag\\
\, = \,&  \lim \sum_i W(x_i) \delta Z(x_i,x_{i+1}) + W'(x_i) \XX(x_i,x_{i+1}) Z'(x_i)^T \notag\\
= \, & \lim  \int_0^1 \Big(  \sum_i W_\lambda (x_i) \delta Z(x_i,x_{i+1}) + W_\lambda '(x_i) \XX(x_i,x_{i+1}) Z'(x_i)^T  \Big) \, \d \lambda. 
\end{align}
Where the limit is taken as the mesh of the partition $0 \leq x_1 \leq  \cdots \leq x_N  \leq 1$ goes to zero.  As all the error estimates on the convergence of the limit are uniform in $\lambda$ the integral and the limits commute. 
\end{proof}


To end this section we quote a version of a classical regularity statement due to Garsia, Rodemich and Rumsey \cite{GRR70} which has proven to be very useful to give quantitative statements about the regularity of a random paths.

\begin{lemma}[Lemma 4 in \cite{Gu04}] \label{lem:GRR}
Suppose $R \in \Omega C$. Denote by
\begin{gather}
U = \int_{[0,1]^2} \Theta  \bigg( \frac{R(x,y)}{p\big(|x-y|/4 \big)} \bigg) \, \d x \, \d y,   \label{eq:GRR1}
\end{gather} 
where $p: \R^+ \to \R^+$ is an increasing function with $p(0)=0$ and $\Theta$ is increasing, convex, and $\Theta(0)=0$. Assume that there is a constant $C$ such that 
\begin{gather}
\sup_{x \leq u \leq v \leq r \leq y} \Big| NR(u,v,r) \Big| \leq \Theta^{-1} \bigg( \frac{C}{|y-x|^2} \bigg) p(|y-x|/4), \label{eq:GRR2}
\end{gather}
for every $0 \leq x \leq y \leq 1$. Then 
\begin{gather}
\big| R(x,y)  \big| \leq 16 \int_0^{|y-x|} \bigg[ \Theta^{-1} \left( \frac{U}{r^2}\right)  +  \Theta^{-1} \left( \frac{C}{r^2}\right)  \bigg] \d p (r). \label{eq:GRR3}
\end{gather}  
for any $x,y \in [0,1]$.
\end{lemma}

This version of the Garsia-Rodemich-Rumsey Lemma is slightly more general than the usual one (see e.g. \cite[Thm A.1]{FV10}). The classical version treats the case of a functions $f$ on $[0,1]$ taking values in a metric space. Then the same conclusion holds if one replaces $R(x,y)$ with $d \big(A(x),A(y) \big)$ and $C=0$.  In the case that we will mostly use $\Theta(u)=u^p$ and $p(x)=x^{\alpha+2/p}$. Then  Lemma  \ref{lem:GRR} states 
\begin{equation}
|f|_\alpha \leq C \bigg(   \int_{[0,1]^2}  \frac{d\big(f(x),f(y) \big)^p}{|x-y|^{\alpha p+2}}  \, \d x \, \d y \bigg)^{1/p} ,
\end{equation} 
which is essentially a version of Sobolev embedding theorem.

\section{Main results}
\label{sec:MR}

Now we are ready to discuss solutions of the equation \eqref{eq:Stoch-Burg}.  Anticipating Proposition \ref{prop:Gau-conv-rp} we will use that the solution $X$ to linear stochastic heat equation
\[
\d X = \Delta X \d t + \d W_t
\]  
can be viewed as a rough path valued process. To be more precise for every $t$ there exists a canonical choice of $\XX(t,\cdot, \cdot) \in \Omega C^{2\alpha}$ such that the pair $(X(t), \XX(t) )$ is a rough path in the space variable $x$. Furthermore the process $t \mapsto  (X(t), \XX(t) )$ is almost surely continuous in rough path topology. We will use this observation to make sense of the non-linearity by assuming that for every $t$ the solution $u(t,\cdot)$ is a rough path controlled by $X(t,\cdot)$.


\begin{definition}\label{def:weak-sol}
A weak solution to \eqref{eq:Stoch-Burg} is an adapted process $u=u(t,x)$ that takes values in  $C^{\alpha/2,\alpha}\big([0,T] \times [0,1]   \big) \cap L^2([0,T];\CaX)$ such that for every smooth periodic test function $\varphi$ the following equation holds
\begin{align}
\langle \varphi,  u(t) \rangle \, = \, & \langle \varphi,  u_0 \rangle + \int_0^t \! \langle  \Delta \varphi,  u(s) \rangle \mathrm{d}s +  \int_0^t  \! \!\Big(  \int_0^1 \varphi(x) g\big(u(s,x) \big)\  \mathrm{d}_x u(s,x)   \Big) \,\mathrm{d}s  \notag\\
   &+\int_0^t \langle \varphi(x) \, \theta \big( u(s,x) \big) , \mathrm{d}W(s) \rangle. \label{eq:weak-sol}
\end{align}
\end{definition}

Here by $u \in  L^2([0,T];\CaX)$ we mean that almost surely for every $t$ the function $u(t, \cdot)$ is  controlled by the random rough path $\big(X( t, \cdot ) , \XX(t, \cdot) \big)$ and $t \mapsto |u|_{\mathcal{C}_{X(t)}^\alpha}$ is almost surely in $L^2$. By $C^{\alpha/2,\alpha}\big([0,T] \times [0,1]   \big)$ we denote the space of functions $u$ with finite parabolic H\"older norm 

\begin{equation}\label{eq:par-hol}
\| u \|_{C^{\alpha/2,\alpha} }= \sup_{s \neq t, x \neq y} \frac{|u(s,x) -u(t,y) |}{|x-y|^{\alpha}- |s-t|^{\alpha/2}} + \| u\|_0.
\end{equation}

The spatial integral in the third term on the right hand side of \eqref{eq:weak-sol} is a rough integral. To be more precise, as for every $s$ the function $u(s, \cdot) $ is controlled by $X(s, \cdot)$ and $g$ and $\varphi$ are smooth functions,   $\varphi g(u(s,\cdot))$ is also controlled by $X(s, \cdot)$. Thus the spatial integral can be defined as in \eqref{eq:IB1},\eqref{eq:IB2}. This integral is the the limit of the second order Riemann sums \eqref{eq:Riem-sum2} and thus in particular measurable in $s$. For every $s$ this integral is bounded by $C |g|_{C^2} |\varphi |_{C^1} |u(s)|^2_{\mathcal{C}^\alpha_{X(s)}}(1+|(X(s), \XX(s))|_{\Da})$. So the outer integral can be defined as a Lebesgue integral.

A priori this definition depends on the choice of initial condition for the Gaussian reference rough path $X$. In our construction we will use zero initial data but this is not important. In the construction of global solutions we will in fact concatenate the constructions obtained this way and thus restart $X$ at deterministic time steps. This does not alter the definition of solution: Indeed, if $u(t)$ is controlled by $X(t)$ with rough path decomposition given as 
\[
\delta u(t,x,y) \,= \, u'(t,x) \delta X(t,x,y) + R^u(t,x,y)
\] 
and the modified rough path $\tilde{X}$ is given as the solution to the stochastic heat equation started at zero at time $t_0$ then for $t \geq t_0$ one has 
\[
\tilde{X}(t,x) = X(t,x) - S(t-t_0) X(t_0, \cdot) (x).
\] 
Therefore, $u$ has a rough path decomposition with respect to $\tilde{X}$ which is given as
\[
\delta u(t,x,y) \,= \, u'(t,x) \delta \tilde{X}(t,x,y) + R^u(t,x,y) - u'(t,x) \delta S(t-t_0) X(t_0) (x,y).
\] 
So $u$ is also controlled by $\tilde{X}$. Due to standard regularisation properties of the heat semigroup 
\[
| u'(t) \delta S(t-t_0) X(t_0) |_{2 \alpha} \leq C |u'|_0 (t-t_0)^{-\frac{\alpha}{2}} |X(t_0)|_\alpha. 
\]
This is square integrable, and so $u$ satisfies the regularity assumption with respect to this reference rough path  as well. Another choice of reference rough path would be to take the stationary solution to the stochastic heat equation 
\[
\d \tilde{X}=  \big( \Delta \tilde{X} - \tilde{X} \big) \d t + \d W_t.
\]
This would also yield the same solutions.


Our notion of a weak solution has an obvious \emph{mild solution} counterpart:

\begin{definition}\label{def:mild-sol}
A mild solution to \eqref{eq:Stoch-Burg} is a process $u$ such as in Definition \ref{def:weak-sol} with equation \eqref{eq:weak-sol} replaced by
\begin{align}
u(x,t)\, = \,&  S(t) u_0(x) + \int_0^t \Big(\int_0^1 \hat{p}_{t-s}(x-y) \, g(u(s,x)) \, \mathrm{d}_y u(s,y)  \Big) \, \mathrm{d}s  \notag\\
& +\int_0^t S(t-s) \, \theta(u(s))\, \mathrm{d}W(s) (x). \label{eq:mild-sol}
\end{align} 
\end{definition}

Here $S(t)=e^{\Delta t}$ denotes the heat semigroup with periodic boundary conditions, given by convolution with the heat kernel $\hat{p}_t$.  As above the spatial integral involving the nonlinearity $g$ is a rough integral. 
\begin{remark}
In Section \ref{sec:Solutions}  we will see that the term involving the nonlinearity $g(u)$  on the right hand side of \eqref{eq:mild-sol} is always $C^1$ in space. Thus, using Proposition \ref{prop:Stoch-Conv-Rough-Path}   we will show that the controlled rough path decompositon of $u(t, \cdot)$ is given as
\begin{equation}\label{eq:rp-dec-u}
\delta u(t, x,y) = \theta \big(u(t,x) \big) \, \delta X(t,x,y) + R(t,x,y).
\end{equation}
\end{remark}

 As expected the notions of weak and mild solution coincide:
%


\begin{proposition}\label{prop:weak-mild}
Every weak solution to \eqref{eq:Stoch-Burg} is a mild solution and vice versa.
\end{proposition}

%
%

\begin{proof}
The proof follows the lines of the standard proof (see e.g. \cite{dPZ} Theorem 5.4). We only need to check that the argument is still valid when we use rough integration. 

Assume first that $u$ is a mild solution to \eqref{eq:Stoch-Burg} and let us show that then \eqref{eq:weak-sol} holds. To this end let $\varphi$ be a periodic $C^2$ function. Testing $u$ against $\Delta \varphi$ and integrating in time we get
\begin{align}
\int_0^t \! \langle \Delta & \varphi, u(s) \rangle \, \d s\,  = \,   \int_0^t  \! \langle \Delta \varphi, S(s)u_0 \rangle \,  \d s + \int_0^t   \! \langle \Delta \varphi,    \Psi^{\theta(u)}(s) \rangle  \notag \\
& +  \int_0^t   \!   \int_0^s \Big( \int_0^1 \Delta \varphi(x)  \int_0^1 \hat{p}_{s-\tau}(x-y) \, g(u(\tau,x)) \, \mathrm{d}_y u(\tau,y)\,  \d x  \Big) \,  \d \tau  \, \d s.  \label{eq:wm1}
\end{align}
Let us treat the term involving the nonlinearity $g$ on the right hand side of \eqref{eq:wm1}. Using Fubini theorem for the temporal integrals and then the Fubini theorem for rough integrals \eqref{eq:Fub} we can rewrite this term as
\begin{align}
 \int_0^t   \! &    \Big(  \int_0^1 \Big[  \int_\tau^t  \int_0^1 \Delta \varphi(x) \,  \hat{p}_{s-\tau}(x-y)  \, \d x\, \d s \Big]\, g(u(\tau,x)) \, \mathrm{d}_y u(\tau,y)    \Big) \,   \d \tau  \notag\\
 &=    \, \int_0^t   \!     \Big(  \int_0^1 \! \Big[   \int_0^1  \varphi(x) \,  \hat{p}_{s}(x-y)\,  \d x  - \varphi(y)     \Big]\, g(u(\tau,x)) \, \mathrm{d}_y u(\tau,y)    \Big) \,   \d \tau  \label{eq:wm2}.
\end{align}
where we have used the definition of the heat semigroup in the first line. Similar calculations for the remaining terms in \eqref{eq:wm1} (see \cite{dPZ} Theorem 5.4 for details) show that
\begin{equation}\label{eq:wm3}
 \int_0^t  \! \langle \Delta \varphi, S(s)u_0 \rangle \,  \d s \, =\,  \langle \varphi , S(t) u_0 \rangle -  \langle \varphi, u_0 \rangle,  
\end{equation}
and 
\begin{equation}\label{eq:wm4}
\int_0^t  \! \langle \Delta \varphi,   \Psi^{\theta(u)}(s) \rangle  \,  \d s =  \langle \varphi,   \Psi^{\theta(u)}(s) \rangle  \rangle \, - \int_0^t  \! \langle \varphi,   \theta(u)(s) \d W(s) \rangle.   
\end{equation}
Thus summarising \eqref{eq:wm1} - \eqref{eq:wm4}  we get \eqref{eq:weak-sol}.

In order to see the converse implication, first note that in the same way as in (\cite{dPZ} Lemma 5.5) one can show that if $u$ is a weak solution the following identity holds if $u$ is tested against a smooth, periodic and time dependent test function $\varphi(t,x)$ 
\begin{align}
\langle \varphi(t) ,  u(t) \rangle \, = \, & \langle \varphi(0),  u_0 \rangle + \int_0^t \! \big\langle \big(  \Delta  \varphi(s) + \partial_s  \varphi(s)  \big),  u(s) \big\rangle \mathrm{d}s +  
   \int_0^t \langle \varphi(s) \, \theta \big( u(s) \big) , \mathrm{d}W(s) \rangle  \notag\\ 
    &+ \int_0^t  \! \!\Big(  \int_0^1 \varphi(s,x)
g\big(u(s,x) \big)\  \mathrm{d}_x u(s,x)   \Big) \,\mathrm{d}s  \label{eq:wm5}
\end{align}
Taking $\varphi(t,x)= S(t-s)  \xi(x)$ for a smooth function $\xi$ and using the definition of the heat semigroup one obtains
\begin{align}
\langle u(t) , \xi \rangle \, = \, & \langle S(t) \xi,   u_0 \rangle +  
   \int_0^t \! \langle S(t-s) \xi \, \theta \big( u(s) \big) , \d W(s) \rangle  \notag\\ 
    &+ \int_0^t  \! \!\Big(  \int_0^1 \! \int_0^1 \hat{p}_{t-s}(x-y)\,  \xi(y) \,  
g\big(u(s,x) \big)\  \mathrm{d}_x u(s,x)  \, \d y  \Big) \,\d s.   \label{eq:wm6}
\end{align}  
Thus using again the Fubini theorem for rough integrals as well as the symmetry of the heat semigroup one obtains  \eqref{eq:mild-sol}.   
\end{proof}


Weak solutions do indeed exist:
\begin{theorem}\label{thm:MR}
Fix $\frac{1}{3} < \alpha  < \beta <\frac{1}{2}$ and assume that $u_0 \in C^\beta$. Furthermore, assume that $g$ is a bounded $C^3$ function with bounded derivatives up to order $3$ and that $\theta$ is a bounded $C^2$ function with bounded derivatives up to order $2$. Then for every time interval $[0,T]$ there exists a unique weak/mild solution to \eqref{eq:Stoch-Burg}.
\end{theorem}

The construction will be performed in the next section using a Picard iteration. The statement remains valid, if one includes extra terms that are more regular, as for example a reaction term $f(u)$ for a bounded $C^1$ function $f$ with bounded derivatives.

%
%

It is important to remark that although the rough path $(X,\XX)$ does not explicitly appear in \eqref{eq:weak-sol} it plays a crucial role in determining the solution uniquely. Let us illustrate this in the one-dimensional case.  

We could define another rough path valued process $(X,\tilde{\XX})$ by setting 
\begin{equation}
\tilde{\XX}(t,x,y) \, = \, \XX(t,x,y) + (x-y).
\end{equation}
This corresponds to introducing an It\^o -- Stratonovich correction term in the iterated integral. Then it is straightforward to check that $(X,\tilde{\XX})$ also satisfies \eqref{eq:cond-it-int} and \eqref{eq:weak-sol} makes perfect sense if we view the space integral in the nonlinearity as a rough integral controlled by $(X,\tilde{\XX})$. Using the definition of the rough integral in Lemma \ref{lem:Gub-int} and recalling \eqref{eq:rp-dec-u} we get 
\begin{equs}
\widetilde{\int_0^1} &\, \varphi(x) \,  g\big( u(x) \big)\, \mathrm{d}_x u(x)  \notag \\
&\, = \, \lim \, \sum_{i} \varphi(x_i) \,  g\big( u(x_i) \big) \big( u(x_{i+1}) - u(x_{i}  )  \big) + \varphi(x_i) g'\big(u(x_i) \big)\,  \theta \big(u(x_i) \big)^2 \, \tilde{\XX}(x_i,x_{i+1}) \notag \\
&\, = \, \lim \sum_{i} \varphi(x_i) g\big( u(x_i) \big) \big( u(x_{i+1}) - u(x_{i}  )  \big) \notag\\
&\qquad \qquad  \qquad  + \varphi(x_i) g'\big(u(x_i) \big)  \theta \big(u(x_i) \big)^2 \Big(  {\XX}(x_i,x_{i+1}) + (x_{i+1} - x_i) \Big) \notag \\
& \, = \,  \int_0^1  \varphi(x)   g\big( u(x) \big)  \mathrm{d}_x u(x) + \int_0^1 \varphi(x) g'(u(x))   \, \theta \big(u(x) \big)^2\mathrm{d}x, \label{eq:extra-term}
\end{equs}   
where by $\widetilde{\int}$ we denote the rough integral with respect to $(X,\tilde{\XX})$. The limit is taken as the mesh of the partition $0 \leq x_1 \leq \cdots \leq x_N \leq 1$ goes to zero.

In particular, the reaction term $g'(u) \theta (u)^2 $ appears. In the additive noise case $\theta=1$ this is precisely the extra term found in the approximation resuls in  \cite{Ha10-2, HV10,HM10}. One can thus interpret these results, saying that the approximate solutions actually converge to solutions of the right equation with a different choice of reference rough path.

In the gradient case $g=DG$ our solution coincides with the classical solution defined using integration by parts. To see this, it is sufficient to see that the nonlinearity can be rewritten as 
\begin{equation*}
 \int_0^1 \!  \varphi(x) \,  g\big(u(s,x) \big) \,  \mathrm{d}_x u(s,x)  \,= \,  - \int \! \partial_x  \varphi(x) \, G\big(u(x) \big) \, \mathrm{d}x.
\end{equation*} 
For simplicity, let us argue component-wise and show this formula for a function $G$ that takes values in $\R$, i.e, we have $g = \nabla G$. We will assume that $G$ is of class $C^3$ with bounded derivatives up to order $3$.  To simplify the notation we leave out the time dependence and write $\varphi_i, u_i$ for $\varphi(x_i), u(x_i)$ etc.

Then we can write 
\begin{align}
\int_0^1 \! & \varphi(x) \,  g\big(u(x) \big) \,  \mathrm{d}_x u(x)  \notag \\
&  =\, \lim \sum_i  \varphi_i \, \Big[   \nabla G (u_i  ) \cdot  \delta u_{i,i+1}  +  \Big(    \partial^{\alpha, \beta} G (u_i ) \,\big( u_i' \big)^{\alpha \alpha'}  \big( u_i' \big)^{\beta \beta'}  \, \XX_{i,i+1}^{\alpha' \beta'}   \Big)   \Big] \notag\\   
& = \,  \lim \sum_i  \varphi_i \, \Big[  G(u_{i+1}) - G(u_i) +Q_i   \Big]. \label{eq:grad-int1}   
\end{align} 
The sum involving the differences of $G$ converges to  $- \int \partial_x \varphi \,  G \, \mathrm{d}x$, so that it only remains to bound $Q_i$. Here the crucial ingredient is the fact that the Gaussian rough paths constructed by Friz and Victoir are \emph{geometric} (see Remark \ref{rem:gRP}). In particular we have for all $x,y$:
\begin{equation}\label{eq:geom-rp2}
\Sym \big( \XX(x,y) \big) \,= \,   \delta X(x,y) \otimes \delta X(x,y).  
\end{equation} 
As $D^2 G$ is a symmetric matrix the antisymmetric part of $\XX$ does not contribute towards the  second term  in the second line of \eqref{eq:grad-int1} and we can rewrite with Taylor formula
\begin{align}
| Q_i | \, \leq \,& \big| D^3 G \big|_0 | \delta u_{i,i+1}|^3 + \big| D^2 G \big|_0 |R_u(x_i,x_{i+1}) |^2 \notag\\
 \leq \,& \ |G|_{C^3} \Big(  |u|_{C^\alpha}^3 \big(x_{i+1} -x_i \big)^{3\alpha} + |R_u|_{2 \alpha}^2 \big(x_{i+1} -x_i \big)^{4 \alpha} \Big). 
\end{align}
This shows that the sum involving $Q_i$ on the right hand side of \eqref{eq:grad-int1} vanishes in the limit. 

This discussion shows in particular that in the gradient case $g=DG$ the concept of weak solution is independent of the reference rough path if we impose the additional assumption that it is geometric i.e. that it satisfies \eqref{eq:geom-rp2}. The extra reaction term we obtained above in \eqref{eq:extra-term} is due to the fact that the modified rough path $(X,\tilde{\XX})$ is not geometric. 

In the non-gradient case extra terms may appear even for geometric rough paths. To see this, one can for example consider the two-dimensional example  
\begin{equation}
\widetilde{\XX}(t,x,y) \,=\, \XX(t,x,y) + (y-x) \left( \begin{array}{cc}
0 & 1  \\
-1 & 0 
\end{array} 
\right)
\end{equation}
and 
\begin{equation}
g(u_1,u_2) \,= \, \left( \begin{array}{cc}
u_2 & -u_1  \\
u_2 & -u_1 
\end{array} 
\right).
\end{equation}
Then a similar calculation to \eqref{eq:extra-term} shows that an extra reaction term  
\begin{equation}
 \tr \bigg[ \left(\begin{array}{cc}
0 & 1  \\
-1 & 0 
\end{array}
\right)
  \theta(u)
  \left(\begin{array}{cc}
0 & 1  \\
-1 & 0 
\end{array}
\right)
   \theta(u)^T \bigg]
\end{equation}
 appears in each coordinate.

We justify our concept of solution by showing that it is stable under smooth approximations. To this end let  $\eta_\eps(x)= \eps^{-1}\eta(x/\eps)$ be standard mollifiers as introduced in \eqref{eq:mollifier}. For a fixed initial data $u_0\in C^\beta$ denote by $u_\eps$ the unique solution to the smoothened equation 
\begin{equation}
\mathrm{d}u_\eps \, =  \, \big[ \Delta u_\eps + g(u_\eps) \partial_x u_\eps \big] \d t + \theta(u_\eps)\, \mathrm{d} \big( W \ast \eta_\eps  \big).  \label{eq:Stoch-Burg-smooth}
\end{equation}
 Note that due to the spatial smoothening of the noise there is no ambiguity whatsoever in the interpretation of \eqref{eq:Stoch-Burg-smooth} . We then get

\begin{theorem}\label{thm:Stability}
Fix $\frac{1}{4}< \alpha < \beta < \frac{1}{2}$. Fix an initial data $u_0 \in C^\beta$ and assume that $g$ and $\theta$ are as in Theorem \ref{thm:MR}. Then we have for any $\delta > 0$ and every $T>0$
\begin{equation}
\Prob \Big[ \| u - u_\eps  \|_{C^{\alpha/2,\alpha} [0,T] \times [0,1]  } \geq \delta \Big] \, \to 0 \qquad \text{as } \eps \downarrow 0.  
\end{equation}
\end{theorem}
Recall that the parabolic H\"older norm on $ \| u \|_{C^{\alpha/2,\alpha} [0,T] \times [0,1] }$ was defined in \eqref{eq:par-hol}.

\section{Stochastic Convolutions}
\label{sec:Prel-Calc}

This section deals with properties of stochastic convolutions. Fix a probability space $(\Omega, \F, \Prob)$ with a right continuous, complete filtration $(\F_t)$. Let $W(t) = (W^1(t), \ldots, W^n(t))$, $t \geq 0$ be a cylindrical Wiener process on $L^2\big([0,1], \R^n \big)$. Then for any adapted $ L^2\big( [0, 1], \R^{n \times n } \big)$ valued process $\theta$ define 
\begin{gather}
\Psi^\theta (t)\, = \, \int_0^t S(t-s)\, \theta(s) \, \mathrm{d}W(s). \label{eq:Stoch-Conv}
\end{gather}
Recall that $S(t)=e^{t \Delta}$ denotes the semigroup associated to the Laplacian with periodic boundary conditions, which acts independently on all of the components. The Gaussian case $\theta=1$ will play a special role and we will denote it with $X= \Psi^1$. 

We will also treat the stochastic convolution for a smoothened version of the noise. To this end let $\eta:\R \to \R$ be a smooth,  non-negative, symmetric function with compact support and with $\int \eta(x) \, \d x =1$. Furthermore, for $\eps >0$ 
\begin{equation}\label{eq:mollifier}
\eta_\eps(x)= \eps^{-1}\eta\left( \frac{x}{\eps} \right).
\end{equation}
Then denote by 
\begin{gather}
\Psi_\eps^\theta (t)\, = \, \int_0^t S(t-s)\, \theta(s) \, \mathrm{d}\big( W \ast \eta_\eps \big)(s), \label{eq:Stoch-Conv-smooth}
\end{gather}
the solution to the stochastic heat equation with spatially smoothened noise. Note that the convolution in \eqref{eq:Stoch-Conv-smooth} only acts on the spatial argument $x$. In particular, the temporal integral still hast to be interpreted as an It\^o integral. As above in the special case $\theta =1$ we will write $X_\eps = \Psi^1_\eps$.

The main purpose of this section is to establish the following facts: In the Gaussian case $\theta=1$ for every $t$ the stochastic convolution $X(t, \cdot) = \Psi^1$ satisfies the criteria of Lemma \ref{lem:gau-rp} and therefore can be lifted to a rough path $(X(t, \cdot) ,\XX(t,\cdot) )$ in space direction. Furthermore, we will be able to conclude that $(X(t, \cdot) ,\XX(t,\cdot) )$ is a continuous process taking values in $\Da$.  

If  $\theta \neq 1$ is not deterministic, the process $\Psi^\theta$ is in general not Gaussian and Lemma \ref{lem:gau-rp} does not apply. But we will establish that as soon as $\theta$ has the right space-time regularity  the stochastic convolution $\Psi^\theta(t, \cdot)$ is controlled by $X(t, \cdot)$ for every $t$. The essential ingredients are the estimates in Lemma \ref{lem:IR}. Finally, we will show that all of these results are stable under approximation with the smoothened version defined in \eqref{eq:Stoch-Conv-smooth}.

Recall that for $f\in L^2[0,1]$ we can write
\begin{equation}
S(t)f (x) \,=\, \int_0^1 \hat{p}_t(x-y) f(y)  \, \mathrm{d} y,
\end{equation}
and that one has an explicit expression for the heat kernel $\hat {p}_t$ either in terms of the Fourier decomposition
\begin{equation}\label{eq:Fou}
\hat{p}_t(x) \,= \, \sum_{k \in \Z} \, \exp \big(  - (2 \pi k)^2 t \big) \,\exp \big(i (2 \pi k) (x)   \big),
\end{equation}  
or from a reflection principle
\begin{equation}\label{eq:RePr}
\hat{p}_t(x) \,= \, \sum_{k \in \Z} p_t(x-k),
\end{equation}
where $p_t(x) = \frac{1}{(4 \pi t)^{1/2}} \exp \big( -x^2/4 t  \big)$ denotes the usual Gaussian heat kernel.

%

%
The following two lemmas about some integrals involving $\hat{p}_t$ are the central part of the proof of  regularity for the stochastic convolutions. We first recall the following bounds:
\begin{lemma}\label{lem:CR}
The following bounds hold:
\begin{enumerate}
\item \emph{Spatial regularity:} For $x,y \in [0,1]$  and $t \in [0,T]$ 
\begin{equation}\label{eq:SR1}
\int_0^t \! \! \int _0 ^1 \Big( \hat{p}_{t-s}(x-z) - \hat{p}_{t-s}(y-z)  \Big)^2   \, \mathrm{d}s \,\mathrm{d}z \, \leq \, C \big| x-y  \big|.
\end{equation}
\item \emph{Temporal regularity:} For $x \in [0,1]$ and $0 \leq s < t \leq T$ 
\begin{align}
\int_0^s  \!  \!\int _0 ^1 \Big( \hat{p}_{t-\tau}(x-z) &- \hat{p}_{s-\tau}(x-z)  \Big)^2 \, \mathrm{d}\tau\, \mathrm{d}z \notag\\
&+ \int_s^t  \! \!\int _0 ^1 \Big( \hat{p}_{t-\tau}(x-z)\Big)^2 \, \mathrm{d}\tau \, \mathrm{d}z  \, \leq \, C (t-s)^ {1/2}. \label{eq:TR}
\end{align}
\end{enumerate}
\end{lemma} 
\begin{proof}
These bounds are essentially well-known and can be derived easily from the expression \eqref{eq:Fou}
by making use of the inequalities $|\sin(x)| \le 1\wedge |x|$ and $1-e^{-x} \le 1\wedge x$ for $x \ge 0$.
\end{proof}
%
%
%
%
%
The following lemma is the essential step to prove that $\Psi^\theta$ is controlled by $X$: 
\begin{lemma}\label{lem:IR}
The following bounds hold:
\begin{enumerate}
\item \emph{Regularisation by a time dependent function: }For all $x,y \in [0,1]$ and every $t \in[0,T]$
\begin{gather}
\int_0^t  \! \! \int_0^1  \Big(\hat{p}_{t-s}(y-z) - \hat{p}_{t-s}(x-z)  \Big)^2 |s-t|^{\alpha}  \, \d z \,\d s \, \leq \, C |x-y|^{1+2 \alpha}.\label{eq:bou-I1}
\end{gather}
\item \emph{Regularisation by a space dependent function: }For all $x,y \in [0,1]$ and every $t \in [0,T]$
\begin{gather}
\int_0^t \!  \! \int_0^1 \Big(\hat{p}_{t-s}(y-z) - \hat{p}_{t-s}(x-z)  \Big)^2  |x-z |^{2\alpha}  \, \d z \, \d s\,  \leq \, C |x-y|^{1+2 \alpha}. \label{eq:bou-I2}
\end{gather}
\end{enumerate}
\end{lemma}

%
%

\begin{proof}
Due the periodicity of $\hat{p_t}$ we can assume that $x=0$ and replace the spatial integral by an integral over $[-\frac{1}{2}, \frac{1}{2}]$. Furthermore, without loss of generality we can assume that $0 \leq y \leq \frac{1}{4}$. Due to the explicit formula \eqref{eq:RePr} one can see that for  $z \in [-\frac{3}{4},\frac{3}{4}]$ one can write  $\hat{p}_t(z)=p_t(z) +p_t^R(z)$ where the remainder $p_t^R$ is a smooth function with uniformly bounded derivatives of all orders. Thus, using
\begin{align*}
 \Big(\hat{p}_{t-s}(z-y) -& \hat{p}_{t-s}(z)  \Big)^2 \\
 & \leq \, 2 \Big(p_{t-s}(z-y) - p_{t-s}(z)  \Big)^2 + 2 \Big(p^R_{t-s}(z-y) - p^R_{t-s}(z)  \Big)^2 ,
\end{align*} 
and that 
\begin{gather}
\int_0^t  \!\!  \int_0^1  \Big(p^R_{t-s}(z-y) - p^R_{t-s}(z)  \Big)^2 |s-t|^{\alpha}  \, \d z \, \d s \leq C | p^R |_{C^1}  |y|^2,
\end{gather}
as well as the analogous bound involving he term $|x-y|^{2 \alpha}$, one sees that it is sufficient to prove \eqref{eq:bou-I1} and \eqref{eq:bou-I2} with $\hat{p_t}$ replaced by the Gaussian heat kernel $p_t$.

%
%
%

(i) Let us treat the integral involving $|s-t|^{\alpha}$ first.  To simplify notation we denote by $I_1=\int_0^t  \int_{-\frac{1}{2}}^{\frac{1}{2}}  \Big(\hat{p}_{t-s}(z-y) - \hat{p}_{t-s}(z)  \Big)^2 |s-t|^{\alpha}  \,\d z \, \d s$. A direct calculation shows

\begin{align}
\int_\R \Big(p_{t-s}(y-z) - p_{t-s}(z) \Big)^2 \d z \, = \, \frac{1}{\sqrt{2 \pi }} (t-s)^{-1/2} \bigg[1 - \exp\bigg(-\frac{y^2}{8(t-s)} \bigg) \bigg]. \label{eq:rp6}
\end{align}

 So one gets:

\begin{align}
I_1 \, \leq \,& \frac{1}{\sqrt{2 \pi }} \int_0^t (t-s)^{ \alpha - 1/2} \bigg[1 - \exp\bigg(-\frac{y^2}{8(t-s)} \bigg) \bigg] \d s \notag\\
   = \,& \frac{1}{\sqrt{2 \pi }} 8^{-1/2-\alpha} |y|^{1+2 \alpha} \int_{\frac{y^2}{8t}}^\infty \tau^{-3/2 - \alpha} \big[1- e^{-\tau} \big] \mathrm{d} \tau. \label{eq:rp7}
\end{align}

Here in the last step we have performed the change of variable $\tau=\frac{y^2}{8(t-s)}$. Note that the integral in the last line of  \eqref{eq:rp7} converges at $0$ due to $\alpha < \frac{1}{2}$. Thus we can conclude 
\begin{align}
 I_1 \, \leq \frac{1}{\sqrt{2 \pi }} 8^{-1/2-\alpha} |y|^{1+2 \alpha} \int_{0}^\infty \tau^{-3/2 - \alpha} \big[1- e^{-\tau} \big]\, \d \tau \leq C |y|^{1+2 \alpha}. \label{eq:rp8}
\end{align}
%
%
%
%
(ii) Let us now prove the bound  \eqref {eq:bou-I2}.  We write   
\[
I_2 \, =  \, \int_0^t \int_0^1 \Big(p_{t-s}(z-y) - p_{t-s}(z)  \Big)^2  |z |^{2\alpha}  \, \d z \, \mathrm\d s. 
\] 
We decompose the integral into
\begin{align}
I_2 \, \leq \,& \int_{(t-|y|^2)\wedge 0}^t \int_\R \Big(p_{t-s}(z-y) - p_{t-s}(z)  \Big)^2  |z |^{2\alpha}  \, \d z \, \d s \label{eq:rp9} \\
&+\int_0^{(t-|y|^2)\wedge 0} \int_\R \Big(p_{t-s}(z-y) - p_{t-s}(z)  \Big)^2  |z |^{2\alpha}  \, \d z \, \d s \notag \\
= \, & I_{2,1} + I_{2,2}. \notag
\end{align}
For the first term we get
\begin{align}
I_{2,1} \, \leq \, & 2 \int_{(t-|y|^2)\wedge 0}^t \int_\R  \Big[ \big(p_{t-s}(z-y) \big)^2 + \big(p_{t-s}(z)  \big)^2 \Big]  |z |^{2\alpha}  \, \d z  \, \d s \notag \\
\leq \,& C \int_{(t-|-y|^2)\wedge 0}^t \int_\R (t-s)^{-1} \exp \bigg(- \frac{z^2}{2(t-s)} \bigg) \Big(|z |^{2\alpha} + |y |^{2\alpha}  \Big) \, \d z \, \d s\notag\\
\leq \,&  C \int_{(t-|y|^2)\wedge 0}^t (t-s)^{-1+1/2 + \alpha}\, \d s  + C \int_{(t-|y|^2)\wedge 0}^t (t-s)^{-1+1/2} |y|^{2 \alpha} \, \d s \notag\\
\leq \,& C |y|^{1+2 \alpha}.   \label{eq:rp10}   
\end{align}
In order to get an estimate on the second term we write
\begin{align}
\Big( p_{t-s}(z-y) - p_{t-s}(z)  \Big)^2 \,\leq\,   \Big( | y | \sup_{\xi \in [z-y,z ] } \big| p_{t-s}'(\xi) \big| \Big)^2.   
\end{align}
We have for $t \geq y^2$
\begin{align}
\sup_{\xi \in [z-y,z ] } \big| p_t'(\xi) \big| \,& \leq \, \sup_{\xi \in [z-y,z ] } \frac{|\xi | }{2 t \sqrt{4 \pi t }} \exp \Big( - \frac{\xi^2}{4t}  \Big) \notag\\
& \leq \,  \frac{|z | + |y| }{2 t \sqrt{4 \pi t }} \exp \Big( - \frac{z^2}{8t}  \Big) e^{1/4},
\end{align}
where we have used the the identity $(a+b)^2 \geq \frac{1}{2}a^2-b^2$ as well as $t \geq y^2$. 
Thus we can write
\begin{align}
I_{2,2} \,& \leq \,  C |y|^2  \int_0^{(t-y^2)\wedge 0} \int_\R  \frac{z^2 + y^2 }{(t-s)^3 } \exp \Big( - \frac{z^2}{4(t-s)}  \Big)   |z|^{2\alpha}  \, \d z \, \d s \notag \\
& \leq \,   C y^2  \int_0^{(t-y^2)\wedge 0}  \frac{(t-s) + y^2 }{(t-s)^3 } (t-s)^{1/2 + \alpha} \d s   \notag \\
& \leq \,   C |y|^{1+2 \alpha} .
\end{align}
This completes the proof.  
\end{proof}
%
%
%
%
%
It is a well known fact and an immediate consequence of Lemma \ref{lem:CR} together with the Kolmogorov criterion that (say for bounded $\theta$) for any $\alpha<\frac{1}{2}$ the stochastic convolution $\Psi^\theta$ is almost surely $\alpha$-H\"older continuous as a function of the space variable $x$ for a fixed value of the time variable $t$ and almost surely $\alpha/2$-H\"older continuous as function of the time variable $t$ for a fixed value of $x$. The following proposition makes this statement more precise:
\begin{proposition}\label{prop:Reg-Stoch-Conv}
For $p \geq 2$ denote by 
\begin{equation}
\$  \theta \$_{p,0} \, = \, \Ex \big[ \sup_{(t,x)} | \theta(t,x)|^p   \big]^{1/p}.
\end{equation}
For every $\alpha < \frac{1}{2}$  and every $p >      \frac{12}{1-2 \alpha} $ 
set $\vartheta = \frac{1- 2 \alpha}{4} -\frac{3}{p}$. Then we have
\begin{equation}
\Ex \Big[   \big\| \Psi^{\theta} \big\|_{C^{\vartheta}( [0,T]; C^{ \alpha} )  }^p        \Big] \, \leq \, C \$ \theta \$_{p,0}^p. \label{eq:Hol-Bou-Psi1}
\end{equation}
\end{proposition}

\begin{remark}\label{rem:reg-stoch-conv}
In particular using that $\Psi(0,x)=0$ this proposition implies that 
\begin{equation}
\Ex \Big[   \big\| \Psi^{\theta} \big\|_{C( [0,T]; C^{ \alpha} )  }^p        \Big] \, \leq \,  C T^{\vartheta} \$ \theta \$_{p,0}^p \label{eq:Hol-Bou-Psi2}. 
\end{equation}
On the other hand, by interchanging the roles of $\vartheta$ and $\alpha$, one obtains $ \Ex \Big[   \big\| \Psi^{\theta} \big\|_{C^{(\alpha+\vartheta)/2}( [0,T]; C^{\vartheta } )  }^p        \Big] \, \leq \,  C  \$ \theta \$_{p,0}^p$ and in particular
\begin{equation}
 \Ex \Big[   \big\| \Psi^{\theta} \big\|_{C^{\alpha/2}( [0,T]; C^0 )  }^p        \Big] \, \leq \,  C T^{\vartheta/2} \$ \theta \$_{p,0}^p. \label{eq:Hol-Bou-Psi3}
\end{equation}
\end{remark}

\begin{proof}
Fix an $0 \leq s < t \leq T$ and $x,y \in [0,1]$. For every $\beta \in (0,1/2)$ we can write using H\"older's inequality
\begin{align}
\Ex &\left[ \left(  \frac{ \big( \Psi(t,x) - \Psi(s,x) \big) -\big( \Psi(t,y) - \Psi(s,y) \big)  }{|x-y|^\beta} \right)^p   \right]    \notag\\ 
& \qquad \leq \, \Ex \Bigg[ \left( \frac{ \big| \Psi(t,x) - \Psi(t,y) \big|  }{ |x-y|^{1/2}} + \frac{ \big| \Psi(s,x) - \Psi(s,y) \big|  }{|x-y|^{1/2}} \right)^{p} \Bigg]^{2 \beta} \notag\\
& \qquad \qquad  \qquad \qquad  \Ex \Big[  \Big(  \big|  \Psi(t,x) - \Psi(s,x) \big| + \big| \Psi(t,y) - \Psi(s,y) \big|   \Big)^p   \Big]^{1-2 \beta}.    \label{eq:HolBou1}
\end{align}
To bound the first term on the right hand side of \eqref{eq:HolBou1} we write using \BDG Inequality (\cite[Theorem 3.28 p. 166]{KS91})
\begin{align}
 \Ex  &\Bigg[ \left( \frac{ \big| \Psi(t,x) - \Psi(t,y) \big|  }{ |x-y|^{1/2}}  \right)^{p} \Bigg] \notag \\
 & \leq \,  \frac{C}{|x-y|^{p/2}} \Ex \left[ \left( \int_0^ t \!\! \int_0^1 \big( \hat{p}_{t-s}(z-x) -\hat{p}_{t-s}(z-y) \big)^2 \theta(s,z)^2 \mathrm{d}s\, \mathrm{d}z   \right)^{p/2}      \right] \notag\\
 & \leq C  \$  \theta \$_{p,0}^p. \label{eq:HolBou2}
 \end{align}
 Here in the last step we have used \eqref{eq:SR1}.  A very similar calculation involving \eqref{eq:TR} shows that 
 \begin{equation}\label{eq:HolBou3}
\Ex \Big[   \big|  \Psi(t,x) - \Psi(s,x) \big|^p   \Big] \, \leq C   \$  \theta \$_{p,0}^p (t-s)^{p/4}.
\end{equation}
So combining \eqref{eq:HolBou1} - \eqref{eq:HolBou3} we get
\begin{equation}\label{eq:HolBou4}
\Ex \left[ \left(  \frac{ \big( \Psi(t,x) - \Psi(s,x) \big) -\big( \Psi(t,y) - \Psi(s,y) \big)  }{|x-y|^\beta} \right)^p   \right]  \, \leq \,  C   \$  \theta \$_{p,0}^p (t-s)^{p/4(1-2 \beta)}.
\end{equation}

Now applying Lemma \ref{lem:GRR} to the function $f= \Psi^\theta(t, \cdot) - \Psi^\theta(s, \cdot)$ one gets that for $x \neq y$ and for $p > \frac{2}{\beta}$
\begin{equation} \label{eq:HolBou5}
\big| f(x)-f(y)   \big| \leq C U^{1/p} |x-y |^{\beta -2/p},
\end{equation}
where
\begin{equation}
U\,= \,  \int_0^1 \! \! \int_0^1  \left( \frac{|f(x)-f(y)|}{|x-y|^\beta} \right)^p \mathrm{d}x \, \mathrm{d}y.
\end{equation}
Noting that \eqref{eq:HolBou4} implies that $\Ex \big[  U \big] \leq C  \$  \theta \$_{p,0}^p (t-s)^{p/4}$ one can conclude that
\begin{equation}
\Ex \left[ | \Psi(t,\cdot) - \Psi(s,\cdot)  |_{\beta-\frac{2}{p}}^p \, \right] \leq \, C  \$  \theta \$_{p,0}^p (t-s)^{p(1-2\beta)/4}.
\end{equation}

Using the identity $|\Psi(t,\cdot) - \Psi(s,\cdot)|_0 \leq C \big( |\Psi(t,0) - \Psi(s,0)|_0 + |\Psi(t,\cdot) - \Psi(s,\cdot)|_\alpha \big)$as well as \eqref{eq:HolBou3} and  \eqref{eq:HolBou4} we get the same bound for the expectation of $|\Psi(t,\cdot) - \Psi(s,\cdot)|_0$ so that finally we get

\begin{equation}
\Ex \left[ | \Psi(t,\cdot) - \Psi(s,\cdot)  |_{C^{\beta-\frac{2}{p}}}^p \, \right] \leq \, C  \$  \theta \$_{p,0}^p (t-s)^{p(1-2\beta)/4}.
\end{equation}

Now setting $\beta = \alpha  + \frac{2}{p}$, which is less than $\frac{1}{2}$ by the assumtion on $p$,  and using Lemma \ref{lem:GRR} once more in the time variable yields the desired result.
\end{proof}

%
%
%

In order to get stability of this regularity result under  smooth approximations it is useful to impose a regularity assumption on $\theta$. For any $\alpha < \frac{1}{2}$ and for $p \geq 2$ denote by 
\begin{equation}
\| \theta \|_{p,\alpha} \,= \, \Ex \bigg[ \sup_{x \neq y, s \neq t}   \frac{|\theta(t,x)- \theta(s,y)|^p}{\big(|t-s|^{ \alpha /2}+ |x-y|^{\alpha }\big)^p}   + \sup_{x,t} |\theta(t,x)|  \bigg]^{1/p}. \label{eq:def-nom} 
\end{equation} 
Also recall the definition \eqref{eq:Stoch-Conv-smooth} of $\Psi_\eps^\theta$.

Then  we have the following result:

\begin{proposition}\label{prop:approx-Stoch-Conv}
Choose $\alpha$ and $p$ as in  Proposition \ref{prop:Reg-Stoch-Conv}. Then for any 
\begin{equation}\label{eq:condgamma1}
\gamma < \alpha \Big(  1-2 \alpha - \frac{12}{p} \Big)
\end{equation}
there exists a constant $C$ such that for any $\eps \in [0,1]$
\begin{equation}
\Ex \Big[   \big\| \Psi^{\theta} - \Psi^\theta_\eps \big\|_{C( [0,T]; C^{ \alpha} )  }^p        \Big] \, \leq \, C \$ \theta \$_{p,\alpha}^p \eps^{ p \gamma}   \label{eq:Hol-Bou-Psi6}
\end{equation}
and
\begin{equation}
\Ex \Big[   \big\| \Psi^{\theta} - \Psi^\theta_\eps \big\|_{C^{\alpha/2}( [0,T]; C )  }^p        \Big] \, \leq \, C \$ \theta \$_{p,\alpha}^p \eps^{p \gamma}  . \label{eq:Hol-Bou-Psi7}
\end{equation}
\end{proposition}

%
%

\begin{remark}
Actually we only need the spatial H\"older regularity of $\theta$ in the proof. We introduce the space-time parabolic regularity condition on $\theta$ as we will  need it in the sequel.
\end{remark}

%
%
\begin{proof}
We first establish the bound
\begin{equation}\label{eq:Hol-Bou-Psi8}
\Ex \big[  \big| \Psi_\eps(t,x) - \Psi(t,x)  \big|^p \big] \leq C \eps^{\alpha p},
\end{equation}
where the constant $C$ is uniform in $x$ and $t \leq T$. 

By \BDG Inequality and the definitions of $\Psi_\eps^\theta$ and $\Psi^\theta$ we get
\begin{align}
&\Ex \big[  \big| \Psi_\eps(t,x) - \Psi(t,x)  \big|^p \big] \notag\\
& \leq \,  C \, \Ex \bigg[ \bigg(\int_0^t \! \int_0^1 \!  \bigg( \int_0^1  \! \eta_\eps(y-z)  \Big(  \hat{p}_{t-s}(x-y) \, \theta(s,y) -\hat{p}_{t-s}(x-z)\, \theta(s,z)     \Big) \d z   \bigg)^2 \! \d y \, \d s  \bigg)^{\frac{p}{2}}  \bigg] \label{eq:reg1}.
\end{align}
The inner integral on the right hand side of \eqref{eq:reg1} can be bounded by
\begin{align}
\int_0^1  \! \eta_\eps(y-z) & \Big(  \hat{p}_{t-s}(x-y) \, \theta(s,y) -\hat{p}_{t-s}(x-z)\, \theta(s,z)     \Big) \d z   \notag\\
& \leq  |\theta(t)|_{\alpha} \,  \hat{p}_{t-s}(x-y)    \int_0^1  \! \eta_\eps(y-z)   \,|y-z|^\alpha \, \d z \notag\\
 & \qquad +    | \theta(t)|_0 \int_0^1  \eta_\eps(y-z) |\hat{p}_{t-s}(x-y) -\hat{p}_{t-s}(x-z)|     \d z   \label{eq:reg2}.
\end{align}
For the first term on the right hand side of \eqref{eq:reg2}  we get
\begin{align}
  \Ex \bigg[ \bigg( \int_0^t & \! \int_0^1 \! \Big(  |\theta(t)|_{\alpha} \,  \hat{p}_{t-s}(x-y)   \int_0^1  \! \eta_\eps(y-z)   \, |y-z|^\alpha \, \d z \Big)^2 \d y\, \d s \bigg)^{\frac{p}{2}} \bigg]  \notag \\
&\leq C  \eps^{\alpha p} \, \Ex \bigg[ \bigg( \int_0^t \! \int_0^1   |\theta(t)|_{\alpha}^2 \, \,  \hat{p}_{t-s}^2(x-y)   \, \d y\, \d s \bigg)^{\frac{p}{2}} \bigg]  \notag\\
&\leq C\,  \$ \theta  \$_{p.\alpha}^p    \eps^{\alpha p}  \label{eq:reg3}.
\end{align}
For the second term in the right hand side of \eqref{eq:reg2} we devide the time integral into the integral over $s \geq t-\eps^2$ and the integral over $s \leq t -\eps^2$. For the first one we get using Young's inequality
\begin{align}
 \Ex \bigg[ \bigg( &  \int_{t-\eps^2}^t  \! \int_0^1 \! \Big(  \int_0^1  | \theta(t)|_0   \eta_\eps(y-z)  |\hat{p}_{t-s}(x-y) -\hat{p}_{t-s}(x-z)|     \d z      \Big)^2 \d y\, \d s \bigg)^{\frac{p}{2}} \bigg] \notag\\
  &\leq C\, \$ \theta \$_{p,0} \Big(  \int_{t-\eps^2}^t  \! \int_0^1   \hat{p}_{t-s}^2(x-y) \, \d y \Big)^{\frac{p}{2}}  \notag\\
  & \leq C\, \$ \theta \$_{p,0} \, \eps^{\frac{p}{2}} \label{eq:reg4}.
\end{align}
Using Young's inequality again the term involving the integral over  $s \leq t -\eps^2$ can be bounded by
\begin{align}
 \$ \theta \$_{p,0}^p \bigg(   \int^{t-\eps^2}_0  \bigg\{  \int_0^1 \int_0^1    \eta_\eps(y-z)  |\hat{p}_{t-s}(x-y) -\hat{p}_{t-s}(x-z)|^2     \d z  \,\d y\, \bigg\}  \d s \bigg)^{\frac{p}{2}}.  \label{eq:reg5}
\end{align}
As in  the proof of Lemma \ref{lem:IR} we decompose $\hat{p}_t \,=\, p_t+p_t^R$ into the Gaussian heat kernel and a smooth remainder. The term in \eqref{eq:reg5} with $\hat{p}_t$ replaced by $p^R_t$ can be bounded by $C \,  \$ \theta \$_{p,0}^p \,\eps^p$ so it is sufficient to consider \eqref{eq:reg5} with $\hat{p}_t$ replaced by $p_t$. But then we get from the scaling property of $p_t$
\begin{align}
 \$ \theta \$_{p,0}^p & \bigg(   \int^{t-\eps^2}_0  \bigg\{  \int_0^1 \int_0^1    \eta_\eps(y-z)  |p_{t-s}(x-y) -p_{t-s}(x-z)|^2     \d z  \,\d y\, \bigg\}  \d s \bigg)^{\frac{p}{2}}  \notag\\
 &\leq \$ \theta \$_{p,0}^p  \bigg(   \int^{t-\eps^2}_0    \frac{\eps^3}{(t-s)^2}     \bigg\{  \int_\R \int_\R    \eta_1(y-z)  \,  |  p_{1}' |_0 | y-z)|^2     \d z  \,\d y\, \bigg\}  \d s \bigg)^{\frac{p}{2}} \notag\\   
 &\leq C\,  \$ \theta \$_{p,0}^p  \eps^{\frac{p}{2}} \label{eq:reg6}.
\end{align}
Thus combining \eqref{eq:reg3}, \eqref{eq:reg4} and \eqref{eq:reg6} we obtain the desired bound \eqref{eq:Hol-Bou-Psi8}.  

The rest of the argument follows along the same lines as the proof of Proposition \ref{prop:Reg-Stoch-Conv}. Writing $f_\eps(t,x)= \Psi^\theta(t,s)- \Psi^\theta_\eps(t,s)$ we get as in \eqref{eq:HolBou1}:
\begin{align}
\Ex& \left[ \left(  \frac{ \big( f_\eps (t,x) - f_\eps (s,x) \big) -\big( f_\eps (t,y) - f_\eps (s,y) \big)  }{|x-y|^\beta} \right)^p   \right]    \notag\\ 
& \quad  \leq \, C  \sup_{t,\eps} \Ex \Bigg[ \left( \frac{ \big| f_\eps (t,x) - f_\eps (t,y) \big|  }{ |x-y|^{1/2}}  \right)^{p} \Bigg]^{2 \beta_1}   \sup_{x,\eps} \Ex \Big[  \Big(  \big|  f_\eps (t,x) - f_\eps (s,x) \big|   \Big)^p   \Big]^{1-2 \beta_1 -\beta_2} \notag\\
& \qquad \qquad    \sup_{x,t} \Ex \Big[  \Big(  \big|  f_\eps (t,x) \big|   \Big)^p   \Big]^{\beta_2}.    \label{eq:reg7}
\end{align}
Noting that due to Young's inequality, the bounds on the space-time regularity of $\Psi$  \eqref{eq:HolBou2} and \eqref{eq:HolBou3} also hold for $\Psi_\eps$ with a constant independent of $\eps$. Using \eqref{eq:Hol-Bou-Psi8} we can bound the right hand side of \eqref{eq:reg7} by 
\begin{equation*}
\Ex \left[ \left(  \frac{ \big( f_\eps (t,x) - f_\eps (s,x) \big) -\big( f_\eps (t,y) - f_\eps (s,y) \big)  }{|x-y|^\beta} \right)^p   \right]   \leq C \$ \theta \$_{p,\alpha}^p (t-s)^{\frac{p(1-2 \beta_1 -\beta_2)}{4}} \eps^{\frac{\beta_2p}{2}}.
\end{equation*}
Thus, as in the proof of Proposition \ref{prop:Reg-Stoch-Conv}, by applying the Garsia-Rodemich-Rumsey Lemma \ref{lem:GRR} twice we get the desired result. 
 \end{proof}

%
%
%
%
In the Gaussian case $\theta=1$ one can apply the results from \cite{FV07} to obtain a canonical rough path version of $X$. We have
\begin{proposition}\label{prop:Gau-conv-rp}
For any fixed time $t$ the stochastic convolution $X(t,x)$ viewed as a process in $x$ can be lifted canonically to an $\alpha$ rough path, which we denote by $\big( X(t), \XX (t) \big)$. Furthermore, the process $\big( X(t), \XX(t) \big)$ is almost surely continuous and  one has for all $p$ 
\begin{equation}
\Ex \Big[   \big\| \mathbf{\XX} \big\|_{\Omega C^{2\alpha}_T}^p  \Big] \,<  \, \infty.
\end{equation}
Furthermore, for $\eps >0$  the Gaussian paths $X_\eps$ are smooth and $\XX_\eps$ can be defined using \eqref{eq:it-int}. Then one has for $\eps \downarrow 0$ 
\begin{equation}\label{eq:approx-gau-rp}
\Ex \Big[   \big\| \mathbf{\XX}  - \XX_\eps \big\|_{\Omega C^{2\alpha}_T}^p  \Big] \,\to  0.
\end{equation}
Here $\big\| \mathbf{\XX} \big\|_{\Omega C^{2\alpha}_T} = \sup_{0 \leq t \leq T} \big| \mathbf{\XX}(t) \big|_{\Omega C^{2\alpha}}$. 
\end{proposition}

\begin{proof}
For a given $t \in [0,T]$ the covariance function of every component of $X(t,\cdot)$ is given by $\Ex[X(t,x) \,X(t,y) ] =: K_t(x-y) +t $ where  
\begin{equation}
K_t(x)= \sum_{k \geq 1} \frac{1}{(2 \pi k)^2} \left[ 1 - \exp(-2(2\pi k)^2 t)  )  \right] \cos\big( 2\pi k x \big). 
\end{equation}
The summand $t$ does not depend on $x$, and therefore it does not contribute to the two-dimensional 1-variation of the covariance. To treat $K_t$ note that 
\begin{equation}
K_t^1(x) \,= \, \sum_{k \geq 1} \frac{1}{(2 \pi k)^2}  \cos\big( 2\pi k x \big) = \frac{1}{2}|x| -\frac{1}{2}x^2 -\frac{1}{12},
\end{equation}
for $x \in [0,1]$ and periodically extended outside of this interval. In particular, the two-dimensional $1$-variation of the term corresponding to $K_t^1$ is given as
\begin{equation}
| K_t^1(x-y)|_{1-\text{var}} [x_1,x_2] \times [y_1,y_2] =  \int_{x_1}^{x_2} \!  \! \int_{y_1}^{y_2} \big( 1 +\delta_{x =y }  \big)\,\mathrm{d}x \, \mathrm{d}y. 
\end{equation}
Noting that the second term 
\begin{equation}
K_t^2(x)= -  \sum_{k \geq 1} \frac{1}{(2 \pi k)^2} \left[   \exp(-2(2\pi k)^2 t)  )  \right] \cos\big( 2\pi k x \big)
\end{equation}
is given by the convolution (on the torus) of $K_t^1$ with the heat kernel $\hat{p}_t$ one can easily see that the two-dimensional $1$-variation of $K$ satisfies condition \eqref{eq:2-dvar2} with a uniform constant in $t$. This shows that for every $t$ the process $X(t, \cdot)$ can be lifted to a rough path. 

To deal with the continuity note that by \eqref{eq:HolBou3}
\begin{equation}
\Ex \big[  \big( X(t,x) - X(s,x) \big)^2  \big] \, \leq \, C |t-s|^{1/4}, 
\end{equation}
such that by an interpolation argument (see e.g. \cite[Rem. 15.32]{FV10}) one gets 
\begin{equation}
|K_{X(t)-X(s)}|_{\beta-\text{var}} \leq C|t-s|^\vartheta  
\end{equation}
for any $\alpha < \beta <1$ and a small $\theta>0$. Recall the definition \eqref{eq:2-dvar1} of the two-dimensional $\rho$-variation. Thus the second part of Lemma \ref{lem:gau-rp} and an application of the Garsia-Rodemich-Rumsey Lemma \ref{lem:GRR} yield the result.  
The bound \eqref{eq:approx-gau-rp} follows in the same way noting that in the case $\theta=1$ the bound \eqref{eq:Hol-Bou-Psi8} reads
\begin{equation}
\Ex \big[\,  | X_\eps(t,x) -X(t,x)|^p  \big] \, \leq \, C \eps^{\frac{p}{2}}.
\end{equation}

\end{proof}

As the processes  $\Psi^\theta$ are not Gaussian for general  $\theta$ we cannot draw the same conclusion to define iterated integrals for $\Psi^\theta$. The crucial observation is that as soon as $\theta$ possesses a certain regularity the worst fluctuations are controlled by those of the Gaussian process $X$.
For $K >0$ and for $\alpha \in (0,1/2)$ we introduce the stopping time 
\begin{equation}\label{eq:def-tau}
\tau_{K}^{X} \,= \, \inf  \bigg\{ t \in  [0,T] \colon  \sup_{\substack{ x_1 \neq x_2 \\ 0 \leq s_1< s_2 \leq t}}  \frac{\big|X(s_1,x_1) -  X(s_2,x_2) \big|  }{|s_1-s_2|^{\alpha/2}  + |x_1-x_2|^{\alpha} }  >  K  \bigg\}
\end{equation}
Note that by Proposition  \ref{prop:Reg-Stoch-Conv} for every $\alpha < \frac{1}{2}$ one has $\tau_{K}^X  > 0$ and $\lim_{K \uparrow \infty} \tau_{K}^X =T$ almost surely. 

%
%
%
%
%
%
\begin{proposition}\label{prop:Stoch-Conv-Rough-Path}
Denote by 
\begin{gather}
R^\theta(t,x,y)=  \Psi^\theta(t,y)- \Psi^\theta(t,x) - \theta(t,x) \big(X(t,y)- X(t,x)\big). \label{eq:def-R}
\end{gather}
For every   
 \begin{equation}\label{eq:condp2}
p \, >  \,   \frac{2(6 \alpha +2)}{\alpha(1-2\alpha)} 
\end{equation} 
set 
\begin{equation}\label{eq:cond-thet2} 
\vartheta \,= \, \frac{\alpha(1-2\alpha)}{2(1+2\alpha)} -\frac{6\alpha +2}{p(1+2\alpha) } \, >0.
\end{equation}
Then one has 
\begin{align}
\Ex  \bigg[   \big\| R^\theta \big\|_{C^{\vartheta} \left( [0,\tau_K^{X}];\Omega C^{2\alpha}  \right)    }^p    \bigg] \,\leq \, C \big( 1+K^p  \big)  \| \theta \|_{p,\alpha}^{p}.     \label{eq:bou-R}
\end{align}
\end{proposition} 

\begin{remark}
Similar to above \eqref{eq:bou-R} implies in particular that 
\begin{align}
\Ex  \bigg[   \big\| R^\theta \big\|_{C^0 \left( [0,\tau_K^{X}];\Omega C^{2\alpha}  \right)    }^p    \bigg] \,\leq \, C T^\vartheta  \big( 1+K^p  \big)  \| \theta \|_{p,\alpha}^{p} .     \label{eq:bou-RT}
\end{align}

\end{remark}

%
%
%
\begin{proof}
Similar to the proof of Proposition \ref{prop:Reg-Stoch-Conv} we begin by noting that for $0 \leq s < t \leq T$ and for any $\beta \in (0,\frac{1}{2}+\alpha)$,
\begin{align}
\Ex &\left[  \left(  \frac{R^\theta(t,x,y)-R^\theta(s,x,y)}{|x-y|^{\beta} }  \right)^p       \right]\, \notag\\
&\leq\, \Ex \left[  \left(  \frac{\big|R^\theta(t,x,y)\big| - \big|R^\theta(s,x,y) \big|}{|x-y|^{\frac{1}{2}+\alpha}}  \right)^p\right]^{\frac{2\beta}{1+2\alpha}}   \, \Ex \Big[    \big| R^\theta(t,x,y) - R^\theta(s,x,y) \big|^{p }     \Big] ^{ \big(1-\frac{2\beta}{1+2\alpha}\big)}.  \label{eq:R-Bou1}
\end{align}
Using the definition \eqref{eq:def-R}of $R^\theta(t,x,y)$ one obtains
\begin{align}
R^\theta(t,x,y) \, = \, \int_0^t \!\!  \int_\R \Big(\hat{p}_{t-s}(z-y) - \hat{p}_{t-s}(z-x)  \Big) \Big( \theta(s,z) - \theta(t,x) \Big) W(\mathrm{ds},\mathrm{dz}). \label{eq:rp1}
\end{align}
%
%
So by \BDG inequality we have
\begin{align}
\Ex &\bigg[ \Big(R^\theta(t,x,y) \Big)^{p}  \bigg] \notag\\
& \leq \, C \, \Ex \bigg[ \Big( \int_0^t \! \! \int_\R \Big(\hat{p}_{t-s}(y,z) - \hat{p}_{t-s}(x,z)  \Big)^2 \Big( \theta(s,z) - \theta(t,x) \Big)^2 \, \mathrm{dz} \, \mathrm{ds} \Big)^{p/2} \bigg]. \notag\\
 &\leq \, C\, \Ex \left[\sup_{z \neq x, s \neq t}\left( \frac{\theta(s,z) - \theta(t,x)}{|z|^\alpha + |s-t|^{\alpha/2}} \right)^{p}   \right]  \notag\\
 & \qquad \qquad \qquad   \Big( \int_0^t \!\! \int_\R \Big(\hat{p}_{t-s}(z-y) -\hat{p}_{t-s}(z-x)  \Big)^2 \Big( |z-x|^\alpha + |s-t|^{\alpha/2} \Big)^2  \, \mathrm{dz} \, \mathrm{ds} \Big)^{p/2} \notag\\
&\leq \, C \| \theta\|_{p,\alpha}^p  \, |x-y|^{(1+2\alpha)p/2}.
\label{eq:rp2} 
\end{align}
Here we have used \eqref{eq:bou-I1} and \eqref{eq:bou-I2} in the last line. 

%
%

Similarly, one can write using \eqref{eq:def-R}
\begin{align}
\Ex &\Big[    \big| R^\theta(t,x,y) - R^\theta(s,x,y) \big|^{p }     \Big] \notag \\
 &\leq  \, C \, \Ex \Big[   \big| \Psi^\theta(t,y)-  \Psi^\theta(s,y) \big|^p  \Big] + C \,\Ex \Big[     \big|   \Psi^\theta(s,x)  - \Psi^\theta(t,x) \big|^p \Big]  \notag\\
  & \quad + C\, \Ex \Big[   \big|  \theta(t,x) -\theta(s,x) \big|^p   \big| X(t,y)- X(t,x)\big|^p \Big]    \notag\\ 
  & \quad   + C\, \Ex \Big[  | \theta(s,x)|^p   \big| X(t,y) - X(s,y) \big|^p \Big]  + C\, \Ex \Big[  | \theta(s,x)|^p   \big| X(s,x) -   X(t,x) \big|^p  \Big] \notag\\
  &= C \big( I_1 +I_2 + I_3 + I_4 + I_5 \big). \label{eq:R-Bou2}
\end{align}
The terms $I_1$ and $I_2$ are bounded by $C \$ \theta \$_{p,0}^p (t-s)^{p/4}$. To bound $I_3$ we use the definition of $\tau_K^{\|X\|_\alpha}$
\begin{equation}
\Ex \Big[   \big|  \theta(t,x) -\theta(s,x) \big|^p   \big|\Psi(t,y)- \Psi(t,x)\big|^p \Big] \,  \leq \, K^p  \| \theta \|_{p,\alpha}^p (t-s)^{\alpha p/2}. \label{eq:R-Bou3}
\end{equation}
To bound $I_4$ and $I_5$ we write using the definition of $\tau_K^{\|X\|_\alpha}$ again 
\begin{equation}
\, \Ex \Big[  | \theta(s,x)|^p   \big| X(t,y) - X(s,y) \big|^p \Big]  \leq K^p \$ \theta \$_{p,0}^p (t-s)^{p \alpha/2}.\label{eq:R-Bou4}
\end{equation}
Thus summarising \eqref{eq:R-Bou1} - \eqref{eq:R-Bou4} we get 
\begin{equation}\label{eq:R-Bou5}
 \Ex \left[  \left(  \frac{R^\theta(t,x,y)-R^\theta(s,x,y)}{|x-y|^{\beta} }  \right)^p       \right] \,    \leq C  \$ \theta \$_{p,\alpha}^p \big(  K^p  |t-s|^{p\alpha/2}\big)^{ \big(1-\frac{2\beta}{1+2\alpha}\big)}.
\end{equation}

%
%

To be able to apply Lemma \ref{lem:GRR} we need similar bounds for 
\begin{equation}
NR^\theta(t,x,y,z)\,=\,  \big( \theta(t,y)-\theta(t,x) \big) \big(X(t,z)-X(t,y) \big). 
\end{equation}
Recall that the operator $N$ was defined in \eqref{eq:def-N}. Similar to the calculation in \eqref{eq:R-Bou1} observe that
\begin{align}
\Ex  \Bigg[&  \sup_{\substack{  0 \leq x \leq u <\\v<r \leq  y \leq 1}}  \frac{\big| NR^\theta(t,u,v,r) -NR^\theta(s,u,v,r)  \big|^p}{|x-y|^{\beta p}}  \Bigg] \notag\\ 
& \qquad \leq \, C\, \Ex \bigg[ \sup \, \frac{\big| NR^\theta(t,u,v,r)\big|^p + \big| NR^\theta(s,u,v,r)  \big|^p}{|x-y|^{2 \alpha p}}  \bigg]^{\frac{\beta}{2\alpha}} \notag\\
& \qquad \qquad \qquad \Ex \bigg[ \sup \,\big|NR^\theta(t,u,v,r) -  NR^\theta(s,u,v,r)  \big|^p  \bigg]^{1-\frac{ \beta}{2\alpha}}.
\end{align}
To bound the first expectation we calculate using the definition of $\tau_K^{X}$
\begin{align}
\Ex \bigg[ \sup \, \frac{\big|NR^\theta(t,u,v,r)  \big|^p}{|x-y|^{\alpha p}}  \bigg] \, \leq \, \Ex \big[ \big| \theta(t,\cdot) \big|_{\alpha}^p  \big| X(t, \cdot)  \big|_{\alpha}^p     \big] \, \leq \, C K^p \$ \theta\$_{p,\alpha}^p . 
\end{align}
As in   \eqref{eq:R-Bou2} - \eqref{eq:R-Bou4} one can see that
\begin{equation}
\Ex \bigg[ \sup \big| NR^\theta(t,u,v,r) -  NR^\theta(s,u,v,r)  \big|^p  \bigg] \, \leq \, C \$ \theta \$_{p,\alpha}^p |t-s|^{p \alpha/2}.
\end{equation}

%
%
%

Thus applying Lemma \ref{lem:GRR} one gets
\begin{equation}
\Ex\Big[ \big|R^\theta(t,\cdot,\cdot) - R^\theta(s,\cdot,\cdot) \big|_{\beta-\frac{2}{p}}  \Big] \, \leq  C (1+K_1^p) \$ \theta \$_{p,\alpha}^p |t-s|^{p \alpha/2(1-\frac{2\beta}{1+2\alpha})}.
\end{equation}
Applying Lemma  \ref{lem:GRR} once more in the time direction and setting $\beta=2\alpha +\frac{2}{p}$ one obtains the desired bound \eqref{eq:bou-R}.
\end{proof}


To finish this section, we show that the bound \eqref{eq:bou-R} is stable under approximations as well. To this end denote by 
\begin{equation}\label{eq:def-tau-eps}
\tau^{\| X\|_\alpha, \eps}_K = \inf  \bigg\{ t  \colon  \sup_{\substack{ x_1 \neq x_2 \\ 0 \leq s_1< s_2 \leq t}}  \frac{|( X_\eps(s_1,x_1)\big)  -  X_\eps(s_2,x_2)   \big|  }{|s_1-s_2|^{\alpha/2}  + |x_1-x_2|^{\alpha} }  >  K  \bigg\}.
\end{equation}
Furthermore, we define the $\eps$-remainder as 
\begin{gather}
R_\eps^\theta(t,x,y)=  \Psi_\eps^\theta(t,y)- \Psi_\eps^\theta(t,x) - \theta(t,x) \big(X_\eps(t,y)- X_\eps(t,x)\big). \label{eq:def-Reps}
\end{gather}

\begin{proposition}\label{prop:regrem}
Assume that  $|\theta |_0$ is a deterministic constant. Then for any $p$ as in \eqref{eq:condp2} and for any 
\begin{equation}\label{eq:cond-gamma2}
\gamma <  \frac{\alpha(1-2 \alpha)}{1 + 2 \alpha} - \frac{12 \alpha +4}{p (1+2 \alpha)}
\end{equation}
there exists a constant $C$ such that
\begin{equation}\label{eq:Stab-Reps}
\Ex  \bigg[   \big\| R^\theta  - R^\theta_\eps \big\|_{C^0 \left( [0,T \wedge \tau_K^{\| X\|_\alpha} \wedge \tau_K^{\| X\|_\alpha,\eps}];\Omega C^{2\alpha}  \right)    }^p    \bigg] \,\leq \, C  \big( 1+K^p  \big)  \| \theta \|_{p,\alpha}^{p} \eps^\gamma.     
\end{equation}
\end{proposition}

%
%

\begin{proof}
Due to the deterministic bound on $|\theta|_0$ we get using \eqref{eq:Hol-Bou-Psi8} as well as the definition of $R_\eps^\theta$ that
\begin{equation}
\Ex  \big[   \big| R^\theta(t,x,y)  - R^\theta_\eps(t,x,y) \big| \big] \, \leq \, C \$ \theta \$_{p,\alpha}^p \eps^{\alpha p}.
\end{equation}
Then in the same way as in \eqref{eq:reg7} we get for $h_\eps= R_\eps^\theta - R^\theta$ that 
\begin{equation}
 \Ex \left[  \left(  \frac{h_\eps(t,x,y)-h_\eps(s,x,y)}{|x-y|^{\beta_1} }  \right)^p       \right] \,    \leq C  \$ \theta \$_{p,\alpha}^p  (1 + K^p)  \big(   |t-s|^{p\alpha/2}\big)^{ \big(1-\frac{2\beta}{1+2\alpha}\big) - \beta_2} \eps^{\alpha \beta_2}.
\end{equation}
Here we have used again that by Young's inequality the bounds on the space-time regularity of $R$ hold for $R_\eps$ as well with constants that are uniform in $\eps$. The bounds on $N h_\eps$ are derived in the same way as in the proof of \eqref{prop:Stoch-Conv-Rough-Path}. Then the proof is again finished by the Garsia-Rodemich-Rumsey Lemma \ref{lem:GRR}.
\end{proof}

\section{Existence of Solutions}
\label{sec:Solutions}

In this section we prove Theorem \ref{thm:MR} and Theorem \ref{thm:Stability}. As the argument is quite long and technical we first give an outline of the proof.
 
We will construct mild solutions to the equation \eqref{eq:Stoch-Burg} i.e. solutions of the equation
\begin{equation}\label{eq:Burg-mild}
u(t) \,= \, S(t) \, u_0 + \int_0^t S(t-s) \, \theta(u(s)) \, \d W(s) + \int_0^t \! S(t-s)  \, g\big(u(s) \big) \, \partial_x u(s) \,\mathrm{d}s,  
\end{equation}
where as above $S(t)$ denotes the heat semigroup which is given by convolution  with the heat kernel $\hat{p}_t(x)$.  Note that the two integral terms in \eqref{eq:Burg-mild} are of a very different nature. The first one $\int_0^t S(t-s) \,  \theta(u(s))\, \mathrm{d}W(s)$ is a stochastic integral in time whereas the second term is a usual Lebesgue integral in time and a rough integral against the heat kernel in space:
\begin{equation}
S(t-s) \, g(u(s)) \partial_x u(s)(x)\,=\, \int_0^1 \! \hat{p}_{t-s}(x-y) \, g(u(s,y))\, \mathrm{d}_y u(s,y).
\end{equation}
These  terms have different natural spaces for solving a fixed point argument. It thus seems advisable to separate the construction into two parts. 

In the additive noise case \cite{Ha10} $\theta=1$ this can be done using the following trick. As it has been observed in several similar cases (see e.g. \cite{DPD03}) subtracting the solution $X$ to the linearised equation \eqref{eq:SHE} regularises the solution. Actually, we expect $v\,=\, u-X -U$ to be a $C^1$ function of the space variable $x$. Here for technical convenience we have also removed the term involving the initial condition by subtracting the solution $U(t)= S(t)u_0$ to the linear heat equation with the given initial data. Then for fixed $\big( X,\mathbf{X}\big)$ one can obtain $v$ by solving the deterministic fixed point problem
\begin{align}
v(t)\,=\, &\int_0^t S(t-s) \,  \Big[ g\big( v(s)+X(s) +U(s)  \big) \partial_x \big( v(s) + U(s) \big) \Big] \, \mathrm{d}s \notag\\
&\qquad \qquad + \int_0^t  \int_0^1 \hat{p}_{t-s}(\cdot - y) g\big( v(s,y)+ X(s,y) + U(s) \big)\,  \mathrm{d}_y X(s,y)  \, \mathrm{d}s  \label{eq:FPv}
\end{align}
in $C\big([0,T], C^1[0,1]\big)$. Then by adding $X$ and $U$ to the $v$ one can recover the solution $u$. The crucial ingredient for this fixed point problem is the regularisation of the heat semigroup. Lemma \ref{lem:scaling} gives a way to express this property in the rough path context.

 In the multiplicative noise case some extra care is necessary. For a general adapted process $\theta$, the stochastic convolution $\Psi^\theta$ is not Gaussian and thus Lemma \ref{lem:gau-rp} can not be applied to get a canonical rough path lift.  But we have seen in Proposition \ref{prop:Stoch-Conv-Rough-Path} that it can be viewed as an $X$ controlled rough path.

Given a fixed controlled rough path valued process $ \Psi$ we can again perform a pathwise construction to obtain $v^{\Psi}$ for this particular rough path valued function in the same manner  as before by solving the fixed point equation \eqref{eq:FPv}  with $X$ replaced by $\Psi$. Furthermore, $v^{\Psi}$ depends continuously on $\Psi$ (and on $(X,\mathbf{X})$). To be more precise, we will prove in Proposition \ref{prop:cont-dep-psi} that the  $C([0,T],C^1)$ and the $C^{1/2}([0,T],C)$ norms depend continuously on the controlled rough path norm of $\Psi$.

 This can then finally be used to construct $u$ as a solution to another fixed point problem in a space of adapted stochastic processes possessing the right space-time regularity. We solve the fixed point problem 
 \begin{equation*}
u \mapsto U + v^{\theta(u)} + \Psi^{\theta(u)}
\end{equation*}
with finite norm $\$u\$_{p,\alpha}$ (recall that the space-time H\"older norm $\$ \cdot \$_{p,\alpha}$ was defined in \eqref{eq:def-nom}). Here it is crucial to assume regularity for the process $u$ as the bounds that control the rough path norms of   $\Psi^{\theta(u)}$ depend on the regularity of $\theta(u)$.  This is where we need that space-time norms of $v$ depend continuously on the rough path norms of $\Psi^\theta$.

In order to perform this fixed point argument we need to introduce several cutoffs. These will then finally be removed in the last step of the proof by deriving some suitable a priori bounds. 

The entire construction is continuous in all the data. Using this and  the stability results derived in Section  \ref{sec:Prel-Calc} we finally prove Theorem \ref{thm:Stability}.

We start by analysing the  deterministic ``inner" fixed point argument. Fix functions $ X(t,x) \in C^\alpha_T\,= \, C([0,T], C^\alpha) $ and  $\mathbf{X}(t,x,y) \big) \in \Omega C^{2 \alpha}_T \, = \, C([0,T], \Omega C^{2 \alpha})$. We assume that for every fixed time $t$ the  pair $\big(X(t,\cdot), \mathbf{X}(t, \cdot) \big)$ is a rough path in space i.e. we assume that \eqref{eq:cond-it-int} holds. This reference rough path will remain fixed throughout the first (deterministic) part of the construction.  We will use the notation
\begin{equation}\label{eq:def -st-rpn}
\big\| (X, \mathbf{X})\big \|_{\mathcal{D}^{\alpha}_T} \,= \, \sup_{0 \leq t \leq T} \big| \big(X(t, \cdot), \mathbf{X}(t, \cdot)\big) \big|_{\mathcal{D}^{\alpha}}.  
\end{equation}
Note that for the moment we will pretend that $(X,\XX)$ are deterministic. Later on, of course, $X$ will be the Gaussian rough path constructed in Proposition \ref{prop:Gau-conv-rp} and the estimates will only hold true outside a universal set of  zero measure.

Furthermore, we fix a function $\Psi \in \CT$. We assume that for every $t$ the function $\Psi(t, \cdot)$ is an $X$-controlled rough path i.e. that there are functions ${\Psi}' \in \CT$ and $R_\Psi \in \Omega C_T^{2 \alpha}$  such that for every $t$ \eqref{eq:contr-rp} holds. We write 
\begin{equation}\label{eq:def-sp-crpn}
\| \Psi   \|_{\mathcal{C}_{X,T}^\alpha} \,= \,  \sup_{0 \leq t \leq T} | \Psi(t, \cdot) |_{\mathcal{C}_{X(t)}^\alpha}. 
\end{equation}

Finally, we fix the initial data $u_0 \in C^\beta$ for $\frac{1}{3} < \alpha < \beta < \frac{1}{2}$. For $t >0$ write $U(t,\cdot)= S(t) u_0$. Note that by standard regularisation properties of the heat semigroup 
\begin{equation}\label{eq:reg-Ut}
|U(t,\cdot) |_{C^{2\alpha}} \,\leq \,  C t^{\frac{\beta-2\alpha}{2}}|u_0|_{\beta} \qquad \text{and} \qquad |U(t,\cdot) |_{C^1} \,\leq \,  C t^{\frac{\beta-1}{2}}|u_0|_{\beta}.
\end{equation} 
For $v \in C^1_T= C([0,T],C^1)$ define the operator $G_{T}= G_{T, \Psi,u_0}$ as
\begin{align}
G_{T}(v)(t,x) \,=\,&  \int_0^t S(t-s)  \Big[  g\big(u(s) \big)  \partial_x \big( v(s) + U(s) \big) \Big]  \, \d s \,(x) \notag\\
&\qquad \qquad +\int_0^t \int_0^{1} p_{t-s} (x-y) g(u(y,s)) \, \mathrm{d}_y \Psi(y,s)     \, \d s. \label{eq:Def-G}
\end{align} 
Here to abreviate the notation we write $u(t,x) = U(t,x) + \Psi(t,x) +v(t,x)$. The spatial integral in the second line of \eqref{eq:Def-G} is to be interpreted as a rough integral: For every $s$ the path $\Psi(s.\cdot)$ is in $\mathcal{C}_{X(s)}^\alpha$ and so is $g\big(u(s)\big) \,=\, g \big(v(s)+U(s)+\Psi(s) \big)$ (see Lemma \ref{lem:comp-reg-func}). Thus the integral can be defined as in Lemma \ref{lem:Gub-int}. Note that as the spatial integral is the limit of the approximations \eqref{eq:Riem-sum2} they are in particular measurable in $s$ and the temporal integral can be defined as a usual Lebesgue integral.

%
%
%

The next lemma is a modification of Proposition 2.5 and Lemma 3.9 in \cite{Ha10}. It is the crucial ingredient in proving the regularising properties of the convolution with the heat kernel if understood in the rough path sense.

\begin{lemma}\label{lem:scaling}
Let $(X,\XX)$  be an $\alpha$ rough path and $Y,Z  \in \mathcal{C}_X^{\alpha}$. Furthermore, assume that $f:\R \to \R$ is a $C^1$ function such that 
\begin{equation}
|f|_{1,1} \,=\, \sum_{k \in \Z}  \sup_{x \in [k, k+1]} |f(x)| + |f'(x)| 
\end{equation} 
is finite. Then for any $\lambda_0 >0$ there exists a $C>0$ such that the following bound holds for every $\lambda \geq \lambda_0$ 
\begin{align}
\Big|  \int_0^1 f(\lambda x) & Y(x) \, \mathrm{d}Z(x) \Big|  \leq\, C |f|_{1,1} \lambda^{-\alpha} 
 \bigg[  |Y|_{0} |Z|_\alpha  
 + |R_Y|_{2\alpha} |Z|_\alpha   
  +  |X|_\alpha |Y'|_0 |R_Z|_{2\alpha} \notag\\
&   +  |\XX|_{2\alpha} \Big(    |Y'|_{0} |Z'|_{\alpha} +  |Y'|_{C^\alpha} |Z'|_{0} \Big)    + |X|_{\alpha}^2 |Y'|_0 |Z|_\alpha   \bigg] .\label{eq:sca1}
\end{align}
\end{lemma}
%
%
\begin{remark}
In most of the sequel we will not need the detailed version \eqref{eq:sca1} but instead work with the simplified bound
\begin{equation}
\Big|\int_0^1 f(\lambda x) Y(x) \, \mathrm{d}Z(x) \Big|  \leq C|f|_{1,1} \lambda^{-\alpha}\big(1+ |(X,\mathbf{X})|_{\mathcal{D}^\alpha}\big) 
 |Y|_{C^{\alpha}_X} |Z|_{C^{\alpha}_X} .\label{eq:sca3}
\end{equation}
Note however, that to derive the a priori bounds to prove global existence we will need to use that in \eqref{eq:sca1} we do not have any term involving the product  $\big| R_Y \big|_{2\alpha} \big| R_Z \big|_{2 \alpha}$ or the product $|Y|_\alpha |Z|_{\alpha}$ . 
\end{remark}
%
%

%

\begin{remark} \label{rem:sca3}
Using the bound \eqref{eq:IB4} in the proof one can see that the same scaling behaviour holds true for the difference of integrals with respect to different rough paths. Assume that  $(\tilde{X},\widetilde{\mathbb{X}})$  is  another $\alpha$ rough path and $\tilde{Y},\tilde{Z}  \in \mathcal{C}_{\tilde{X}}^{\alpha}$. Then one has
\begin{align}
\Big|\int_0^1 f(\lambda x) & Y(x) \, \mathrm{d}Z(x)   -  \int_0^1 f(\lambda x) \tilde{Y}(x) \, \mathrm{d}\tilde{Z}(x)    \Big|  \notag \\
&\leq C|f|_{1,1} \lambda^{-\alpha} \bigg[ \big( |Y|_{\CaX} + |\tilde{Y}|_{\mathcal{C}^\alpha_{\tilde{X}}}     \big)  \big( |Z|_{\CaX} + |\tilde{Z}|_{\mathcal{C}^\alpha_{\tilde{X}}}    \big) \big( |X-\tilde{X} |_{C^\alpha} +  |\XX-\tilde{\XX} |_{2 \alpha} \big)     \notag\\
 & + C  \big( |Z|_{\CaX} + |\tilde{Z}|_{\mathcal{C}^\alpha_{\bar{X}}}    \big)  \big( 1+ \big| \big( X,\XX \big) |_{\Da} + \big| \big(\tilde{ X},\tilde{\XX} \big) |_{\Da}  \big) \notag\\
 & \qquad \qquad  \times \big( |Y- \tilde{Y}|_{C^\alpha} +   |Y'- \tilde{Y'}|_{C^\alpha}+  |R_Y - R_{\tilde{Y}}|_{2 \alpha}   \big) \notag\\
 &+ C   \big( 1+ \big| \big( X,\XX \big) |_{\Da} + \big| \big(\tilde{ X},\tilde{\XX} \big) |_{\Da}  \big)  \big( |Y|_{\CaX} + |\tilde{Y}|_{\mathcal{C}^\alpha_{\tilde{X}}}    \big) \notag\\
 & \qquad \qquad  \times \big( |Z- \tilde{Z}|_{C^\alpha} +   |Z'- \tilde{Z'}|_{C^\alpha}+  |R_Z - R_{\tilde{Z}}|_{2 \alpha}   \big) \bigg]. \label{eq:sca-bou-drp}
\end{align}
\end{remark}

%
%
%

\begin{remark} \label{rem:sca2}
We will only apply Lemma \ref{lem:scaling} in the case when $f(x)$ is the heat kernel $\hat{p}_t$ or its derivative. Actually using the expression \eqref{eq:RePr} it is easy to see that for every $t \in [0,1]$ there are functions $f_t$ and $g_t$ such that $\sup_{t \in [0,1]} |f_t|_{1,1} + |g_t|_{1,1} < \infty$ and such that 
$\hat{p}_t(x)= \frac{1}{\sqrt{t}}f_t \big( x / \sqrt{t}\big)$ and  $\partial_x \hat{p}_t(x)= \frac{1}{t} g_t\big( x / \sqrt{t}\big)$ . Then applying \eqref{eq:sca3} for $\lambda = (t-s)^{- 1/2}$ one gets   
\begin{gather}
\Big| \int_0^{1} \partial_x \hat{p}_t (x,y) Z(x) \mathrm{d} X(x) \Big| \, \leq \, C t^{\alpha/2 -1}   \big(1 + |(X,\mathbf{X})|_{\mathcal{D}^\alpha} \big)  | Y |_{\mathcal{C}^\alpha_X}  | Z |_{\mathcal{C}^\alpha_X} . \label{eq:hk-bou2}
\end{gather}
and a similar bound for $\hat{p}$ instead  of  $\partial_x \hat{p}_t$ with scaling $t^{\alpha/2-1}$ instead of $t^{\alpha/2-1/2}$ . 
\end{remark}
%
%
\begin{proof}(of Lemma \ref{lem:scaling})
Without loss of generality we can assume $\lambda \in \N$. Then we can write 
\begin{equation}\label{eq:sca-bou1}
\int_0^1 f(\lambda x) \, Y(x) \, \mathrm{d}Z(x) \,= \, \sum_{k=1}^\lambda \int_0^1 f( x+k) \, Y_{\lambda,k}(x)\, \mathrm{d}Z_{\lambda,k}(x),
\end{equation}
where $Y_{\lambda.k}(x) = Y\big((x+k)/\lambda \big)$ and similarly $Z_{\lambda.k}(x) = Z\big((x+k)/\lambda \big)$. The integrals on the right hand side have to be understood as rough integrals with respcect to the reference rough path $X_{\lambda,k}(x)=X\big((x-k)/\lambda\big)$ and $\mathbf{X}_{\lambda,k}(x,y) = \mathbf{X}_{\lambda,k}\big( (x-k)/\lambda, (y-k)/\lambda \big)$. Then according to \eqref{eq:IB1} and \eqref{eq:IB2}   the integrals on the right hand side of \eqref{eq:sca-bou1} are given as
\begin{align}
&\int_0^1 f( x+k) \, Y_{\lambda,k}(x)\, \mathrm{d}Z_{\lambda,k}(x)\notag \\
 &\, = \, f(k) Y \bigg( \frac{k}{\lambda} \bigg) \delta Z \bigg( \frac{k}{\lambda},\frac{k+1}{\lambda} \bigg) + f(k)Y' \bigg( \frac{k}{\lambda} \bigg)   \mathbf{X} \bigg( \frac{k}{\lambda},\frac{k+1}{\lambda} \bigg)Z'\bigg( \frac{k}{\lambda} \bigg)^T  + Q_{\lambda,k}\label{eq:sca-bou2},
\end{align} 
where 
\begin{align}
|Q_{\lambda,k}|   \, \leq \, C \alpha_k \lambda^{-3\alpha}  \bigg[ |Y|_0 & |Z|_\alpha + |R_Y|_{2\alpha} |Z|_\alpha  +  |X|_\alpha |Y'|_0 |R_Z|_{2\alpha}  + \notag\\
& |\XX|_{2\alpha}\big(  |Y'|_{0} |Z'|_{\alpha }  +    |Y'|_{C^\alpha} |Z'|_{0 }    \big)   + |X|_{\alpha}^2 |Y'|_0 |Z|_\alpha   \bigg] \label{eq:sca-bou3}. 
\end{align}
Here we have set $\alpha_k = \sup_{x \in [k, k+1]} |f(x)| + |f'(x)| $. The first two terms on the right hand side of \eqref{eq:sca-bou2} can be bounded by
\begin{equation}\label{eq:sca-bou4}
\alpha_k \Big(\lambda^{-\alpha} |Y|_{0} |Z|_\alpha + \lambda^{-2\alpha} |Y'|_0  |Z'|_0   |\mathbf{X}|_{2 \alpha}  \Big).
\end{equation}
Thus summing over $k$ and recalling we obtain \eqref{eq:sca1}. 
\end{proof}

%
%
%
%
We need the following property:

\begin{lemma}\label{lem:comp-reg-func}
Let $(X,\mathbf{X}) \in \mathcal{D}^\alpha$ be a rough path and $\Psi \in \mathcal{C}^\alpha_X$. Furthermore, assume that $w$ is a $C^{2 \alpha}$ path and $g$ is a $C^3$ function with bounded derivatives up to order $3$. Then 
\begin{equation}
Y=g(\Psi +w)
\end{equation}
is a controlled rough path with derivative process $Y' = Dg(\Psi +w) \Psi'$. Furthermore, we have the following bounds
\begin{align}
|Y|_{C^\alpha} \, \leq\,& |g|_0 + |Dg|_0 \big( |\Psi|_\alpha + |w|_\alpha  \big)
\notag\\
|Y'|_{C^\alpha} \, \leq\,& |Dg|_0 |\Psi'|_{C^\alpha} + |D^2 g|_0  | \Psi|_{0} \big( | \Psi|_{\alpha} + |w|_{\alpha}  \big)\notag\\
|R_Y|_{2\alpha} \, \leq \, &  |Dg|_0 |w|_{2\alpha} + |D^2 g|_0 |\Psi|_{\alpha}^2 + |Dg|_0 |R_{\Psi} |_{2\alpha}. \label{eq:comp-reg-func1}
\end{align}
Furthermore let  $(\bar{X}, \bar{\XX})$ be another rough path, $\bar{\Psi} \in \mathcal{C}_{\bar{X}}^\alpha$,  and $\bar{w} \in C^{2\alpha}$ and write $\bar{Y} = g(\bar{\Psi} +\bar{w})$. Then we have the following bounds
\begin{align}
|Y-\bar{Y}|_{C^\alpha} \, \leq \, &   |g|_{C^2} \Big(1+ |\Psi |_\alpha 
 + |\bar{\Psi} |_\alpha  + |w|_{\alpha}  + |\bar{w}|_\alpha  \Big) \Big( |\Psi - \bar{\Psi}|_{C^\alpha} + |w - \bar{w}|_{C^\alpha}  \Big) \notag\\
|Y'-\bar{Y}'|_{C^\alpha} \, \leq \, &C |g|_{C^3} \Big(1+ |\Psi |_{C^\alpha} 
 + |\bar{\Psi} |_{C^\alpha}  + |w|_{C^\alpha}  + |\bar{w}|_{C^\alpha}  \Big) \notag\\
 & \quad   \Big[ \big( |\Psi'- \bar{\Psi}|_{C^\alpha}  \big) + \big( |\Psi'|_{C^\alpha} + |\bar{\Psi}|_{C^\alpha} \big) \Big( |\Psi - \bar{\Psi}|_{C^\alpha} + |w - \bar{w}|_{C^\alpha}  \Big)  \Big]    \notag\\
|R_Y - R_{\bar{Y}}|_{2 \alpha} \, \leq \, & C |g|_{C^3} \Big[ |w - w|_{2\alpha} + \big| R_{\Psi}- R_{\bar{\Psi}}\big|_{2 \alpha}   \label{eq:comp-reg-func2a}
\\
&  + \Big( 1 + |\Psi|_{C^\alpha}^2 + |\bar{\Psi}|_{C^\alpha}^2    + |w|_{C^{2\alpha}} +  |\bar{w}|_{C^{2\alpha}  }  \Big) \, \big( |\Psi - \bar{\Psi}|_{C^\alpha} +  |w - \bar{w}|_{C^\alpha}  \big)    \Big]. \notag
\end{align}

\end{lemma}

\begin{remark}
It will sometimes be convenient to work with the following simplified version of \eqref{eq:comp-reg-func1}
\begin{equation}\label{eq:comp-reg-func3}
|Y|_{\mathcal{C}_X^\alpha} \, \leq \,  C |g|_{C^2} \big( 1+ |\Psi|_{\mathcal{C}^\alpha_X}^2 \big) \big(1 + |w|_{2\alpha}      \big).
\end{equation}
In the first fixed point argument we will use the following simplified version of  \eqref{eq:comp-reg-func2a} for the case when $X=\bar{X}$ and $\Psi = \bar{\Psi}$:
\begin{align}
|Y-\bar{Y}|_{\mathcal{C}_X^\alpha} \, \leq \,  C |g|_{C^3} \Big(1 +  |\Psi|_{\mathcal{C}_X^\alpha }^2   \Big)   \Big( 1 + |w|_{C^{2 \alpha}}  \Big)  |w-\bar{w}|_{C^{2 \alpha}} . \label{eq:comp-reg-func2}
\end{align}
\end{remark}

%
%

\begin{proof}
We can write
\begin{align}
\delta Y(x,y)\,=\,& \int_0^1 Dg \big( \Psi(x)+ w(x) +\lambda \, \delta \Psi(x,y)  \big)   \delta \Psi(x,y) \,  \mathrm{d}\lambda \notag \\
& \quad +  \int_0^1 Dg \big( \Psi(y)+ w(x) +\lambda \, \delta w(x,y)  \big)   \delta w(x,y) \,  \mathrm{d}\lambda. \label{eq:reg-fun1}
\end{align}
The second integral can be bounded by $|Dg|_0 |w|_{2\alpha} |x-y|^{2\alpha}$. The first integral in \eqref{eq:reg-fun1} can be rewritten as
\begin{align}
\int_0^1 \Big(  &Dg \big( \Psi(x)+ w(x) +\lambda \, \delta \Psi(x,y)  \big) - Dg \big( \Psi(x)+ w(x)   \big) \Big)     \delta \Psi(x,y)  \,  \mathrm{d}\lambda \notag\\
&+ Dg \big( \Psi(x)+ w(x)   \big)      \delta \Psi(x,y).  \label{eq:reg-fun2}
\end{align}
The integral in \eqref{eq:reg-fun2} can be bounded by $| D^2 g|_0 | \Psi|_{\alpha}^2 |x-y|^{2\alpha}$. Rewriting the term in the second line as 
\begin{equation}  \label{eq:reg-fun3q}
Dg \big( \Psi+ w   \big)      \delta \Psi \, = \, Dg \big( \Psi+ w   \big)    \big( \Psi' \delta X + R_\Psi \big) 
\end{equation}
finishes the proof of \eqref{eq:comp-reg-func1}.

Let us now derive \eqref{eq:comp-reg-func2a}. To get a bound on the $\alpha$ -H\"older semimorm of $Y-\bar{Y}$ write
\begin{align}
\delta & Y(x,y) - \delta \bar{Y}(x,y) \notag \\
 & = \, \delta \int_0^1 Dg \Big( \lambda( \Psi + w) + (1 - \lambda) (\bar{\Psi} + \bar{w}) \Big)  \Big( \Psi - \bar{\Psi}  + w- \bar{w} \Big) \d \lambda \,  (x,y) \notag\\
 & \leq \,  |x-y|^\alpha  \int_0^1 \Big| Dg \Big( \lambda( \Psi + w) + (1 - \lambda) (\bar{\Psi} + \bar{w})  \Big) \Big|_\alpha  \big( \big| \Psi - \bar{\Psi}\big|_0   + \big| w- \bar{w} \big|_0  \big) \notag\\
 & \qquad \qquad  \big| Dg \big( \lambda( \Psi + w) + (1 - \lambda) (\bar{\Psi} + \bar{w})  \big) \big|_0  \big(\big| \Psi - \bar{\Psi} \Big|_\alpha  + \big| w- \bar{w} \big|_\alpha  \big)  \d \lambda \, .  \label{eq:reg-fun4}
\end{align}
This together with the  bound $|Y-\bar{Y}| \leq |g|_{C^1} \big| \Psi - \bar{\Psi}\big|_0  + \big|w- \bar{w} \big|_0$ yields the first bound in \eqref{eq:comp-reg-func2a}. Note in particular, that $\Psi$ and $W$ appear quadratically in this estimate.The bound on 
\begin{equation}
Y'-\bar{Y}' \, = \, Dg(w+\Psi) \Psi' -  Dg(\bar{w}+\bar{\Psi}) \bar{\Psi}' 
\end{equation}
is obtained in the same way. To treat the remainder $R^Y$ observe that \eqref{eq:reg-fun1} - \eqref{eq:reg-fun3q} show that 
\begin{align}
R^Y(x,y) \, = \, & \Big(  g\big(\Psi(y) + w(y)\big) -   g\big(\Psi(y) + w(x)\big) \Big) + Dg\big(\Psi(x) + w(x)\big) R_\Psi(x,y) \notag\\
&  \int_0^1 Dg \big( \lambda ( \Psi(x) + w(x)) +(1- \lambda )( \Psi(y) + w(y) \big) \, \d \lambda \, \delta \Psi(x,y), 
\end{align}
and similarly for $R^{\bar{Y}}$. The bounds on the individual terms of $R^Y(x,y) - R^{\bar{Y}}(x,y)$ are obtained in an elementary way using similar integration as in \eqref{eq:reg-fun4} and we therefore omit the details.
\end{proof}

%
%
%
%

Finally, we will need a  version of Gronwall's Lemma (See \cite[Lemma 7.6]{NR02} for a similar calculation):

\begin{lemma}\label{lem:Gronwall}
Fix $\alpha, \beta \geq 0$ with $\alpha + \beta <1$ and $a,b >0$. Let $x:[0, \infty) \to [0, \infty)$ be continuous and suppose that for every $t \geq 0$
\begin{equation}\label{eq:cond-Gron}
x(t) \leq a + b \int_0^t (t-s)^{\alpha} s^{-\beta} x(s) \d s.
\end{equation}    
Then there exists a constant $c_{\alpha, \beta}$ that only depends on $\alpha, \beta$ such that for all $t$ 
\begin{equation}\label{eq:res-Gron}
x(t) \leq  a \, c_{\alpha, \beta}  \exp\Big( c_{\alpha, \beta}  \,  b^{\frac{1}{1-\beta-\alpha}} \, t \Big).
\end{equation}
\end{lemma}

%
%

\begin{proof}
Define recursively the sequences 
\begin{align}
A_0(t) \, =\,& 1 \qquad \text{and} \qquad R_0(t) \, = \, x(t), \notag\\
A_{n+1}(t) \, = \,&  b \int^t_0  (t-s)^{\alpha} s^{-\beta} A_n(s) \, \d s  \qquad R_{n+1}(t) \, = \, b \int^t_0  (t-s)^{\alpha} s^{-\beta} R_n(s) \,  \d s. 
\end{align}
Then one can show by induction that
\begin{align}
A_n(t) \, = \,  \big( b\, t^{1 - \beta - \alpha} \,  \Gamma(1-\alpha) \big)^{n} \prod_{i=1}^n \frac{\Gamma \big(i - (i-1)\alpha - i \beta  \big)}{ \Gamma \big(i+1 - i\alpha - i \beta  \big)}. \label{eq:calc-bet}
\end{align}

One can furthermore show by another induction using \eqref{eq:cond-Gron} that for any $N$
\begin{align}
x_t \leq a \sum_{n=0}^{N} A_n(t) +R_{N+1}(t)
\end{align}
and that for all $0 \leq t \leq T$
\begin{align}
R_n(t) \leq   \Big( \sup_{0 \leq s \leq T} x_s \Big) \, A_n(t).
\end{align}
In order to see \eqref{eq:res-Gron} it thus remains only to bound the sum over the $A_n$. To this end note that the product in \eqref{eq:calc-bet} is almost telescoping as due to the monotonicity of the Gamma function on $[1,\infty)$ for $i \geq N_0 = \left\lceil  \frac{\beta}{1 - \alpha - \beta} \right\rceil $
\begin{equation}
\frac{\Gamma\big( i+1-i \alpha - i \beta  - \beta \big)  }{\Gamma\big( i+1 -i\alpha - i\beta \big) } \leq 1.
\end{equation}  
For $i < N_0$ we bound the same quotient by  
\begin{equation}
\frac{\Gamma\big( i+1-i \alpha - i \beta  - \beta \big)  }{\Gamma\big( i+1 -i\alpha - i\beta \big) } \leq c_0 = \sup_{t \in [ 2- \alpha- 2 \beta  ]}\Gamma(t) ,
\end{equation}
so that finally we get
\begin{align}
\beta_n(t) \, \leq  \,&  \big( b\, t^{1 - \beta - \alpha} \,  \Gamma(1-\alpha) \big)^{n} \Gamma(1-\beta) c_0^{N_0} \frac{1}{  \Gamma \big(n+1 - n\alpha - n \beta   \big)} .
\end{align} 

Then the claim follows, as for all $z \geq 0$ and $\gamma>0$,
\begin{equation}\label{eq:calcGron}
1+ \sum_{n=1}^\infty \frac{z^n}{\Gamma\big(n \gamma  +1 \big) } \leq C(\gamma) \, e^{2 z^{1/\beta}}.
\end{equation}
Actually, to see \eqref{eq:calcGron} one can assume without loss of generality that $z \geq 1$ and calculate
\begin{align}
1+ \sum_{n=1}^\infty \frac{z^n}{\Gamma\big(n \gamma  +1 \big) } \, \leq \, & 1 + \sum_{k=0 }^\infty \bigg(  \sum_{n \colon \lfloor n \gamma \rfloor = k}   \frac{ \big( z^{1/\gamma} \big)^{\gamma n}  }{k! }   \bigg) \notag\\
\leq \, & 1 + \sum_{k=0 }^\infty  \bigg\lceil \frac{1}{ \gamma} \bigg\rceil   \bigg(   z  \frac{ \big( z^{1/\gamma} \big)^{k }}{k! }   \bigg) \notag\\
\leq \, & C(\gamma) \,z \,  e^{ z^{1/\gamma}} \leq \,  C(\gamma)  \, e^{ 2 z^{1/\gamma}}.
\end{align}
Setting $\gamma = 1 - \beta - \alpha$ this finishes the argument. 
\end{proof}

%
%
%
%
Now we have all the ingredients to treat the operator $G_T$ defined in \eqref{eq:Def-G}. For a fixed $K > 0$. Define
\begin{gather}
B_{K} = \Big\{v \in C^1_T \colon \| v \|_{C^1_T}  \leq K \Big \},\label{eq:def-ball} 
\end{gather}
where we have set
\begin{equation}
\| v \|_{C^1_T} \,= \,      \sup_{0 \leq t \leq T}  | v(t) |_{C^1}. 
\end{equation}

\begin{proposition}\label{lem:ex-fp-g}
Let $(X,\mathbf{X})$ and $\Psi$ be as described above and fix the constant $K>0$. Then there exists a $T=T(K,\|(X,\mathbf{X})\|_{\mathcal{C}_{X,T}^\alpha}, \|\Psi \|_{\mathcal{C}_{X,T}^\alpha} )>0$ such that the mapping $G_{\Psi,T}$ is a contraction on $B_{K}$. In particular, there exists a unique fixed point which we will denote by $v^\Psi$.
\end{proposition}

%
%

%
%
\begin{remark}
The proof is very similar to the proof of Theorem 4.7 in \cite{Ha10}. We can not apply this result directly as here the integrator $\Psi$ is an $X$ controlled rough path instead of a rough path by itself. Of course, $\Psi$ can be viewed as a rough path by itself, by defining the iterated integral as a rough integral, but this would lead to  problems in the next step of the construction as the mapping $\Psi \mapsto \int \delta \Psi \d \Psi$ maps a stochastic $L^p$ space continuously into a $L^{p/2}$ but not into itself, and it is not clear how to define a suitable stopping time to avoid this problem.
\end{remark}

\begin{proof}

We write 
\begin{align}
G_{T}(v)(t,x) \,=\,&  \int_0^t S(t-s)  \Big[  g\big(u(s) \big)  \partial_x \big( v(s) + U(s) \big) \Big]  \, \d s \,(x) \notag\\
&\qquad \qquad +\int_0^t \int_0^{1} p_{t-s} (x-y) g(u(y,s)) \, \mathrm{d}_y \Psi(y,s)     \, \d s \notag\\
= \,& G_{T}^{(1)}v (t,x) +G_{T}^{(2)} v (t,x). \label{eq:FP1}
\end{align} 

%
%
%

Let us treat the operator $G_{T}^{(1)}$ first. Using the fact that the operator $S(t)$ is bounded by $Ct^{-1/2}$ from $L^\infty$ to $C^1$ one gets
\begin{align}
\big| G_{T}^{(1)}v(t,\cdot) \big|_{C^1} \,  \leq \,&  \int_0^t \big| S(t-s) \big|_{L^\infty \to C^1} \,  \big| g \big|_0 \,  \Big( \big| v(s) \big|_{C^1}  +   \big| U(s) \big|_{C^1}  \Big)  \,\mathrm{d}s\notag \\
 \leq \, &   C |g|_{0}  \Big(   t^{1/2}   \| v \|_{C^1_T}  + t^{\frac{\beta}{2}} |u_0|_\beta  \Big), \label{eq:FP2}
\end{align} 
such that for $T$ small enough $  G_{T}^{(1)}$ maps $B_{K}$ into $B_{K/2}$. Regarding the modulus of continuity of $  G_{T}^{(1)}$ we write for $v,\bar{v}$ in  $B_{K}$ 
\begin{align}
\big| G_{T}^{(1)}v(t) \,-\,&G_{T}^{(1)} \bar{v}(t) \big|_{C^1}  \,  \leq  \,\bigg|  \int_0^t  S(t-s) \Big[ \Big( g \big(u(s)\big) - g \big(\bar{u}(s)\big) \Big)  \partial_y \big(  v(s) + U(s) \big) \Big]  \d s \bigg|_{C^1} \notag\\
&+ \, \Big|  \int_0^t  S(t-s) \Big[ \Big( g \big(\bar{v}(s)+\Psi(s) \big) \Big)  \partial_y \big( v(s) -\bar{v}(s) \big) \Big] \,  \d s \Big|_{C^1} , \notag
\end{align} 
such that
\begin{equation}
\big| G_{T}^{(1)}v(t) -G_{T}^{(1)} \bar{v}(t) \big|_{C^1} \leq C|g|_{C^1} \Big(  t^{1/2} (K+1)   + t^{\frac{\beta}{2}}   |u_0|_\beta \Big)  \| v -\bar{v} \|_{C^1_T}. \notag
\end{equation}
For $T$ small enough the last expression can be bounded by $\frac{1}{3}\| v -\bar{v} \|_{C^1_T}$.
%
%

Let us now treat the operator $G_{T}^{(2)}$. Using Lemma \ref{lem:scaling} (see also Remark \ref{rem:sca2}) and then Lemma \ref{lem:comp-reg-func} as well as \eqref{eq:comp-reg-func3} we get
\begin{align}
 \big| \partial_x  G_{T}^{(2)} & v (t,x)\big| \, = \,  \Big| \int_0^t  \int_0^{1} \partial_x \hat{p}_{t-s}(x-y) g\big(u(s,y) \big) \mathrm{d}_y \Psi(s,y) \,    \d s  \Big| \notag\\
 \leq \, & C \int_0^t   (t-s)^{\frac{\alpha}{2}-1} \big(1 + |(X(s),\mathbf{X}(s))|_{\mathcal{D}^\alpha} \! \big)\, \big| g\big(u(s) \big) \big|_{\mathcal{C}_{X(s)}^\alpha}  \big|\Psi(s) \big|_{\mathcal{C}_{X(s)}^\alpha}    \mathrm{d} s \notag\\ 
 \leq \, &  C  \big(1 + \|(X,\mathbf{X})\|_{\mathcal{D}^\alpha_T} \! \big)  \, |g|_{C^2} \notag\\
 & \qquad  \int_0^t   (t-s)^{\frac{\alpha}{2}-1}  \Big( 1+  \big|\Psi(s) \big|_{\mathcal{C}_{X(s)}^\alpha}^2  + K + |u_0|_\beta\,  s^{\frac{\beta-2 \alpha}{2}}  \Big)   \big|\Psi(s) \big|_{\mathcal{C}_{X(s)}^\alpha} \mathrm{d}s \notag\\
 \leq \, &  C  \big(1 + \|(X,\mathbf{X})\|_{\mathcal{D}^\alpha_T} \! \big)  \, |g|_{C^2} 
 \Big[   t^{\frac{\alpha}{2}}  \Big( 1+ K + \big\|\Psi(s) \big\|_{\mathcal{C}_{X}^\alpha}^3  \Big)    +t^{\frac{\beta-\alpha}{2}}  \, |u_0|_\beta\,    \big\|\Psi \big\|_{\mathcal{C}_{X}^\alpha}  \Big]  .\notag
\end{align}

A very similar calculation for the heat kernel without derivative using \ref{lem:scaling} (see also Remark \ref{rem:sca2}) gives
\begin{align}
 \big|   G_{T}^{(2)}v (t,x)\big| \, \leq \,  C   \big(1& + \|(X,\mathbf{X})\|_{\mathcal{D}^\alpha_T} \! \big)  \, |g|_{C^2} \notag \\
 & \Big( 1+  \big|\Psi(s) \big|_{\mathcal{C}_{X,T}^\alpha}^2 \Big)  
 \Big[   t^{\frac{\alpha+1}{2}} \Big( 1 + K  \Big)    +t^{\frac{\beta-\alpha+1}{2}}  \, |u_0|_\beta\,   \Big]    \big\|\Psi \big\|_{\mathcal{C}_{X}^\alpha} . \label{eq:ball-bou}
\end{align} 
So we can conclude that for $T$ small enough $G_{T}^{(2)}$ maps $B_{K}$ into $B_{K/2}$ as well.

%
%

To treat the modulus of continuity of $G_{T}^{(2)}$ write
\begin{align}
\partial_x   G_{T}^{(2)} &v(t,x) - \partial_x G_{T}^{(2)} \bar{v}(t,x)  \notag \\  = & \,  \int_0^t  \int_0^{1} \partial_x \hat{p}_{t-s}(x,y)  \Big( g \big(u(s,y)\big)  -g\big(\bar{u}(s,y)\big) \Big)  \mathrm{d}_y \Psi(s,y) \,  \d s. \label{eq:mod-cont1} 
\end{align}
Recall that $\bar{u}(s)= \bar{v}(s) + U(s) + \Psi(s) $.  Using Lemma \ref{lem:comp-reg-func} we see that for every $s$ the function $ g \big(u(s)\big)  -g\big(\bar{u}(s)\big) $ is an $X(s)$  controlled rough path. Thus applying \eqref{eq:hk-bou2} once more and then using \eqref{eq:comp-reg-func2} we get
\begin{equs}
\big| \partial_x   G_{T}^{(2)}&v(t,x) -\partial_x G_{T}^{(2)} \bar{v}(t,x) \big| \notag\\ 
\leq  \, & C \int_0^t   (t-s)^{\frac{\alpha}{2}-1} \big(1 + |(X(s),\mathbf{X}(s)|_{\mathcal{D}^\alpha} \! \big)\, \big| g\big(u(s) \big) -  g\big(\bar{u}(s) \big) \big|_{\mathcal{C}_{X(s)}^\alpha}  \big|\Psi(s) \big|_{\mathcal{C}_{X(s)}^\alpha}    \mathrm{d} s \notag\\ 
 \leq \, &  C  \big(1 + \|(X,\mathbf{X})\|_{\mathcal{D}^\alpha_T} \! \big)  \, |g|_{C^3}  \Big( 1+  \big|\Psi \big|^3_{\mathcal{C}_{X,T}^\alpha} \Big)  \notag\\
 & \qquad \qquad  \int_0^t   (t-s)^{\frac{\alpha}{2}-1} \big( 1+  K + |u_0|_\beta\,  s^{\frac{\beta-2 \alpha}{2}}   \Big)    |v(s)-\bar{v}(s)|_{C^1}    \mathrm{d}s \notag\\
 \leq \, &  C  \big(1 + \|(X,\mathbf{X})\|_{\mathcal{D}^\alpha_T} \! \big)  \, |g|_{C^3}   \Big( 1+  \big|\Psi \big|^3_{\mathcal{C}_{X,T}^\alpha} \Big) \notag\\
 & \qquad \qquad \Big[   t^{\frac{\alpha}{2}}  \Big( 1+ K  \Big)    +t^{\frac{\beta-\alpha}{2}}  \, |u_0|_\beta\,    \Big]  \|v(s)-\bar{v}(s)\|_{C^1_T}   .\label{eq:roca}
 \end{equs}
Repeating the same calculation yields a similar bound for  $ \big|   G_{T}^{(2)}v(t,x) -G_{T}^{(2)} \bar{v}(t,x) \big|$ with the same exponents of $t$ as in \eqref{eq:ball-bou} . Thus for $T$ small enough we can also bound $\big\|  G_{T}^{(2)}v -G_{T}^{(2)} \bar{v} \big\|_{C^1_T}$ by 
$ \frac{1}{3}\| v -\bar{v} \|_{C^1_T}$.  This finishes the proof. 
\end{proof}

%
%
%

We now discuss the continuous dependence of the fixed point $v^\Psi$ on the data. To this end let $(X, \XX) \in \DaT$,  $\Psi \in \CaT$ and the initial condition $u_0$  be as above and let $(\bar X, \bar{\XX}) \in \DaT$,  $\bar \Psi \in \mathcal{C}^\alpha_{\bar{X},T}$, $\bar{u}_0$ be another set of data. Assume that   all the norms $\|(X,\XX) \|_{\DaT}$, $\| \Psi \|_{\CaX}$ and $|u_0|_\beta$ as well as the corresponding norms for $\bar{X}$, $\bar{\Psi}$, and $\bar{u}$ are bounded by the constant $K$ which defines the size of the ball in which we solve the fixed point problem (see \eqref{eq:def-ball}). Note that by changing the existence time of the fixed point argument this can always be achieved. In particular, the existence time $T$ only depends on $K$ and will be fixed for the next proposition
\begin{proposition}\label{prop:cont-dep-psi}
Denote by $v=v^{X,\Psi,u_0}$ and by $\bar{v}=\bar{v}^{\bar{X}, \bar{\Psi} , \bar{u}_0}$ the fixed points constructed in Proposition \ref{lem:ex-fp-g} with the corresponding data. Then we have the following bounds
\begin{gather}
\| v - \bar{v} \|_{C^1_T } \, \leq \,C   \Delta_{\Psi, \bar{\Psi}}   \label{eq:c1-bou}
\end{gather} 
and 
\begin{gather}
\| v - \bar{v}\|_{C^{1/2}([0,T],C^0 )} \, \leq \,  C  \Delta_{\Psi, \bar{\Psi}}       \label{eq:calpha-bou},
\end{gather}
where we use $ \Delta_{\Psi, \bar{\Psi}} $ as abbreviation for
\begin{align}
 \Delta_{\Psi, \bar{\Psi}} \, = \, &  \bigg[ \|X-\bar{X} \|_{C^\alpha_T} +  \|\XX-\bar{\XX} \|_{\Omega C^{2\alpha}_T}    +    \| \Psi- \bar{\Psi} \|_{C^\alpha_T} +   \|\Psi'- \bar{\Psi'} \|_{C^\alpha_T} \notag \\
 &\qquad +  \|R_\Psi - R_{\bar{\Psi}}\|_{\Omega  C^{2\alpha}_T}     + |u_0 - \bar{u}_0 |_\beta  \bigg]. 
\end{align}
The constant $C$ here depends on $T,K,$ and $|g|_{C^3}$.
 \end{proposition}
\begin{remark}\label{rem:cont-dep-rp}
Performing the same proof with a single set of data $X,\Psi,u_0$ shows that one also has the bounds 
\begin{gather}
\| v \|_{C^1_T } \, \leq \,C    \bigg[ \|(X,\XX)  \|_{\DaT}  +    \| \Psi \|_{\CaT}  + |u_0  |_\beta  \bigg]
\end{gather} 
and 
\begin{gather}
\| v \|_{C^{1/2}([0,T],C^0 )} \, \leq \,  C  \bigg[ \|(X, \XX)  \|_{\DaT}  +    \| \Psi \|_{\CaT}  + |u_0  |_\beta  \bigg].
\end{gather}
\end{remark}

\begin{proof}
%
%
%
Let us show the bound \eqref{eq:c1-bou} for the spatial regularity first. By the definition of $v$ and $\bar{v}$ we can write
\begin{align}
v-  \bar{v} \,= \,   \Big(   G^{(1)} v -  \bar{G}^{(1)} \bar{v} \Big)+  \Big(   G^{(2)}v -  \bar{G}^{(2)}\bar{v} \Big). \label{eq:SCB1}
\end{align}
The operators $G^{(i)}$ and $\bar{G}^{(i)}$ are the same as those defined as in \eqref{eq:FP1} with respect to the respective data $X, \Psi, u_0$ and $\bar{X}$, $\bar{\Psi}$, $\bar{u_0}$. To simplify the notation we omit the dependence on $T$.

 As above we use that the heat semigroup $S(t)$ is bounded by $C t^{-1/2}$ as operator from $L^\infty$ to $C^1$ we get for the first term
\begin{align}
\big|  G^{(1)} v(t) -  \bar{G}^{(1)}  \bar{v}(t) & \big|_{C^1}  \leq \, C(1+  K)  |g|_{C^1}  \int_0^t (t-s)^{-1/2}   s^{\frac{\beta-1}{2}}  \notag\\ 
 &   \cdot \Big(  \big| v(s)-  \bar{v} (s) \big|_{C^1}  +    \big| \Psi(s)-  \bar{\Psi}(s) \big|_{0} + |u_0-\bar{u}_0|_\beta     \Big)\, \mathrm{d}s .    \label{eq:SCB2}
\end{align}
For the second term in \eqref{eq:SCB1} a similar calculation to \eqref{eq:roca} using the scaling bound for different reference rough paths \eqref{eq:sca-bou-drp} shows
\begin{align}
\big|  G^{(2)} & v(t)-  \bar{G}^{(2)}  \bar{v}(t)  \big|_{C^1} 
  \leq    C   \, |g|_{C^3}  
 \int_0^t   (t-s)^{\frac{\alpha}{2}-1} \Big( 1 + K^4  + K^3 s^{\frac{\beta-2 \alpha}{2}} \Big)  \d s  \,   \Delta_{\Psi,\bar{\Psi}}  \notag\\
 &+    C   \, |g|_{C^3}  
 \int_0^t   (t-s)^{\frac{\alpha}{2}-1} \Big( 1 + K^4  + K^3 s^{\frac{\beta-2 \alpha}{2}} \Big) |v(s)- \bar{v}(s)|_{C^1} \d s .  \label{eq:SCB3}
 \end{align}

Now applying the Gronwall Lemma \ref{lem:Gronwall} to   \eqref{eq:SCB2} and  \eqref{eq:SCB3} yields \eqref{eq:c1-bou}.

%

%
%

To treat the time regularity we write as in \eqref{eq:SCB1}
\begin{align}
\big( v(t) &-  \bar{v}(t) \big)  - \big(  v(s)-  \bar{v}(s) \big)
\,= \,   \Big(   G v(t) -  \bar{G} \bar{v}(t) \Big) -  \Big(   G v (s)-  \bar{G} \bar{v}(s)\Big) \label{eq:SCB4}.
\end{align} 
We again separate  $G$ into $G^{(1)}$ and $G^{(2)}$ and get  
\begin{align}
\Big(   G^{(1)} & v(t) -  \bar{G}^{(1)} \bar{v} (t) \Big) -  \Big(   G^{(1)} v (s)-  \bar{G}^{(1)}  \bar{v} (s) \Big)\notag\\
 =\,& \Big( S(t-s) -1  \Big) \int_0^s S(s-\tau) \Big[ g\big( u(\tau) \big) \partial_x \big( v(\tau) +U(\tau)\big)  - g\big(\bar{u}(\tau) \big)  \partial_x \big( v^{\bar{\Psi}}(\tau) + \bar{U}(\tau)\ \big)     \Big] \mathrm{d}\tau  \notag\\
 &\quad +  \int_s^t S(t-u)\Big[ g\big( u(\tau) \big) \partial_x \big( v(\tau) +U(\tau)\big)  - g\big(\bar{u}(\tau) \big)  \partial_x \big( \bar{v}(\tau) + \bar{U}(\tau)\ \big)     \Big]   \mathrm{d}\tau. \label{eq:SCB5}
\end{align}
Note that here we write $\bar{u}=\bar{v} + \bar{\Psi} + \bar{U}$.  The first term on the right hand side of \eqref{eq:SCB5} can be bounded using \eqref{eq:c1-bou}
\begin{align}
C \big\| & S(t-s) -1  \big\|_{C^1 \to C^0}  |G^{(1)}  v(s) -  \bar{G}^{(1)} \bar{v} (s) |_{C^1}  \leq C (t-s)^{\frac{1}{2}}     \Delta_{\Psi,\bar{\Psi}}  \notag  ,
\end{align}
 because $\big\| S(t-s) -1  \big\|_{C^1 \to C^0} \leq C(t-s)^{1/2}$. The second term on the right hand side of \eqref{eq:SCB5} can be bounded by
\begin{align}
\int_s^t & \Big( |g|_0  \Big( \big|v(\tau)-  \bar{v}(\tau) \big|_{C^1} + \tau^{\frac{\beta-1}{2}} |u_0 -\bar{u}_0|_\beta \Big)\notag\\
 &  + \big( 1 +\tau^{\frac{\beta-1}{2}}\big) K  |Dg|_0 \Big( \big| v(\tau)-  \bar{v}(\tau) \big|_{0} +  \big|\Psi(\tau)- \bar{\Psi}(\tau) \big|_{0}  + |u_0 -\bar{u}_0|_\beta    \Big) \d \tau \notag\\
&\qquad \leq C (t-s)^{\frac{\beta +1}{2}}  \Delta_{\Psi,\bar{\Psi}},   \notag
\end{align}
where we use the fact that $S(t-u)$ is a contraction from $C^0$ into itself.

%
%

For the term involving $G^{(2)}$ we write
\begin{equs}
\Big(&   G^{(2)}  v(t,x) -  \bar{G}^{(2)} \bar{v} (t,x) \Big) -  \Big(   G^{(2)} v (s,x)-  \bar{G}^{(2)}  \bar{v} (s,x) \Big)\notag\\
=\,& \int_0^s  \! \bigg( \int_0^1 \!  \big( \hat{p}_{t-\tau}(x-y) -  \hat{p}_{s-\tau}(x-y)   \big)\,  g(u(\tau,y)) \, \mathrm{d}_y \Psi(\tau,y)   \notag\\
&\qquad  \qquad - \int_0^1 \big( \hat{p}_{t-\tau}(x-y) -  \hat{p}_{s-\tau}(x-y)   \big) \,g\big( \bar{u}(\tau,y) \big) \, \mathrm{d}_y   \bar{\Psi}(\tau,y)   \bigg) \, \mathrm{d}\tau  \label{eq:SCB7}\\
&+ \int_s^t \bigg( \int_0^1   \hat{p}_{t-\tau}(x-y)     g\big(u(\tau,y) \big) \, \mathrm{d}_y  \Psi(\tau,y)  
- \int_0^1 \hat{p}_{t-\tau}(x-y)     \,g\big( \bar{v}(y,\tau) \big) \, \mathrm{d}_y \bar{\Psi}(\tau,y)  \bigg) \, \mathrm{d}\tau. \notag
\end{equs}
The second summand in \eqref{eq:SCB7} can be bounded as above (see e.g. \eqref{eq:roca}, \eqref{eq:SCB3}) by:
\begin{align}
C |g|_{C^3} (t-s)^{\frac{\beta-\alpha+1}{2}} \Delta_{\Psi, \bar{\Psi}}.  \label{eq:SCB8}
\end{align}
In order to treat the first summand note that by the semigroup property
\begin{equation}
\hat{p}_{t-\tau}(x-y) \, =  \int_0^1  \hat{p}_{t-s}(x-z)   \hat{p}_{s-u}(z-y) \,  \mathrm{d} z.
\end{equation}
Thus in the first integral in \eqref{eq:SCB7} we can rewrite  using Lemma \ref{lem:fubini}, the Fubini Theorem for rough integrals: 
\begin{align}
  \int_0^1 &   p_{t-u}(x-y) \, g \big(u(\tau,y) \big)\, \mathrm{d}_y \Psi(u,y) \,\mathrm{d}u \notag\\
 &=\, \int_0^1 \Big(  \int_0^1 \hat{p}_{t-s}(x-z)  \, \hat{p}_{s-\tau}(z-y) \, \mathrm{d} z \Big)   \, g\big(u(\tau,y) \big)\, \mathrm{d}_y \Psi(\tau,y) \,  \notag\\
&= \,  S(t-s)   \Big( \int_0^1 \hat{p}_{s-\tau}(\cdot -y)   \, g\big(u(\tau,y) \big)\, \mathrm{d}_y \Psi(\tau,y)      \Big)(x),
\end{align}
and similarly for the integral involving $\bar{\Psi}$. Thus the first difference in \eqref{eq:SCB7} can be bounded by
\begin{align}
\big| S(t-s) - 1\big|_{C^1 \to C^0}  \Big|   G^{(2)} v (s) -  \bar{G}^{(2)} \bar{v} (s) \Big|_{C^1}  \leq C(t-s)^{1/2}    \Delta_{\Psi,\bar{\Psi}}     .    
\end{align}
Here we have used  \eqref{eq:c1-bou} again. This finishes the proof.
\end{proof}

Now we are ready to construct local solutions to \eqref{eq:Stoch-Burg}.  To this end we introduce  the following families of stopping times.  For $K_1, K_2, K_3>0$ and for any adapted processes $\Psi$, $R$ taking values in $C^\alpha$ resp. $\Omega C^{\alpha}$ denote by 
\begin{align}\label{eq:def-st}
\tau_{K_1}^X\,  =  \, & \inf  \bigg\{ t \in  [0,T] \colon  \sup_{\substack{ x_1 \neq x_2 \\ 0 \leq s_1< s_2 \leq t}}  \frac{\big|X(s_1,x_1) -  X(s_2,x_2) \big|  }{|s_1-s_2|^{\alpha/2}  + |x_1-x_2|^{\alpha} } +| \XX(t) |_{2\alpha}    >  K_1  \bigg\}
 \notag\\
\tau_{K_2}^{|\Psi|_{C^\alpha} } \,=  \, & \inf  \bigg\{ t \in  [0,T] \colon    | \Psi(t) |_{C^\alpha}  >  K_2  \bigg\} \notag \\
\tau_{K_3}^{|R |_{2\alpha} } \,=  \, & \inf  \bigg\{ t \in  [0,T] \colon    | R(t) |_{2\alpha}  >  K_3  \bigg\} .
\end{align}
Note that the definition of $\tau^X_K$ is slightly different from \eqref{eq:def-tau} as it also contains information about $\XX (t)$. It is almost surely less or equal to the stopping time defined above and in particular Proposition \ref{prop:Stoch-Conv-Rough-Path} is valid also with this modified definition.

%
%
%
%

\begin{proposition}\label{prop:loc-ex}
For any initial data $u_0 \in C^\beta$ there exists a $T^*>0$,  an adapted process $u^*(t,x)$ for $(t,x) \in [0,T^*] \times [0,1]$ and a stopping time $\stf$ such that up to $T^* \wedge \stf$ the process $u^*$ satisfies \eqref{eq:mild-sol}.  Furthermore, the time $T^*$ only depends on $|u_0|_\beta, K_1, K_2, K_3$ and the stopping time $\stf$ is given by
\begin{equation}
\stf = \tau_{K_1}^{X} \wedge  \tau_{K_2}^{|\Psi^{\theta(u)}|_{C^\alpha} } \wedge  \tau_{K_3}^{|R^{\theta(u)} |_{2\alpha} }.
\end{equation} 
\end{proposition}
%
%
%
%
\begin{proof}
We construct $u$ by another fixed point argument, this time in the space of adapted stochastic processes on a time interval $[0,T]$, such that $\$ \cdot \$_{p,\alpha}$ is finite for suitably chosen $p$. (Recall that $\$ \cdot \$_{p,\alpha}$ was defined in \eqref{eq:def-nom}). Let us denote the space of all these processes by $\Ap$. Furthermore, denote by $\Ap_{K_4}$ the set of processes $u \in \Ap$ with $\$ u \$_{p,\alpha} \leq K_4$.

We first need to introduce a cutoff function: note that by definition of the stopping times for $ t \leq \stf$ and for $u \in \Ap_{K_4}$ one has $\| \Psi^{\theta(u)}\|_{\CaT} \leq C |\theta|_{C^1} K_4 + K_2 +K_3 = K_5$. Then for or $K_5 > 1$ let $\chi_{K_5}$ be a decreasing, non-negative $C^1$ function which is constantly equal to one on $(-\infty, K_5)$ and such that $ |\chi_{K_5}'|_0 \leq 1$ and for all $x \geq K_5$
\begin{equation}
x \, \chi_{K_5} (x) \leq 2 K_5. 
\end{equation}

Now for $u \in \Ap$ define the operator $\mathcal{N}$ as
\begin{align}
\mathcal{N}u(t) = S(t_{K_1} ) u_0 + \Psi^{\theta(u)}(t_{K_1}) + v^{\Psi_{K_2}^{\theta(u)}}(t_{K_1}),
\end{align}
where we write $t_{K_1} =  t \wedge \tau_{K_1}^{X }$. 
 
Here the $\Psi^{\theta(u)}$ is the stochastic convolution defined in \eqref{eq:Stoch-Conv} for the adapted process $\theta \big(u(t,x) \big)$. By $\Psi_{K_2}^{\theta(u)}$ we denote the same stochastic convolution cut off at $K_2$, i.e.
\begin{equation}
\Psi_{K_2}^{\theta(u)}(t)  \,= \, \Psi^{\theta(u)} (t   ) \, \chi_{K_2} \Big( \big|  \Psi^{\theta(u)}(t)  \big|_{ \mathcal{C}_{X(t)}^\alpha}    \Big). \label{eq:cutoffpsi}
\end{equation} 
Finally, by $v^{\Psi_{K_2}^{\theta(u)}}$ we denote the fixed point constructed in Proposition \ref{lem:ex-fp-g}. Note that  due to the definition of $\chi_{K_2}$  clearly $\| \Psi_{K_2}^{\theta(u)}\|_{\CaT} \leq 2 K_2 $. Furthermore, due to the definition of the stopping time $\tau^X_{K_1}$ all relevant norms of the reference rough path $(X,\XX)$ are bounded by $K_1$ so that this fixed point is defined up to a final time $\tilde{T}(K_1,K_5) \wedge \tau_{K_1}^X$ that does not depend on $u$. 

For $T$ small enough and for $K_4$ big enough (depending on $|u_0|_\beta$) the operator $\mathcal{N}$ maps $\Ap_{K_4}$ into itself. Indeed, the deterministic part $S(t) u_0$ has the right regularity due to the assumption $u_0 \in C^\beta$ and standard properties of the heat semigroup. Due to the boundedness of $\theta$  we can apply Proposition  \ref{prop:Reg-Stoch-Conv} (see also Remark \ref{rem:reg-stoch-conv}) to see that for $p >  \frac{12}{1-2\alpha} $ and for $T$ small enough the stochastic convolution $\Psi^{\theta(u)}$ also takes values in $\Ap_{K_4}$. Finally, as clearly $\$ \theta(u) \$_{p,\alpha} \leq  | \theta |_{C^1}  \big(1 + \$  u\$_{p,\alpha} \big)$ we can apply   Proposition \ref{prop:Stoch-Conv-Rough-Path} to conclude that for $p$ satisfiying \eqref{eq:condp2}  $\Psi^{\theta(u)}$ is in fact a controlled rough path with 
\begin{equation}
\Ex \big[  \sup_{0 \leq t \leq T \wedge \tau_{K_1}^X} |R^{\Psi^{\theta(u)}}|_{2\alpha}^p  \big]  \, \leq \, C(1 + K_1^p) \, T^\vartheta  \, | \theta |_{C^1} \,  \big( 1 + K_4 \big),
\end{equation}
for any $\vartheta$ satisfying \eqref{eq:cond-thet2}. So Proposition \ref{prop:cont-dep-psi} (see also Remark \ref{rem:cont-dep-rp}) implies that also $ v^{\Psi_K^{\theta(u)}}$ takes values in $\Ap_{K_4}$ for $T$ small enough. Note that $K_4$ can be chosen depending only  on $|u_0|_{\beta}$ and $K_1, K_2, K_3$.  

Let us show that $\mathcal{N}$ is indeed a contraction. To this end for $u, \bar{u}$ calculate
\begin{equation}
\$ \mathcal{N} u - \mathcal{N}\bar{u} \$_{p,\alpha} \leq  \$ \Psi^{\theta(u)} - \Psi^{\theta(\bar{u})} \$_{p,\alpha} +  \$ v^{\Psi_K^{\theta(u)}} - v^{\Psi_K^{\theta(\bar{u})}} \$_{p,\alpha}.
\end{equation}
To deal with the stochastic convolutions note that by Proposition \ref{prop:Reg-Stoch-Conv} and by Remark \ref{rem:reg-stoch-conv} for any $p > \frac{12}{1-2 \alpha}$ there exists a $\vartheta >0$ such that  
\begin{align}
  \$ \Psi^{\theta(u)}(t) - \Psi^{\theta(\bar{u})}(t) \$_{p,\alpha} \, \leq \,   & C T^{\vartheta/2}  \$ \theta(u) -\theta(\bar{u}) \$_{p,0} \notag\\
  \leq \, & C T^{\vartheta/2} |\theta |_{C^1} \$ u -\bar{u} \$_{p,0}. \label{eq:concl1}
\end{align}
So for $T$ small enough this can be bounded by $\frac{1}{3} \$ u -\bar{u} \$_{p,\alpha}$.

A calculation similar to the proof of Lemma \ref{lem:comp-reg-func} shows that 
\begin{align}
\$ \theta(u) - \theta(\bar{u}) \$_{p,\alpha} \, \leq \, &C(1+K_4)   |  \theta |_{C^2}   \$u -\bar{u} \$_{p,\alpha} \label{eq:concl2}.   
\end{align}
Note that in this bound the factor $K_4$ appears and we do not expect a similar bound independent of $K_4$ to hold . It is the only place where we use that $\mathcal{N}$ is defined on a ball. 

Therefore, using Proposition \ref{prop:Stoch-Conv-Rough-Path} we can see that for $p$ even bigger (satisfying \eqref{eq:condp2}) we can chose an even smaller $\vartheta$ (as in \eqref{eq:cond-thet2}) such that 

\begin{align}
\Ex  \bigg[   \big\| R^{\theta(u)}  - R^{\theta(\bar{u})}  \big\|_{C \left( [0,\tau_{K_1}^{X}];\Omega C^{2\alpha}  \right)    }^p    \bigg] \,\leq \, C T^{\vartheta/2}  \$ u-\bar{u} \$_{p,\alpha}^{p}\label{eq:concl3}  . 
\end{align}
Now using \eqref{eq:concl1} - \eqref{eq:concl3} and the continuous dependence of the fixed point $v^\Psi$ on the controlled rough path $\Psi$, Proposition \ref{prop:cont-dep-psi}, we get 
\begin{align}
 \$ v^{\Psi_K^{\theta(u)}} - v^{\Psi_K^{\theta(\bar{u})}} \$_{p,\alpha} \leq C  T^{\vartheta/2}  \$ u-\bar{u} \$_{p,\alpha}^{p}. 
\end{align}
Choosing a $T^*$ small enough this quantity can also be bounded by  $\frac{1}{3} \$ u-\bar{u} \$_{p,\alpha}^{p}$. This shows that for this $T^*$ and for $u,\bar{u} \in \Ap_{K_4}$ indeed $\$ \mathcal{N} u - \mathcal{N}\bar{u} \$_{p,\alpha} \leq   \frac{2}{3}\$ u-\bar{u} \$_{p,\alpha}^{p} $ and in particular there exists a unique fixed point $u^*$. 

To confirm that $u^*$ is indeed a mild solution up to $T^* \wedge \stf$ note that for $t \leq \stf $ the fixed point equation reads
\begin{align}
u(t,x) \,=  \,& S(t) u_0 + \Psi^{\theta(u)}(t) + v^{\Psi^{\theta(u)}}(t) \notag\\
=  \,& S(t) u_0 + \Psi^{\theta(u)}(t)  + G_{\Psi(u)}   v^{\Psi^{\theta(u)}}(t),
\end{align}
which is precisely the definition of a mild solution. 
\end{proof}
%
%
%

Now we are ready to finish the construction of global solutions to \eqref{eq:Stoch-Burg}. We now show global existence as well as uniqueness:

\begin{proof}(of Theorem \ref{thm:MR})
We will show global existence first. Fix a $T > 0$. We will show that we can construct solutions up to time $T$. According to Proposition \ref{prop:loc-ex} for fixed $K_1, K_2, K_3$ and any initial data $u_0 \in C^\beta$ we can construct a process $u^*$ up to time $T^*$ which then is a solution up to some stopping time 
$T^* \wedge \stf$. By taking $u^*(T)$ as new initial condition and then iterating this procedure (for fixed $K_i$) one can extend this process up to $T \wedge T^{**}$ where $T^{**}$ is the first blowup time of $|u^*|_{\beta}$. 

A priori one should be careful at this point. The bounds used in construction of local solution have been derived using that the reference rough path $X$ starts at $0$. So in this way we get solutions with a reference rough path that is restarted at $T^*$. But the discussion after Definition \ref{def:weak-sol} shows that this does not matter.

Using the definition of the fixed point $u^*$ we get for any $t \leq T$ that
\begin{align}
| u^*(t)|_{\beta} \, \leq  \, & | u_0|_\beta  + | \Psi^*(t) |_{\beta}  +  \big|  G^* v^* (t)\big|_\beta. \label{eq:ge1}
\end{align}     
In the first term we have used the contraction property of the heat semigroup. Here to shorten the notation we write $\Psi^*(t)$ for $\Psi_{K_2}^{\theta(u^*)} (t_{K_1})$,   $v^*(t)$ for $ v^{\Psi_{K_2}^{\theta(u^*)}}(t_{K_1})(t)$, as well as $G^*$ for $G_{\Psi_{K_2}^{\theta(u^*)}}$. Recall that $t_{K_1} = t \wedge\tau^X_{K_1}$.  In particular, the construction of the process $u^*$ still includes all the cutoffs, and therefore the right hand side of \eqref{eq:ge1} depends on the $K_i$, although this is suppressed in the notation. 

 For the second term we get due to Proposition \ref{prop:Reg-Stoch-Conv} and Remark \ref{rem:reg-stoch-conv}
\begin{equation}
\Ex \Big[  \sup_{0 \leq t \leq T }  \big|  \Psi^*(t)  \big|_{\beta}^p \Big] 
\leq C | \theta |_0^p \label{eq:ge2}
\end{equation}
for $p$ big enough, so that this quantity can not blow up. Note that the constant in \eqref{eq:ge2} does not depend on the $K_i$.

To deal with  $ v^*$ we need to perform a calculation which is very similar to the proof of Proposition \ref{lem:ex-fp-g} and Proposition \ref{prop:cont-dep-psi}. The difference is that we need to use the full formulas \eqref{eq:sca1} and \eqref{eq:comp-reg-func1} instead of \eqref{eq:sca3} and \eqref{eq:comp-reg-func3}. Using the definition of the fixed point $v$ and recalling the definitions of the operators $G^{(1)}_{\Psi^*}$ and $G^{(2)}_{\Psi^*}$ from \eqref{eq:FP1} we get
\begin{align}
|v^*(t)|_{C^1} \, \leq \, | G^{(1)}_{\Psi^*} v^*(t) |_{C^1} + | G^{(2)}_{\Psi^*} v^*(t) |_{C^1} . \label{eq:ge3}
\end{align}
In the same way as in in \eqref{eq:FP2} we get for the term involving $G^{(1)}_{\Psi^*}$:
\begin{equation}
| G^{(1)}_{\Psi^*} v^*(t) |_{C^1}  \,  \leq  \,  C |g|_{0}  \int_0^{t_{K_1}} (t_{K_1}-s)^{-\frac{1}{2}} \Big( |v^*(s)|_{C^1} +  s^{\frac{\beta-1}{2}} |u_0|_\beta  \Big)\,  \d s. \label{eq:ge}
\end{equation}
To treat $G^{(2)}_{\Psi^*}$ write very as in  \eqref{eq:ball-bou} but using the more precise bounds \eqref{eq:sca1} and \eqref{eq:comp-reg-func1}
\begin{align}
\big| &  G^{(2)}_{\Psi^*} v^*(t) \big|_{C^1} 
\, \leq \, C |g|_{C^2} \! \int_0^{t_{K_1}}  \! (t_{K_1}-s)^{\frac{\alpha}{2}-1} \notag\\ 
&\bigg[  |\Psi^*(s) |_\alpha \Big( 1+  |v^*(s)|_{C^1} + s^{\frac{\beta-2 \alpha}{2}} |u_0|_\beta 
    +   |\Psi^*(s) |_\alpha^2 + \big| R_{\Psi^*} (s) \big|_{2 \alpha}    \Big)   \notag\\
   &  + |X(s)|_\alpha |\theta |_0 |R_{\Psi^*}(s) |_{2 \alpha}   + | \XX(s)|_{2 \alpha}|\theta|_{C^1}^2  \Big( |v^*(s)|_{C^1}  + |\Psi^*(s) |_\alpha + |u_0|_{\alpha} \Big)  \notag\\ 
 &  +  | \XX(s)|_{2 \alpha} |\Psi^*(s) |_0 \, |\theta|_0 \, \Big(   |\Psi^*(s) |_\alpha + |v(s)|_{C^1} + |u_0|_\alpha  \Big)      +|X(s)|_\alpha^2 | \theta |_0 |\Psi^*(s) |_\alpha  \bigg] \, \d s. \label{eq:cf} 
\end{align}
We observe that all the terms on the right hand side of \eqref{eq:cf} except for $|v^*(s)|_{C^1}$ are bounded by powers of the constants $K_i $ and $ |v^*(s)|_{C^1}$ only appears linearly. Thus the Gronwall Lemma \ref{lem:Gronwall} gives $\sup_{0 \leq t \leq T } | v^*(t)|_{C^1} \leq C$ for a finite deterministic constant $C$ that only depends on the $K_i$ and $|u_0|_\beta$.  In particular for all $K_1, K_2, K_3 >0$ we can construct $u^*$ up to time $T$. Furthermore, up to $\stf$ the process $u^*$ solves \eqref{eq:mild-sol}. 

We want to show that by choosing the constants $K_i$ large enough it is possible to guarantee that $\stf \geq T$. Let us start by showing that $\lim_{K_3 \to \infty}   \stf  \geq  \tau_{K_2}^{|\Psi|_{C^\alpha} } \wedge  \tau_{K_1}^X = \stt$ i.e. that $|R |_{2\alpha}$ cannot explode for bounded $\Psi$ and $X$.

Now by Proposition \ref{prop:Stoch-Conv-Rough-Path} we know that for $p$ satisfying \eqref{eq:condp2} 
\begin{equation}\label{eq:Rb}
\Ex \Big[  \big\| R^{\theta(u^*)} \big\|_{C([0, t \wedge \tau_{K_1}^X];\Omega C^{2\alpha})}^p   \Big] \, \leq \, C (1+K_1^p) | \theta|_{C^1} \$ u^* \|_{p,\alpha,t}^p, 
\end{equation}
where  
\begin{equation}
\$ u^* \$_{p,\alpha,t} \,= \, \Ex \bigg[ \sup_{x_1 \neq x_2, s_1 \neq s_2 \leq t}   \frac{|u^*(s_2,x_2)- u^*(s_1,x_1)|^p}{\big(|s_1-s_2|^{ \alpha /2}+ |x_1-x_2|^{\alpha }\big)^p}     \bigg]^{1/p}, 
\end{equation}
i.e. the $ \$ \cdot \$_{p,\alpha} $-norm evaluated up to time $t$. 
 
Let us derive a bound on $ \$ u \$_{\alpha,p}$. First of all using   \eqref{eq:ge1} -- \eqref{eq:cf} and noting that on the right hand side of \eqref{eq:cf} all the terms except for those involving $v^*$ or $R_{\Psi^*}$ are bounded by powers of the constants $K_1, K_2$ and that no term involving any product of norms of $v^*$ or $R_{\Psi^*}$ appears one can write for any $p \geq 1$
\begin{align}
\Ex & \Big[ \sup_{0 \leq s \leq t \wedge \stt } |v^*(s)|_{C1}^p  \Big] \notag\\
& \leq \,  C \,  \Ex \Big[ \Big( \sup_{0 \leq s \leq t \wedge \stt}  \int_0^s (s - \tau)^{\frac{\alpha}{2}-1} \big(1 + \tau^{\frac{\beta - 2 \alpha}{2}}   + |v^*(s)| + \big| R^{\theta(u^*)} (s)\big| \big)       \d \tau \Big)^p  \Big]   \\
&\leq   C  \int_0^t (t - \tau)^{\frac{\alpha}{2}-1} \big( 1 + s^{\frac{\beta - 2 \alpha}{2}}   +\Ex \big[ |v^*(s \wedge \stt)|^p_{C^1} \big] + \Ex \big[ \big| R^{\theta(u^*)} (s \wedge \stt)\big|_{2 \alpha}^p \big]     \big) \, \d \tau  .\notag
\end{align}
Here the constant $C$ depends on $|g|_{C^3}, |\theta|_{C^1}, p, |u_0|_\beta, K_1$, and $K_2$ but not on $K_3$. Repeating the same calculation as in the second half of the proof of Proposition \ref{prop:cont-dep-psi} which is essentially an exploitation of the bound $|S(t) -1|_{C^1 \to C^0} \leq C t^{1/2}$ one gets the bound 
\begin{align}
\Ex & \Big[ \|v^*\|_{C^{1/2}([0, t \wedge \stt ], C)}^p \Big] \\
&\leq   C  \int_0^t (t-s)^{\frac{\alpha}{2}-1} \big( 1 + s^{\frac{\beta - 2 \alpha}{2}}   +\Ex \big[ |v^*(s \wedge \stt)|^p_{C^1} \big] + \Ex \big[ \big| R^{\theta(u^*)} (s \wedge \stt)\big|_{2 \alpha}^p \big]  \big)     \d s  .\notag
\end{align}
Noting that according to \eqref{eq:ge2} we know that 
\begin{equation}
\$ u^* \$_{p,\alpha,t} \leq C\Big(1 + |\theta|_0^p + \Ex\Big[  \sup_{0 \leq s \leq t \wedge \stt } |v^*(s)|_{C1}^p +   \|v^*\|_{C^{1/2}([0, t \wedge \stt ], C)}^p \Big]  \Big),  \label{eq:ubohmw}
\end{equation} 
we get using \eqref{eq:Rb}, that for $p$ large enough to satisfy \eqref{eq:condp2}  
\begin{align}
\$ v \$_{p,1,1/2,t}^p  \leq   C  \int_0^t (t-s)^{\frac{\alpha}{2}-1} \big( 1 + s^{\frac{\beta - 2 \alpha}{2}}  + \$ v \$_{p,1,1/2,s}^p \big)       \d s  ,
\end{align}
where
\begin{equation}
	\$ v \$_{p,1,1/2,t}^p =  \Ex  \Big[ \sup_{0 \leq s \leq t \wedge \stt } |v(s)|_{C1}^p + |v|_{C^{1/2}([0,t \wedge \stt],C)}  \Big] .
\end{equation}

So the Gronwall lemma \ref{lem:Gronwall} gives a uniform bound on this quantity for all $t \leq T$. Thus using \eqref{eq:Rb} as well as \eqref{eq:ubohmw} one more time, we can deduce that 
\begin{equation}
\sup_{K_3} \Ex \Big[   \big\| R^{\theta(u^*)} \big\|_{C([0, T \wedge \stt ];\Omega C^{2\alpha})}^p   \Big]  \leq  C(K_1,K_2) <\infty. 
\end{equation}
This implies by Markov inequality that 
\begin{equation}
\Prob \Big[ \tau_{K_3}^{|R|_{2 \alpha}} \leq T \Big] =  \Prob \Big[     \big\| R^{\theta(u)} \big\|_{C[0, t];\Omega C^{2\alpha})} \geq K_3   \Big] \leq \frac{C}{K_3^p}.
\end{equation}
This gives the desired result concerning non-explosion of $R_\Psi$ before $\stt$. But then as the expectation of the $p$-th moment of the $\alpha$-H\"older norm of $\Psi^{\theta(u)}$ is controlled by $|\theta|_0^p$ and as we know a priori that the stochastic convolution  cannot blow up we see that we can remove $\stt$ as well. This finishes the proof of global existence.

Uniqueness up to the stopping time $\stf$ follows from the construction as a fixed point in Proposition \ref{prop:cont-dep-psi}. Then by the same argument as above we can remove the stopping times and obtain uniqueness. This finishes the proof.
\end{proof}

\begin{proof}(of Theorem \ref{thm:Stability})
For $\eps > 0$ the approximation $u_\eps$ is the unique solution to the fixed point problem
\begin{align}
u_\eps(t)  &\,= \, S(t) u_0 + \int_0^t S(t-s) \, g\big(u(s)\big) \, \partial_x u(s) \, \d s + \int_0^t S(t-s) \, \theta \big(u_\eps(s) \big) \d (W \ast \eta_\eps)  (s) \notag\\
  & \, = \, U(t) + G u_\eps(t) + \Psi^{\theta(u_\eps)}_\eps(t).
\end{align}
Of course, as $u_\eps$ is smooth in the space variable $x$ there is no need for rough path theory to define the integral involving $g$. Nonetheless we can view this integral as a rough integral. To be more precise as a reference rough path one chooses  
\begin{equation}
\XX_\eps(t,x,y) \, = \, \int_x^y \big( X_\eps(t,z) -X_\eps(t,x)  \big) \d_z X_\eps(t,z). 
\end{equation}
As $X_\eps$ is a smooth function of the space variable $x$ this integral can be defined as a usual Lebesgue integral.
 Note that by  Proposition \ref{prop:regrem} $\Psi^{\theta(u)_\eps}_\eps $ is indeed an $X_\eps$ controlled rough path with decompostion given as 
\begin{equation}\label{eq:eps-dec}
 \delta \Psi^{\theta(u_\eps} (t,x,y) \, = \, \theta(u_\eps(t,x)) \, \delta X_\eps(t,x,y)    + R_\eps^{\theta(u_\eps)}(t,x,y). 
\end{equation}
As $X_\eps$ is a smooth function this decomposition is by no means unique but it has the advantage that Proposition  \ref{prop:regrem} provides moment estimates on the remainder, that are uniform in $\eps$.  

Introduce the stopping time  
\begin{align}
\tau_{K_1,K_2,K_3,K_4,\eps}\,= \,   \sigma^{4,\eps}_{K_1, K_2,K_3,K_4 } \wedge \sigma^{4,\eps}_{K_1, K_2,K_3,K_4 } ,
\end{align}
where $\sigma^{4,\eps}_{K_1, K_2,K_3,K_4 }$ is defined as $\stf$ stopped when $\| u\|_{C^{\alpha/2,\alpha}}$ exceeds $K_4$ and is $\sigma^{3,\eps}_{K_1, K_2,K_3 }$ is defined analogously with respect to the norms derived from $u_\eps$. Sometimes, we will use the shorthand $\tau_\eps$.
 
We will show first, that before $\tau_\eps$ the process   $u_\eps$ coincides with the solution to the cut-off fixed point problem 
\begin{align}
u_\eps(t)  &\,= \, S(t) u_0 + v^{\chi_{K_5}(\Psi_\eps)}+ \int_0^t S(t-s) \, \theta \big(u_\eps(s) \big) \d (W \ast \eta_\eps) (s).  \label{eq:smoothfixpoint}
\end{align}

Here the fixed point $v^{\chi_{K_5}(\Psi_\eps)}$ is defined as in Proposition \ref{lem:ex-fp-g}.  (In particular, we use the cutoff at $K_5 = C K_4 + K_2+ K_3 $ to define the cutoff of $\Psi_\eps$).

As $\Psi^{\theta(u_\eps)} $ is also a smooth function of the space variable $x$ there is no need for rough path theory in order to define the integral that determines the fixed point map in Proposition \ref{lem:ex-fp-g}.  But actually, the two integrals coincide, because in the approximation 
\begin{align}
\sum_i \hat{p}_{t-s}(y-x_i) g(u_i) \delta u_{i,i+1} +  \hat{p}_{t-s}(y-x_i) g'(u_i) \theta (u_i) \XX_\eps(x_i, x_{i+1} )\theta(u_i)^T
\end{align}
the first terms converge to the usual integral $\int_0^1 \hat{p}_{t-s}(y-x) g\big( u(x) \big) \partial_x u(x) \, \d x$ and the sum involving the iterated integrals converges to zero due to $\XX_\eps \in \Omega C^{2}$. In particular, all the bounds derived for the fixed point apply in the present context. So due  to uniqueness of the smooth evolution \eqref{eq:smoothfixpoint} holds.

In Section \ref{sec:Prel-Calc} we have already derived uniform bounds for the quantities appearing in this decomposition. First of all, for the derivative processes we have
\begin{equation}
\$ \theta(u) - \theta(u_\eps)  \$_{p,\alpha} \, \leq \, C K_4 |\theta|_{C^2} \$ u - u_\eps  \$_{p,\alpha} .
\end{equation}
From \eqref{eq:Hol-Bou-Psi2} and \eqref{eq:Hol-Bou-Psi6} we have
\begin{align}
\Ex \Big[   \big\| \Psi^{\theta(u)} - \Psi_\eps^{\theta(u_\eps)}   \big\|_{C( [0,T]; C^{ \alpha} )  }^p        \Big] \, &\leq \,  C \eps^{\gamma} \$ \theta(u) \$_{p,\alpha}^p +  C T^{\vartheta}  \$ \theta(u) - \theta(u_\eps)  \$_{p,0}^p \notag\\
& \leq  C \eps^{\gamma} + C T^{\vartheta}  \$ u - u_\eps  \$_{p,0}^p ,
\end{align}
where $\vartheta$ is as defined in Proposition \ref{prop:Reg-Stoch-Conv} and $\gamma$ is as in \eqref{eq:condgamma1}. In the Gaussian case $\theta =1$ this bound simplifies to 
\begin{align}
\Ex \Big[   \big\| X - X_\eps   \big\|_{C( [0,T]; C^{ \alpha} )  }^p        \Big] \, &\leq \,  C \eps^{\gamma} .
\end{align} 
Using \eqref{eq:Hol-Bou-Psi3} and \eqref{eq:Hol-Bou-Psi7}  we obtain the same bounds for the $C( [0,T]; C^{ \alpha} )$ norm replaced by the $C^{\alpha/2}( [0,T]; C )$ norm. 

From \eqref{eq:bou-RT} and \eqref{eq:Stab-Reps} we get
\begin{align}
\Ex  \bigg[   \big\| R^{\theta(u)}  - & R^{\theta(u_\eps)}_\eps \big\|_{C \left( [0,\tau_{K_1,K_2,K_3,\eps}];\Omega C^{2\alpha}  \right)    }^p    \bigg] \,\notag\\
 & \leq \, C  \big( 1+K_1^p  \big)  \| \theta (u_\eps)\|_{p,\alpha}^{p} \eps^{\gamma}  +  C T^{\vartheta}  \big( 1+K_1^p  \big)  \| \theta(u)- \theta(u_\eps)  \|_{p,\alpha}^{p}  .     
\end{align}
Here we assume that $\vartheta$ is as defined in \eqref{eq:cond-thet2} and $\gamma$ as in \eqref{eq:cond-gamma2}.  Note that these assumptions on $\vartheta$ and $\gamma$ imply that the conditions needed above are automatically satisfied.

For the terms in the integral involving $g$ we get using Proposition \ref{prop:cont-dep-psi} that for $T$ small enough
\begin{align}
\Ex \Big[ \|  v - v^\eps   \|_{C([0,T\wedge \tau_{K_1,K_2,K_3,K_4,\eps}],C^1 )}^p \Big] \leq C \, \Ex \Big[   \Delta_{\Psi^{\theta(u)}, \Psi_\eps^{\theta(u_\eps)}  } \Big], 
\end{align}
and 
\begin{align}
\Ex   \Big[ \|  v - v^\eps   \|_{C^{1/2}([0,T],C^0 )}^p \Big]  \, \leq \,  C  \, \Ex \Big[   \Delta_{\Psi^{\theta(u)}, \Psi_\eps^{\theta(u_\eps)}  } \Big],
\end{align}
where
\begin{align}
 \Delta_{\Psi^{\theta(u)}, \Psi_\eps^{\theta(u_\eps)}}   \, = \,&    \|X-X_\eps \|_{C^\alpha_T} +  \|\XX-\XX_\eps \|_{\Omega C^{2\alpha}_T}    +    \| \Psi^{\theta(u)} -   \Psi_\eps^{\theta(u_\eps)}   \|_{C^\alpha_T} \notag\\
 &+   \| \theta(u)- \theta(u_\eps) \|_{C^\alpha_T}  +  \|    R^{\theta(u)}  -  R^{\theta(u_\eps)}_\eps    \|_{\Omega  C^{2\alpha}_{T \wedge \tau_{K_1,K_2,K_3,eps} }}     + |u_0 - u_\eps  |_\beta.   
 \end{align}
Thus summarising for $p$ big enough we get the bound 
\begin{align}
\Ex \Big[ | u - u_\eps \|_{C^{\alpha/2,\alpha}_{T \wedge \tau_\eps} }^p    \Big]\, \leq C \eps^{\gamma,p} +  C T^\vartheta\Ex \Big[ | u - u_\eps \|_{C^{\alpha/2,\alpha}_{T \wedge \tau_\eps} }^p    \Big]\ + \Ex \Big[|u- u_\eps |_\beta^p \Big], \label{eq:finubou}
\end{align}
where we use the notation  $\| u - u_\eps \|_{C^{\alpha/2,\alpha}_{T \wedge \tau_\eps} } =\| u - u_\eps \|_{C^{\alpha/2,\alpha}[0,T\wedge \tau_{K_1,K_2,K_3,\eps}]\times [0,1]}$. Here the constant depends on the $K_i$ but not on $\eps$. We have included the dependence on different initial data in the bound, although we start the two evolutions at the same initial data, because we will use this bound in an iteration over different time intervals.

For $T$ small enough the prefactor of the term involving $u - u_\eps $ on the right hand side of \eqref{eq:finubou} is less than $\frac{1}{2}$ and this term can be absorbed into the term on the left hand side. We then have for $T$ small enough  
\begin{align}
\Ex \Big[ | u - u_\eps \|_{C^{\alpha/2,\alpha}_{T \wedge \tau_\eps} }^p    \Big] \, \leq C \eps^{\gamma p} + C\Ex \Big[|u_0- u_{0,\eps} |_\beta^p \Big]. 
\end{align}
Since we start off with the same initial condition, we get convergence of the cutoff approximations before $T$. To get the same result for arbitrary times we iterate this procedure using $u(T)$ and $u_\eps(T)$ as new initial data. To derive bounds on their difference in $C^\beta$ (not $C^\alpha$!) we use the following trick. The solution $u(T)$ is a rough path controlled by $X(T)$ and we have 
\begin{equation}
\delta u(T,x,y) = \theta(u(t,x)) \delta X(t,x,y)  + R^{\theta(u)}(t,x,y),
\end{equation}
and similarly for $u_\eps$. Thus we have
\begin{equation}\label{eq:MT}
|u(T) - u_\eps(T)|_\beta \leq |\theta| |X(T)-X_\eps(T)  |_\beta + |\theta|_{C^1}|u(T)-u_\eps(T)|_0 K +  \big| R^{\theta(u)} - R_\eps^{\theta(u_\eps)} \big|_{\Omega C^\beta}. 
\end{equation}
We have already bounded all the quantities on the right hand side above so that we can iterate the argument. Note here that the bound on the Gaussian rough path holds for every exponent less than $\frac{1}{2}$ and thus in particular for $\beta$.  The iteration gives us 
\begin{align}
\Ex \Big[ \| u - u_\eps \|_{C^{\alpha/2,\alpha}_{T \wedge \tau_\eps} }^p    \Big]  \to 0  
\end{align}
for arbitrary $T$ and for any choice of the $K_i$ (but not uniformly in the $K_i$). Now to conclude we only need to remove the stopping times. To this end note that for $\delta < 1$ due to the definition of the stopping time $\tau_{K_1,K_2,K_3,K_4, \eps}$
\begin{align}
\Prob\Big[ \| u  -u_\eps \|_{C^{\alpha/2,\alpha},T}  \geq \delta  ]  \, \leq  \,  &\Prob\Big[ \| u-u_\eps \|_{C^{\alpha/2,\alpha},T\wedge \tau_{K_1,K_2,K_3,K_4, \eps}} \geq \delta   ]  \notag\\
&+ \Prob\Big[\sigma^{3,\eps}_{K_1,K_2,K_3,K_4} \leq T \Big]  + \Prob\Big[\sigma^{3}_{K_1,K_2,K_3,K_4} \leq T \Big] 
\end{align}
Now as in the proof of global existence by choosing the $K_i$ sufficiently large the second probabilities for the stopping times to be less than $T$ can be made arbitrarily small  . Then by choosing $\eps$ small enough the first probability can be made arbitrarily small as well. This finishes the argument.
\end{proof}

\bibliography{./Burger_Bib}

\begin{thebibliography}{{Hai}11b}
\expandafter\ifx\csname url\endcsname\relax
  \def\url#1{\texttt{#1}}\fi
\expandafter\ifx\csname urlprefix\endcsname\relax\def\urlprefix{URL }\fi

\bibitem[AR91]{AlbRock}
\textsc{S.~Albeverio} and \textsc{M.~R{\"o}ckner}.
\newblock Stochastic differential equations in infinite dimensions: solutions
  via {D}irichlet forms.
\newblock \emph{Probab. Theory Related Fields} \textbf{89}, no.~3, (1991),
  347--386.

\bibitem[BCJL94]{BCJ94}
\textsc{L.~Bertini}, \textsc{N.~Cancrini}, and \textsc{G.~Jona-Lasinio}.
\newblock The stochastic {B}urgers equation.
\newblock \emph{Comm. Math. Phys.} \textbf{165}, no.~2, (1994), 211--232.

\bibitem[BDP97]{BDP}
\textsc{F.~E. Benth}, \textsc{T.~Deck}, and \textsc{J.~Potthoff}.
\newblock A white noise approach to a class of non-linear stochastic heat
  equations.
\newblock \emph{J. Funct. Anal.} \textbf{146}, no.~2, (1997), 382--415.

\bibitem[CFO11]{CFO09}
\textsc{M.~Caruana}, \textsc{P.~K. Friz}, and \textsc{H.~Oberhauser}.
\newblock A (rough) pathwise approach to a class of non-linear stochastic
  partial differential equations.
\newblock \emph{Ann. Inst. H. Poincar\'e Anal. Non Lin\'eaire} \textbf{28},
  no.~1, (2011), 27--46.

\bibitem[Cha00]{MR1743612}
\textsc{T.~Chan}.
\newblock Scaling limits of {W}ick ordered {KPZ} equation.
\newblock \emph{Comm. Math. Phys.} \textbf{209}, no.~3, (2000), 671--690.

\bibitem[DPD02]{MR1941997}
\textsc{G.~Da~Prato} and \textsc{A.~Debussche}.
\newblock Two-dimensional {N}avier-{S}tokes equations driven by a space-time
  white noise.
\newblock \emph{J. Funct. Anal.} \textbf{196}, no.~1, (2002), 180--210.

\bibitem[DPD03]{DPD03}
\textsc{G.~Da~Prato} and \textsc{A.~Debussche}.
\newblock Strong solutions to the stochastic quantization equations.
\newblock \emph{Ann. Probab.} \textbf{31}, no.~4, (2003), 1900--1916.

\bibitem[DPDT94]{dPDT94}
\textsc{G.~Da~Prato}, \textsc{A.~Debussche}, and \textsc{R.~Temam}.
\newblock Stochastic {B}urgers' equation.
\newblock \emph{NoDEA Nonlinear Differential Equations Appl.} \textbf{1},
  no.~4, (1994), 389--402.

\bibitem[DPZ92]{dPZ}
\textsc{G.~Da~Prato} and \textsc{J.~Zabczyk}.
\newblock \emph{Stochastic equations in infinite dimensions}, vol.~44 of
  \emph{Encyclopedia of Mathematics and its Applications}.
\newblock Cambridge University Press, Cambridge, 1992.

\bibitem[FV10a]{FV07}
\textsc{P.~Friz} and \textsc{N.~Victoir}.
\newblock Differential equations driven by {G}aussian signals.
\newblock \emph{Ann. Inst. Henri Poincar\'e Probab. Stat.} \textbf{46}, no.~2,
  (2010), 369--413.

\bibitem[FV10b]{FV10}
\textsc{P.~K. Friz} and \textsc{N.~B. Victoir}.
\newblock \emph{Multidimensional stochastic processes as rough paths}, vol. 120
  of \emph{Cambridge Studies in Advanced Mathematics}.
\newblock Cambridge University Press, Cambridge, 2010.
\newblock Theory and applications.

\bibitem[GRR71]{GRR70}
\textsc{A.~M. Garsia}, \textsc{E.~Rodemich}, and \textsc{H.~Rumsey, Jr.}
\newblock A real variable lemma and the continuity of paths of some {G}aussian
  processes.
\newblock \emph{Indiana Univ. Math. J.} \textbf{20}, (1970/1971), 565--578.

\bibitem[GT10]{GT10}
\textsc{M.~Gubinelli} and \textsc{S.~Tindel}.
\newblock Rough evolution equations.
\newblock \emph{Ann. Probab.} \textbf{38}, no.~1, (2010), 1--75.

\bibitem[Gub04]{Gu04}
\textsc{M.~Gubinelli}.
\newblock Controlling rough paths.
\newblock \emph{J. Funct. Anal.} \textbf{216}, no.~1, (2004), 86--140.

\bibitem[Gy{\"o}98]{Gy98}
\textsc{I.~Gy{\"o}ngy}.
\newblock Existence and uniqueness results for semilinear stochastic partial
  differential equations.
\newblock \emph{Stochastic Process. Appl.} \textbf{73}, no.~2, (1998),
  271--299.

\bibitem[{Hai}11a]{Ha10}
\textsc{M.~{Hairer}}.
\newblock {Rough Stochastic PDEs}.
\newblock \emph{Comm. Pure Appl. Math.} \textbf{64}, (2011), 1547 -- 1585.

\bibitem[{Hai}11b]{Ha10-2}
\textsc{M.~{Hairer}}.
\newblock {Singular perturbations to semilinear stochastic heat equations}.
\newblock \emph{Probab. Theory Related Fields} (2011).
\newblock Online first.

\bibitem[Hai11c]{KPZPreprint}
\textsc{M.~Hairer}.
\newblock Solving the {KPZ} equation, 2011.
\newblock Preprint.

\bibitem[HM11]{HM10}
\textsc{M.~{Hairer}} and \textsc{J.~{Maas}}.
\newblock {A spatial version of the It\^o-Stratonovich correction}.
\newblock \emph{Ann. Prob.} (2011).
\newblock To appear.

\bibitem[HSV07]{HairerStuartVoss07}
\textsc{M.~Hairer}, \textsc{A.~M. Stuart}, and \textsc{J.~Voss}.
\newblock Analysis of {SPDE}s arising in path sampling. {II}. {T}he nonlinear
  case.
\newblock \emph{Ann. Appl. Probab.} \textbf{17}, no. 5-6, (2007), 1657--1706.

\bibitem[HV11]{HV10}
\textsc{M.~{Hairer}} and \textsc{J.~{Voss}}.
\newblock {Approximations to the Stochastic Burgers Equation}.
\newblock \emph{Jour. Nonl. Sci.} (2011).
\newblock Online first.

\bibitem[JLM85]{MR815192}
\textsc{G.~Jona-Lasinio} and \textsc{P.~K. Mitter}.
\newblock On the stochastic quantization of field theory.
\newblock \emph{Comm. Math. Phys.} \textbf{101}, no.~3, (1985), 409--436.

\bibitem[KPZ86]{KPZ}
\textsc{M.~Kardar}, \textsc{G.~Parisi}, and \textsc{Y.-C. Zhang}.
\newblock Dynamic scaling of growing interfaces.
\newblock \emph{Phys. Rev. Lett.} \textbf{56}, no.~9, (1986), 889--892.

\bibitem[KS91]{KS91}
\textsc{I.~Karatzas} and \textsc{S.~E. Shreve}.
\newblock \emph{Brownian motion and stochastic calculus}, vol. 113 of
  \emph{Graduate Texts in Mathematics}.
\newblock Springer-Verlag, New York, second ed., 1991.

\bibitem[LCL07]{LCL07}
\textsc{T.~J. Lyons}, \textsc{M.~Caruana}, and \textsc{T.~L{\'e}vy}.
\newblock \emph{Differential equations driven by rough paths}, vol. 1908 of
  \emph{Lecture Notes in Mathematics}.
\newblock Springer, Berlin, 2007.
\newblock Lectures from the 34th Summer School on Probability Theory held in
  Saint-Flour, July 6--24, 2004, With an introduction concerning the Summer
  School by Jean Picard.

\bibitem[LQ02]{LQ02}
\textsc{T.~Lyons} and \textsc{Z.~Qian}.
\newblock \emph{System control and rough paths}.
\newblock Oxford Mathematical Monographs. Oxford University Press, Oxford,
  2002.
\newblock Oxford Science Publications.

\bibitem[LS98]{LS98}
\textsc{P.-L. Lions} and \textsc{P.~E. Souganidis}.
\newblock Fully nonlinear stochastic partial differential equations.
\newblock \emph{C. R. Acad. Sci. Paris S\'er. I Math.} \textbf{326}, no.~9,
  (1998), 1085--1092.

\bibitem[Lyo98]{Ly98}
\textsc{T.~J. Lyons}.
\newblock Differential equations driven by rough signals.
\newblock \emph{Rev. Mat. Iberoamericana} \textbf{14}, no.~2, (1998), 215--310.

\bibitem[NR97]{NuRo}
\textsc{D.~Nualart} and \textsc{B.~Rozovskii}.
\newblock Weighted stochastic {S}obolev spaces and bilinear {SPDE}s driven by
  space-time white noise.
\newblock \emph{J. Funct. Anal.} \textbf{149}, no.~1, (1997), 200--225.

\bibitem[NR02]{NR02}
\textsc{D.~Nualart} and \textsc{A.~R{\u{a}}{\c{s}}canu}.
\newblock Differential equations driven by fractional {B}rownian motion.
\newblock \emph{Collect. Math.} \textbf{53}, no.~1, (2002), 55--81.

\bibitem[{Tei}11]{Te09}
\textsc{J.~{Teichmann}}.
\newblock {Another approach to some rough and stochastic partial differential
  equations}.
\newblock \emph{Stoch. Dyn.} \textbf{11}, no. 2-3, (2011), 535--550.

\end{thebibliography}


\begin{thebibliography}{HMW12}
\expandafter\ifx\csname url\endcsname\relax
  \def\url#1{\texttt{#1}}\fi
\expandafter\ifx\csname urlprefix\endcsname\relax\def\urlprefix{URL }\fi

\bibitem[HMW12]{HMW}
\textsc{M.~{Hairer}}, \textsc{J.~{Maas}}, and \textsc{H.~{Weber}}.
\newblock {Approximating rough stochastic PDEs}.
\newblock \emph{ArXiv e-prints} .

\bibitem[HW13]{HW10}
\textsc{M.~{Hairer}} and \textsc{H.~{Weber}}.
\newblock {Rough Burgers-like equations with multiplicative noise}.
\newblock \emph{Probab. Theory Related Fields} \textbf{155}, no. 1-2, (2013),
  71--126.

\end{thebibliography}
\bibliographystyle{./Martin}
\end{document}


\title[Erratum]{Erratum for ``ROUGH BURGERS-LIKE EQUATIONS WITH MULTIPLICATIVE NOISE"}

\author{Martin Hairer}
\address{Martin Hairer, The University of Warwick, Coventry CV4 7AL, UK} 
\email{M.Hairer@Warwick.ac.uk}

\author{Hendrik Weber}
\address{Hendrik Weber, The University of Warwick, Coventry CV4 7AL, UK} 
\email{hendrik.weber@warwick.ac.uk}

\date\today

\maketitle
\thispagestyle{empty}

Unfortunately, the proof of Proposition 4.8 in \cite{HW10} is not correct. 
In fact, equation (4.60), where $R^\theta$ is rewritten as a stochastic integral, 
has no meaning because the integrand is not adapted. 
Hence in (4.61) we are not allowed to apply the Burkholder-Davies-Gundy~inequality.

The argument was corrected in \cite[Lemma 3.6 and Corollary 3.7]{HMW} in a more complicated situation. In order to transfer these statement to the situation in \cite{HW10} some small changes are needed. 
First of all, the regularity of the linearised process $X$ should be measured in a H\"older norm with slightly bigger index $\aX$ than the solution $u$.  This can be done without further problems. With this change, 
the definition of the stopping time $\tau_K^X$ in (4.53) should be
replaced by the following.

\emph{For $K >0$ and for an $\aX \in (\alpha,1/2)$ we introduce the stopping time }
\begin{equation}\label{eq:def-tau}
\tau_{K}^{X} \,= \, \inf  \bigg\{ t \in  [0,T] \colon  \sup_{\substack{ x_1 \neq x_2 \\ 0 \leq s_1< s_2 \leq t}}  \frac{\big|X(s_1,x_1) -  X(s_2,x_2) \big|  }{|s_1-s_2|^{\aX/2}  + |x_1-x_2|^{\aX} }  >  K  \bigg\}\;.
\end{equation}

With this changed definition, Lemma 3.6 of \cite{HMW} implies, 
in the notation of \cite{HW10}, the following result.
\begin{lemma}\label{le:mistakecorrected}
Suppose that  $0<\alpha<\aX <\frac12 $ and let $\tau$ be a stopping time that almost surely satisfies 
\begin{equ}
 0 \leq \tau  \leq \tau_{K}^{X} \wedge T.
\end{equ}
For every $0 \leq t \leq T$ we set 
\begin{equs}
\tilde{\theta}(t) &:=\theta(t \wedge \tau),\\
\tilde{\Psi}^\theta(t) &:= \int_0^{t\wedge\tau}  S(t-r) \, \tilde{\theta}(r) \, dW(r),\\
\tilde{X}(t) &:= \int_{-\infty}^{t \wedge\tau}  S(t-r) \,  dW(r),\\
\tRte(t;x,y) &:= \delta \tilde{\Psi}^\theta(t;x,y) - \tilde{\theta}(t ,x) \, \delta \tilde{X} (t;x,y).
\end{equs} 
Then, for any $p$ large enough and for any $\gamma > 0$ such that
\begin{equ}
 \gamma <  \aX +\alpha -  \frac{1}{p}   - \sqrt{\frac{1}{2p} (1 + \alpha - \aX) } \;, 
\end{equ}
the following bound holds true:
\begin{equ}[e:Rt0]
\sup_{0 < t \leq T}  \, \E  \big| \tRte(t)  |_{\Omega\mathcal{C}^\gamma}^p   \,  \ls \, \$ \theta \$_{p,\alpha}^{p} .
\end{equ}
\end{lemma}
The statement given here is actually slightly stronger than the bound stated in \cite{HMW} because the norm appearing
on the left hand side of \eqref{e:Rt0} is bounded uniformly in $t$ instead of allowing a blow up near $0$. 
In \cite{HMW} we had to introduce this blowup due to a slightly modified definition of the Gaussian process $X$: 
the process used in \cite{HMW} does not start at $0$, but with stationary initial condition, which was convenient for other reasons. 
When going through the proof given in 
\cite{HMW}, one realises that when considering the process with zero initial condition, one can 
apply bound (3.74) for all times $t$ and there is no need to use (3.75) for small times. 

Based on this version of Lemma~\ref{le:mistakecorrected}, it is then straightforward to use the a 
priori information on the time regularity of $R^\theta$, combined with the fact that the ``tilde'' processes
coincide with the ``non-tilde'' processes before time $\tau$, to obtain the bound
\begin{equ}
\E  \bigg[   \big\| R^\theta \big\|_{C^{\kappa} \left( [0,\tau];\Omega C^{2\alpha}  \right)    }^p    \bigg] \,\leq \, C(K,T)  \$ \theta \$_{p,\alpha}^{p}\;,
\end{equ}
for sufficiently small values of $\kappa$ and sufficiently large values of $p$, as required.

\subsection*{Acknowledgements}

{\small
The authors are very grateful to the anonymous referee of \cite{HMW} for
pointing out this mistake.
}

\bibliographystyle{Martin}
\bibliography{burgers}